%% file: main.tex
\documentclass[11pt]{article}
\usepackage{fullpage}
\usepackage[utf8]{inputenc} 
\usepackage[T1]{fontenc}    
\usepackage{hyperref}       
\usepackage{url}            
\usepackage{booktabs}       
\usepackage{amsfonts}       
\usepackage{nicefrac}       
\usepackage{microtype}      
\usepackage{tcolorbox}
\usepackage{diagbox}
\usepackage{enumerate}
\usepackage[shortlabels]{enumitem}
\usepackage{algorithmic}
\usepackage{algorithm}
\usepackage{subfigure}
\usepackage{graphicx} 
\usepackage{caption}
\usepackage{amsmath}
\usepackage{amsthm}
\usepackage{amssymb}
\usepackage{tikz}
\usepackage{tablefootnote}
\usepackage{multirow}
\usepackage{enumerate}
\usepackage{color}
\usepackage{xcolor}
\usepackage[square,sort,comma,numbers]{natbib}

\allowdisplaybreaks[4]
\usepackage{mathrsfs}

\usepackage{hyperref}
\usepackage{bm,todonotes}

\allowdisplaybreaks

\theoremstyle{plain}
\newtheorem{thm}{Theorem}[section]
\newtheorem{lem}{Lemma}[section]
\newtheorem{cor}{Corollary}[section]
\newtheorem{prop}{Proposition}[section]

\theoremstyle{definition}
\newtheorem{assump}{Assumption}[section]
\newenvironment{assumpt}[1]{
	\renewcommand\theassump{#1}
	\assump
}{\endassump}
\newtheorem{defn}{Definition}[section]
\newtheorem{fact}{Fact}[section]

\newtheorem{rema}{Remark}[section]
\newtheorem{exam}{Example}[section]

\newcommand{\inner}[2]{\langle #1, #2 \rangle}

\def\LGx{{L}_{G,x}}
\def\LGy{{L}_{G,y}}
\def\Hsmooth{{\delta_H}}

\def\Fsmooth{{\delta_F}}
\def\Gsmooth{{\delta_G}}

\def\nablaH{\nabla H}
\def\Hholder{{\widetilde{S}_H}}
\def\SBF{{S}_{B,F}}
\def\SBG{{S}_{B,G}}

\def\invdifffast{\tilde{\alpha}}
\def\invdiffslow{\tilde{\beta}}
\def\cti{C^{ti}}

\def\hyz#1 {\textcolor{red}{Hyz: #1 }}
\def\lx#1 {\textcolor{magenta}{Lx: #1 }}
\def\rev#1 {\textcolor{blue}{#1}}

\input{tex/math_commands}

\title{Decoupled Functional Central Limit Theorems for Two-Time-Scale Stochastic Approximation}

\author{
Yuze Han 
\thanks{Center for Applied Statistics and School of Statistics, Renmin University of China; email: \texttt{hanyuze97@ruc.edu.cn}.}
\and
Xiang Li 
\thanks{School of Mathematical Sciences, Peking University; email: \texttt{lx10077@pku.edu.cn}. } 
\and
Jiadong Liang 
\thanks{School of Mathematical Sciences, Peking University; email: \texttt{jdliang@pku.edu.cn}. }
\and
Zhihua Zhang 
\thanks{School of Mathematical Sciences, Peking University; email: \texttt{zhzhang@math.pku.edu.cn}. }
}

\begin{document}

\maketitle

\begin{abstract}%
In two-time-scale stochastic approximation (SA), two iterates are updated at different rates, governed by distinct step sizes, with each update influencing the other. Previous studies have demonstrated that the convergence rates of the error terms for these updates depend solely on their respective step sizes, a property known as decoupled convergence. However, a functional version of this decoupled convergence has not been explored. Our work fills this gap by establishing decoupled functional central limit theorems for two-time-scale SA, offering a more precise characterization of its asymptotic behavior. Our results show that, on each time scale, the limiting dynamics has the same form as in standard SA, and the coupling between the two iterates enters the limit only through the associated coefficients. To achieve these results, we leverage the martingale problem approach and establish tightness as a crucial intermediate step. Furthermore, to address the interdependence between different time scales, we introduce an innovative auxiliary sequence to eliminate the primary influence of the fast-time-scale update on the slow-time-scale update.

\end{abstract}

\section{Introduction}

\input{tex/intro}

\input{tex/fclt}

\begin{appendix}
\input{tex/append}
\end{appendix}

\bibliography{twotime}
\bibliographystyle{plainnat}

\end{document}

%% file: tex/math_commands.tex
\usepackage{amsmath,amsfonts,bm}

\def\Figref#1{Figure~\ref{#1}}

\def\1{\bm{1}}

\def\eps{{\varepsilon}}

\def\ermU{{\textnormal{U}}}

\def\ermW{{\textnormal{W}}}
\def\ermX{{\textnormal{X}}}
\def\ermY{{\textnormal{Y}}}

\def\mI{{\bm{I}}}

\DeclareMathAlphabet{\mathsfit}{\encodingdefault}{\sfdefault}{m}{sl}
\SetMathAlphabet{\mathsfit}{bold}{\encodingdefault}{\sfdefault}{bx}{n}

\def\d{\mathrm{d}}

\def\0{{\bf 0}}
\def\1{{\bf 1}}

\def\MM{{\mathcal M}}
\def\AM{{\mathcal A}}

\def\FM{{\mathcal F}}

\def\NM{{\mathcal N}}
\def\OM{{\mathcal O}}
\def\PM{{\mathcal P}}

\def\RM{{\mathcal R}}
\def\SM{{\mathcal S}}
\def\TM{{\mathcal T}}

\def\RB{{\mathbb R}}

\def\EB{{\mathbb E}}

\def\NB{{\mathbb N}}
\def\PB{{\mathbb P}}

\def\tr{\mathrm{tr}}

\newcommand{\xhat}{{\hat{x}}}
\newcommand{\yhat}{{\hat{y}}}

\newcommand{\xcheck}{{\check{x}}}
\newcommand{\ycheck}{{\check{y}}}
\newcommand{\zcheck}{{\check{z}}}

\newcommand{\ucheck}{{\check{u}}}

\newcommand{\xbar}{{\bar{\textnormal{X}}}}
\newcommand{\ybar}{{\bar{\textnormal{Y}}}}
\newcommand{\zbar}{{\bar{\textnormal{Z}}}}
\newcommand{\xx}{{x}}
\newcommand{\yy}{{y}}
\newcommand{\uu}{{\textnormal{U}}}
\def\mI{{I}}

\def\oprob{{{o}_P}}

\newcommand{\ssum}[3]{\sum\limits_{{#1}={#2}}^{#3}}

\def\Xii{\mbox{\boldmath$\Xi$\unboldmath}}

\def\Psii{\mbox{\boldmath$\Psi$\unboldmath}}
\def\bfB{{\bf B}}

\def\weakconverge{\Rightarrow}

\def\Ga{G_{\mathrm{GTD2}}}
\def\Gb{G_{\mathrm{TDC}}}
\def\psia{\psi^{\mathrm{GTD2}}}
\def\psib{\psi^{\mathrm{TDC}}}

\newcommand{\norm}[1]{\left\|{#1}\right\|}
\def\inner#1#2{\left\langle #1, #2 \right\rangle}

%% file: tex/intro.tex
\label{sec:intro}

Stochastic approximation (SA) originally introduced in \cite{robbins1951stochastic} is an iterative method for finding the root of an unknown operator from noisy observations. It has been widely adopted in stochastic optimization and reinforcement learning~\citep{kushner2003stochastic, borkar2009stochastic, mou2020linear, mou2022banach, mou2022optimal, li2021polyak, li2023online}.
Although traditional SA addresses single-time-scale updates, many problems require managing two coupled updates with distinct step sizes, as exemplified by stochastic bilevel optimization~\citep{ghadimi2018approximation, chen2021closing, hong2023two}, temporal difference learning~\citep{sutton2009fast, xu2019two, xu2021sample, wang2021non}, and actor-critic methods~\citep{borkar1997actor, konda2003onactor, wu2020finite, xu2020non}. These scenarios necessitate a two-time-scale framework to handle the interaction between updates operating at different speeds.

In this paper, we study the two-time-scale SA~\citep{borkar1997stochastic}, a variant of classical SA designed to solve systems of two coupled equations. We consider two unknown Lipschitz operators,
$F:\RB^{d_x}\times\RB^{d_y}\rightarrow\RB^{d_x}$ and $G:\RB^{d_x}\times\RB^{d_y}\rightarrow\RB^{d_y}$, and aim to find a root pair $(x^{\star}, y^{\star})$ that satisfies
\begin{equation}
	\left\{
	\begin{aligned}
		&F(x^{\star},y^{\star}) = 0,\\
		&G(x^{\star},y^{\star}) = 0.
	\end{aligned}\right.\label{prob:FG}
\end{equation}
Because $F$ and $G$ are unknown, we assume access to a stochastic oracle that provides noisy evaluations $F(x,y)+ \xi$ and $G(x,y) + \psi$ for any $(x,y)$, where $\xi$ and $\psi$ are noise components.
With this oracle, we iteratively update estimates $x_{n}$ and $y_{n}$ using
\begin{subequations} \label{alg:xy}
	\begin{align}
		x_{n+1} &= x_{n} - \alpha_{n}\left(F(x_{n},y_{n}) + \xi_{n}\right), \label{alg:x} \\   
		y_{n+1} &= y_{n} - \beta_{n}\left(G(x_{n},y_{n}) + \psi_{n}\right), \label{alg:y}
	\end{align}
\end{subequations}
where $x_{0}$ and $y_{0}$ are  arbitrarily initialized, $\xi_n$ and $\psi_n$ are noise terms, and $\alpha_{n}$, $\beta_{n}$ are step sizes satisfying $\beta_{n}\ll \alpha_{n}$,
in $\RB^{d_x}$ and $\RB^{d_y}$.
This disparity in step sizes ensures that $y_n$ is ``quasi-static'' relative to $x_n$~\citep{konda2004convergence}.
Accordingly, $x_n$ and $y_n$ are termed the \textit{fast iterate} and \textit{slow iterate}, with their updates referred to as fast-time-scale and slow-time-scale updates, respectively.

For example, in stochastic bilevel optimization, a lower-level problem is nested within an upper-level optimization. 
Two-time-scale SA updates the lower-level variable $x$ with a larger step size, so that $x$ rapidly tracks an approximate solution of the inner problem while the outer variable $y$ remains quasi-static~\citep{hong2023two}.
In gradient temporal-difference (TD) learning with linear function approximation, one seeks a parameter $y$ that minimizes the mean-squared projected Bellman error. 
The resulting gradient typically involves the inverse of the feature-covariance matrix. 
To avoid explicit matrix inversion, an auxiliary variable $x$ is introduced and updated on a faster time scale to (approximately) fit the TD error using features while treating $y$ as quasi-static \citep{sutton2009fast}. 
In actor–critic methods, the critic parameter $x$ is updated on the faster time scale to approximate the value function of the current (quasi-static) policy parameterized by $y$.
The actor then uses the critic’s evaluation to update $y$ on the slower time scale to improve the policy \citep{konda2003onactor}.

The above examples and previous studies demonstrate that if the solution to $F(x,y) = 0$ is unique for each fixed $y$, two-time-scale SA can be interpreted as a single-loop approximation of a double-loop process~\citep{kaledin2020finite, doan2022nonlinear, han2024finite}
\begin{equation} \label{eq:two-loop-approx}
	\begin{aligned}    
		& \text{Inner loop: compute } x = H(y)
		, \text{ or equivalently, solve } F(x, y) = 0
		\text{ for a fixed } y; \\
		& \text{Outer loop: iteratively find the root of } G(H(y), y) = 0.
	\end{aligned}
\end{equation}
Here $H(y)$ is the unique solution to $F(x,y) = 0$ for a given $y$, satisfying  $x^\star = H(y^\star)$.
From this viewpoint, the error terms for the outer and inner loops are represented as $y_n - y^\star$ and $x_n - H(y_n)$, respectively.
Previous studies have shown that the convergence rates of these error terms depend sorely on their respective step sizes \citep{konda2004convergence, mokkadem2006convergence, kaledin2020finite, han2024finite}, 
a property known as \textit{decoupled convergence}~\citep{han2024finite}.
In particular, under certain regular conditions, it holds that $ \beta_{n}^{-1/2}  (y_{n} - y^\star)$ and $ \alpha_{n}^{-1/2} (x_{n} - H(y_{n}))$ both convergence weakly to Gaussian distributions, with non-asymptotic convergence rates $\EB \| y_n - y^\star \|^2 = \OM(\beta_n) $ and $\EB \| x_n - H(y_n) \|^2 = \OM(\alpha_n)$ under martingale difference noise.

However, existing results primarily focus on single-point convergence rates, overlooking the asymptotic behavior of trajectories. Analyzing trajectory behavior is a classical topic in SA research~\citep{kushner2003stochastic, borkar2009stochastic, benveniste2012adaptive}, with the functional central limit theorem (FCLT) serving as a foundational result. The FCLT extends Donsker's theorem to iterative methods by examining continuous-time trajectories of rescaled iterates, offering a deeper characterization of the method's asymptotic properties. 
Specifically, the weak limit of these trajectories is a multi-dimensional Ornstein-Uhlenbeck process.
Compared with single-point convergence results, FCLTs are substantially more informative because they characterize the limiting behavior of the trajectory. 
For instance, integrating over a finite interval yields the asymptotic distribution of averaged iterates \citep{negrea2022statistical}, while taking a supremum functional yields uniform control of the iterates over a prescribed range (see Remark~\ref{rema:simul-converge}).

For two-time-scale SA, the only existing FCLT analysis~\citep{faizal2023functional} rescales both iterates with the larger step size, leading to a deterministic ordinary differential equation (ODE) as the limiting behavior for the slow iterate, which fails to capture stochastic effects.
Motivated by the decoupled convergence property, our goal is to 
\begin{center}
	Establish \textit{decoupled FCLTs} for two-time-scale SA. 
\end{center}
Decoupling involves rescaling the error terms $y_{n} - y^\star$ and $x_{n} - H(y_{n})$ by the square root of their respective step sizes. Specifically, we construct continuous-time trajectories for the rescaled error terms $ \beta_{n}^{-1/2} (y_{n} - y^\star)$ and $ \alpha_{n}^{-1/2} (x_{n} - H(y_{n}))$ and analyze their weak convergence limits. 

We focus on the regime $\beta_n/\alpha_n \to 0$, which can be referred to as the \emph{strict} two-time-scale setting. 
In contrast to the single-time-scale regime $\beta_n/\alpha_n \to \kappa>0$, strict time-scale separation can substantially reduce sensitivity to the initial step-size ratio: the algorithm can converge for a much broader range of $\beta_0/\alpha_0$, thereby easing step-size tuning in practice~\citep{haque2023tight}. 
From a theoretical standpoint, FCLTs in the single-time-scale regime can largely be obtained by adapting standard arguments for SA~\citep{borkar2009stochastic}. By comparison, FCLTs for the error sequences  $ \{\beta_{n}^{-1/2} (y_{n} - y^\star)\}$ and $\{\alpha_{n}^{-1/2} (x_{n} - H(y_{n}))\}$, which are rescaled on distinct time scales, are far less developed. 
Such results are particularly valuable because they can reveal how the coupling between the two iterates affects the limiting dynamics.

We would emphasize that the goal of our paper is not to optimize step-size schedules; rather, it characterizes the asymptotic behavior via FCLTs under appropriately chosen strict two-time-scale step sizes.
Practical guidance on step-size selection across different scenarios can be found in \cite{han2024finite, chandak2025non, hong2023two} and the references therein.

\subsection{Contributions} Our contributions can be summarized as follows.

\begin{itemize}
    \item \textbf{Theoretical contribution.} We establish decoupled FCLTs for two-time-scale SA under certain regularity conditions.
    Specifically, we construct two sequences of continuous-time processes, one on the fast time scale and one on the slow time scale, and show that each converges weakly to its own stationary multi-dimensional Ornstein–Uhlenbeck limit (see Theorem~\ref{thm:fclt-twotime}).
    In particular, the limiting process on the fast time scale coincides with that of standard SA. The limiting process on the slow time scale can be interpreted as the FCLT limit for SA applied to the operator $G(H(\cdot), \cdot)$, with a modified noise term that depends on the original noise on both time scales.
    As a simple application, our FCLTs yield high-probability simultaneous convergence of the iterates over a prescribed range (see Remark~\ref{rema:simul-converge}). 
	
    \item \textbf{Technical contribution.} We introduce a novel construction of continuous-time trajectories for two-time-scale SA, rescaling the error terms for each time scale by the square root of their respective step sizes (see Section~\ref{sec:fclt:construct}). Despite the increased complexity introduced by this rescaling, we establish the tightness of the trajectory sequences (see Lemma~\ref{lem:tight-twotime}) as a crucial intermediate result. 
    Building on the classical martingale problem approach for Markov processes (see Section~\ref{sec:fclt:prelim:mart}), we extend it to the two-time-scale SA setting and use it to establish the desired weak convergence results.
    To address the interdependence between different time scales, we propose a novel auxiliary sequence that mitigates the primary influence of the fast-time-scale update on the slow-time-scale update (see Section~\ref{sec:fclt:one-step}), providing a powerful new analytical tool for studying two-time-scale algorithms.

\end{itemize}

\subsection{Related Work}

Our research explores the asymptotic behavior of trajectories in two-time-scale SA. To contextualize our results, we provide additional background on the decoupled convergence for two-time-scale SA and FCLTs for various SA-based algorithms.
Importantly, unlike studies on diffusions or systems with two time scales \citep{Khasminskij1968OnTP, kushner1988almost, kokotovic1984applications, khasminskii2005limit}, our investigation begins with the iterative algorithm in \eqref{alg:xy}. Our primary goal is to analyze the asymptotic properties of its trajectories through the lens of limiting continuous-time processes.
\vspace{0.1cm}

\textbf{Decoupled convergence for two-time-scale SA.}
Two-time-scale SA was originally introduced by \cite{borkar1997stochastic}, where the almost sure convergence was established. This result was further explored and extended in \cite{tadic2004almost, karmakar2018two, yaji2020stochastic}.
Regarding convergence rates, both asymptotic and non-asymptotic decoupled convergence rates have been rigorously analyzed and established for two-time-scale SA.

For linear $F$ and $G$, \citet{konda2004convergence} showed that $\beta_n^{-1/2} (y_n - y^\star)$ weakly converges to a normal distribution under martingale difference noise. Their analysis also implies that $\alpha_n^{-1/2} (x_n - x^\star)$ and $\alpha_n^{-1/2} (x_n - H(y_n))$ weakly converge to normal distributions. Inspired by this, \citet{kaledin2020finite} derived finite-time convergence rates, proving that $\EB \| y_n - y^\star \|^2 = \OM(\beta_n)$ and $\EB \| x_n - H(y_n) \|^2 = \OM(\alpha_n)$ for both the martingale noise and Markovian noise. \citet{haque2023tight} achieved the same rates with asymptotically optimal leading terms, while \citet{kwon2024two} provided a more refined characterization of the bias and variance terms for constant step-size schemes. Fora sparse projection variant, \citet{dalal2020tale} showed that $\| y_n - y^\star \| = \widetilde{\OM} (\sqrt{\beta_n})$ and $\| x_n - x^\star \| = \widetilde{\OM} (\sqrt{\alpha_n})$ with high probability.

For the general nonlinear case, \citet{mokkadem2006convergence} established a joint weak convergence limit for $\alpha_n^{-1/2} (x_n - x^\star)$ and $\beta_n^{-1/2} (y_n - y^\star)$ under local linearity. This result has been extended to the Markovian noise by~\citet{hu2024central} and continuous-time dynamics by \citet{sharrock2022two}. \citet{han2024finite} further established the finite-time decoupled convergence for nonlinear two-time-scale SA under a slightly different local linearity condition.
\vspace{0.1cm}

\textbf{Asymptotic trajectory analysis for SA-based algorithms.}
Characterizing the asymptotic behavior of iterate trajectories is a central focus in SA analysis. The functional central limit theorem~(FCLT) leverages diffusion processes to approximate the method's asymptotic behavior~\citep{kushner2003stochastic, borkar2009stochastic, benveniste2012adaptive, borkar2021ode}. These weak convergence results have been extended to various SA variants, including generalized regularized dual averaging algorithms (gRDA) \citep{chao2019generalization}, accelerated stochastic gradient descent \citep{wang2020asymptotic}, stochastic gradient Langevin dynamics (SGLD) \citep{negrea2022statistical}, and loopless projection stochastic approximation (LPSA) \citep{liang2023asymptotic}.
Furthermore, \citet{blanchet2024limit} established an FCLT for SA with heavy-tailed noise under constant step sizes or in the linear case.

Recently, FCLTs for rescaled partial-sum processes have been explored as a foundation for efficient online inference in fields such as stochastic optimization~\citep{lee2021fast}, federated learning~\citep{li2022statistical}, reinforcement learning~\citep{li2013statistical, xie2022statistical}, gradient-free optimization~\citep{chen2021online}, and non-smooth regression~\citep{lee2022fast}. Additionally, \citet{xie2024asymptotic} established an almost surely Gaussian approximation for the averaged iterate to perform time-uniform inference.
\vspace{0.1cm}

\textbf{FCLTs for two-time-scale SA.}
The only existing FCLT analysis for the general case is provided by \cite{faizal2023functional}. 
However, their analysis rescales both time scales by the same factor.\footnote{After communicating with the authors,
they recognized that an error exists in their derivation of the limiting behavior for Equation (30) on page 20 in \cite{faizal2023functional}.} 
It is also important to highlight the difference in motivation between their work and ours. Their goal is to characterize the fluctuations of the iterates around their limiting ODEs. In contrast, our objective is to establish decoupled FCLTs for rescaled error terms to describe the asymptotic behavior around the solution.

For specific examples, \citet{kushner1993stochastic} provided a heuristic discussion of SA with whole-trajectory averaging. However, their rigorous theoretical results apply to a different averaging scheme that does not fit in the framework of \eqref{alg:xy}. \citet{gadat2018stochastic} focused on the stochastic heavy ball (SHB) method, but their analysis is restricted to the single-time-scale setting where $\beta_n / \alpha_n$ remains a positive constant.

Thus, our work is the first to derive decoupled FCLTs for general two-time-scale SA by rescaling the iterates according to the square root of their respective step sizes.

\subsection{Notation and Organization}

For a vector $x$, $\| x \|$ denotes the Euclidean norm; 
for a matrix $A$, $\| A \|$ denotes the spectral norm and $\| A \|_F$ denotes the Frobenius norm.
We use $o(\cdot)$, $\OM(\cdot)$, $\Omega(\cdot)$, and $\Theta(\cdot)$ to hide universal constants and $\widetilde{\OM}(\cdot)$ to hide both universal constants and log factors.
The notation $o_p(1)$ and $\OM_p(1)$ indicate convergence in probability and stochastic boundedness, respectively.
For two non-negative numbers $a$ and $b$, $a \lesssim b$ indicates the existence of a positive number $C$ such that $a \le Cb$ with $C$ depending on parameters of no interest.
For two positive sequences $\{ a_n \}$ and $\{ b_n \}$, $a_n \sim b_n$ signifies $\lim_{n \to \infty} a_n / b_n = 1$.
For a random variable $\ermX$, $\ermX \sim \pi$ means $\ermX$ follows the distribution $\pi$.
We use $\weakconverge$ to denote weak convergence or convergence in distribution and $\overset{p}{\rightarrow}$ to signify convergence in probability.
For the step size sequences $\{ \alpha_{n} \}_{n=0}^\infty$ and $\{ \beta_{n} \}_{n=0}^\infty$, their ratio is defined as
$\kappa_{n} := \frac{\beta_{n}}{\alpha_{n}}$.
For $m < n$ we stipulate $\sum_{i=n}^{m} \alpha_{i} = \sum_{i=n}^{m} \beta_{i} = 0 $.
The identity matrix is denoted by $\mI$, with an appropriate size.
The standard $d$-dimensional Brownian motion is represented as $\ermW^d(t)$.
For a differentiable operator $T\colon \RB^{d_1} \to \RB^{d_2}$, we use $\nabla T \in \RB^{d_2 \times d_1}$ to denote its Jacobian matrix.

Given a set $D \subseteq \RB^{m}$ , $C(D; \RB^d)$ denotes the space of all $\RB^d$ -valued continuous functions defined on the domain $D$. When $d=1$, we use $C(D)$ in shorthand.
Furthermore, we use 
$C_c(D)$ to denote the subspace of $C(D)$ consisting of all continuous  functions with compact support, $C^m_c(D)$ for the subspace of $C(D)$ containing all functions with compact support and continuous $m$-th order derivatives, and $C^\infty_c(D)$ for the subspace of $C(D)$ consisting of all compactly supported smooth functions (i.e., infinitely differentiable).

The remainder of this paper is organized as follows. Section~\ref{sec:fclt:prelim} presents the preliminaries, followed by the necessary assumptions in Section~\ref{sec:fclt:assump}. Our main results are provided in Section~\ref{sec:fclt:main}, with the proof framework outlined in Section~\ref{sec:fclt:proof-frame}. Finally, Section~\ref{sec:fclt:conclude} concludes the main text, and the omitted details and proofs are provided in the appendices.

%% file: tex/fclt.tex
\section{Preliminaries}\label{sec:fclt:prelim}
In this section, we introduce essential definitions and properties concerning tightness, weak convergence and the martingale problem approach.
For more details, please refer to \cite{billingsley2013convergence, ethier2009markov}.

\subsection{Tightness and Weak Convergence}
\label{sec:fclt:prelim:tight}

Throughout this paper, we primarily focus on the space $C([0, \infty); \RB^d)$.
Following the approach in
\cite{whitt1970weak}, we equip this space with the following metric: for any $f,g \in C([0, \infty); \RB^d)$,
\begin{equation}\label{eq:metric}
	\rho (f, g) := \ssum{j}{1}{\infty} 2^{-j} \frac{\rho_j(f,g)}{1+\rho_j(f,g)} \ \text{with} \ 
	\rho_{j} (f, g) := \max_{0 \le t \le j} \| f(t) - g(t) \|.
\end{equation}
This metric induces the topology of uniform convergence on compact sets, i.e., 
$\rho(f,g) \to 0$ if and only if $\rho_j (f,g) \to 0$ for each positive integer $j$.
Furthermore, $(C([0, \infty); \RB^d), \rho) $ is a complete separable metric space~\citep[Theorem~1]{whitt1970weak}.
A random element in this space can be interpreted as a stochastic process indexed by
$t \in  [0, \infty)$, with the state space $\RB^d$ and continuous sample paths. 
For brevity, we refer to such random elements as random functions in $C([0, \infty); \RB^d)$.

The topology induced by $\rho$ allows us to define compact sets in $C([0, \infty); \RB^d)$.
Moreover, each random function $U(\cdot)$ in $C([0, \infty); \RB^d)$ induces a probability measure on this space, namely its law (the pushforward of the underlying probability measure on the probability space $(\Omega, \FM, \PB)$; for simplicity, we omit the explicit dependence of $U(\cdot)$ on $\omega \in \Omega$). 
Using the standard definition of tightness for probability measures, we extend the concept of tightness to random functions as follows.

\begin{defn}[Tightness of random functions in $C( [0, \infty); \RB^d)$ ]\label{defn:tight}
Let $\{ \uu_{n}(\cdot) \}_{n=0}^\infty$ be a sequence of random functions in $C( [0, \infty); \RB^d)$.
The sequence $\{ \uu_{n}(\cdot) \}_{n=0}^\infty$ is \textit{tight}, if for any $\eps > 0$, there exists a compact set $K \subset C( [0, \infty); \RB^d)$ 
such that $\PB( \uu_{n}(\cdot) \in K ) \ge 1 - \eps$ for any $n$. 
\end{defn}

Tightness is a key property for establishing the weak convergence of random functions. The following proposition shows that tightness ensures weak convergence in $C([0,\infty); \RB^{d})$ can be characterized through finite-dimensional distributions or convergent subsequences.

\begin{prop}[Properties of tightness]\label{prop:tight}
	Suppose that $\{ \uu_{n}(\cdot) \}_{n=0}^\infty$ is a tight sequence of random functions in $C([0, \infty); \RB^{d})$ 
	and $\uu(\cdot)$ is another random function in $C([0, \infty); \RB^{d})$. 
	\begin{enumerate}[(i)]
		\item \label{prop:tight:fdd}
		If the finite-dimensional distributions of $ \uu_{n}(\cdot) $ converge weakly to the finite-dimensional distributions of $\uu(\cdot)$, i.e., for each positive integer $k$ and $t_1, t_2, \dots, t_k > 0$, it holds that $(\uu_{n}(t_1), \uu_{n}(t_2), \dots, \uu_{n}(t_k)) \weakconverge (\uu(t_1), \uu(t_2), \dots, \uu(t_k))$, then $\uu_{n}(\cdot) \weakconverge \uu(\cdot)$ {\citep[Theorem~3]{whitt1970weak}}.
		
		\item \label{prop:tight:prok} Prokhorov's theorem {\citep[Theorem~5.1]{billingsley2013convergence}}: Each subsequence of $\{ \uu_{n}(\cdot) \}_{n=0}^\infty$ contains a further subsequence that converges weakly to some random function in $C([0, \infty); \RB^{d})$.
		\item \label{prop:tight:converge} If each convergent subsequence of $\{ \uu_{n}(\cdot) \}_{n=0}^\infty$ converges weakly to $\uu(\cdot)$, then $\uu_{n}(\cdot) \weakconverge \uu(\cdot)$ \citep[Corollary on page~59]{billingsley2013convergence}.
	\end{enumerate}
\end{prop}

\subsection{The Martingale Problem Approach}
\label{sec:fclt:prelim:mart}

In the 1960s, Stroock and Varadhan introduced the martingale problem as a framework for characterizing Markov processes~\citep{stroock1997multidimensional}. This approach also serves as a powerful tool for establishing weak convergence to a Markov process~\citep{ethier2009markov} and has been applied in \cite{benveniste2012adaptive, gadat2018stochastic, liang2023asymptotic} to derive FCLTs for various SA-based algorithms. In this paper, we also leverage this method to prove our main results.

Before introducing the formal definition of the martingale problem, we first specify the limiting process that arises in our setting.
In our context, the limiting process is a multi-dimensional Ornstein-Uhlenbeck process, defined as the solution to the following stochastic differential equation (SDE)
\begin{equation}\label{eq:sde-general}
	d \uu(t) = - B \uu(t) dt + \Sigma^{1/2} d \ermW^d(t),
\end{equation}
where $B, \Sigma \in \RB^{d \times d}$ and $\Sigma$ is positive definite.
Given an initial distribution for $\uu(0)$,
the solution to~\eqref{eq:sde-general} exists uniquely. Furthermore, this solution is a time-homogeneous Markov process~\citep[Chapter 5, Section 3]{ethier2009markov}.
To characterize this process, its \textit{infinitesimal generator} is defined as 
\begin{equation}\label{eq:generator-general}
    \AM f(x) = - \inner{Bx}{\nabla f(x)} + \frac{1}{2} \tr ( \nabla^2 f(x) \Sigma ), \ \forall f \in C^\infty_c (\RB^d) .
\end{equation}
Here we restrict attention to the space of compactly supported smooth functions $C^\infty_c(\RB^d)$.
For the SDE in \eqref{eq:sde-general}, this choice is sufficient because the generator $\AM$ restricted to $C^\infty_c(\RB^d)$
uniquely characterizes the associated diffusion process; see \cite[Chapter~8, Theorem~1.6]{ethier2009markov}.

Given the generator in \eqref{eq:generator-general}, the corresponding martingale problem is formulated as follows \cite{stroock1997multidimensional, ethier2009markov}.

\begin{defn}[The martingale problem]\label{def:mart_prob}
	Suppose $\AM$ is defined in \eqref{eq:generator-general}
    and $\uu(\cdot)$ is a $(\FM_t)_{t \ge 0}$-adapted stochastic process. 
	We say $\uu(\cdot)$ solves the \textit{martingale problem} for $(\AM, C^\infty_c(\RB^d))$, if for any $f \in C^\infty_c(\RB^d)$, the following process
	\begin{equation}\label{eq:mart-prob}
		f(\uu(t)) - f(\uu(0)) - \int_0^t \AM f(\uu(s)) ds
	\end{equation}
	 is a martingale w.r.t.\ to filtration $(\FM_t)_{t \ge 0}$.
\end{defn}

The following proposition 
establishes the equivalence between the martingale problem for $ (\AM, C^\infty_c (\RB^d) )$ and the solution to  SDE \eqref{eq:sde-general}.
By this equivalence, the existence and uniqueness of the solution to \eqref{eq:sde-general} imply that the martingale problem for $(\AM, C^\infty_c (\RB^d) )$ also has a unique solution, given the initial distribution of $\uu(0)$.

\begin{prop}[{\cite[Chapter~5, Proposition~3.1 and Theorem~3.3]{ethier2009markov}}]
\label{prop:equiv-sde-martprob}
Suppose that the initial distribution $\uu(0)$ is given and the generator $\AM$ is defined in \eqref{eq:generator-general}.
    \begin{enumerate}[(i)]
        \item If $\uu(\cdot)$ is a solution to the SDE in \eqref{eq:sde-general}, then $\uu(\cdot)$ solves the martingale problem for $(\AM, C^\infty_c (\RB^d) )$.
        \item If $\uu(\cdot)$ solves the martingale problem for $(\AM, C^\infty_c (\RB^d) )$, then $\uu(\cdot)$ is a solution to the SDE in~\eqref{eq:sde-general}.
    \end{enumerate}
\end{prop}

\section{Assumptions}\label{sec:fclt:assump}

In this section, we present the main assumptions.

\begin{assump}[Lipschitz conditions]
	\label{assump:smooth}
	There exists an operator $H\colon \RB^{d_y}\rightarrow\RB^{d_x}$ such that for any fixed $y\in\RB^{d_y}$, $x = H(y)$ is the unique solution of
		$F(x,y) = 0$.    
	The operators $H$, $F$, and $G$ satisfy the following Lipschitz conditions: 
	$\forall\, x \in \RB^{d_x}, y_{1}, y_{2} \in\RB^{d_y}$,
	\begin{align}
		\|H(y_{1}) - H(y_{2})\| & \leq L_{H}\|y_{1}-y_{2}\|,\label{assump:smooth:FH:ineqH}\\
		\|F(x,y_1) - F(H(y_1),y_1)\| & \leq L_{F} \|x - H(y_1) \|, \nonumber    
		\\
		\|G(x,y_{1}) - G(H(y_{1}),y_{1})\| & \leq \LGx \|x - H(y_{1})\|, \nonumber 
		\\   
		\| G(H(y_{1}),y_{1}) - G(H(y^\star), y^\star) \| & \leq \LGy \| y - y^\star \|. \nonumber 
	\end{align}	
\end{assump}

\begin{assump}
	[Uniform local linearity of $H$ up to order $1+\Hsmooth$]
	 \label{assump:smoothH}
	Assume that $H$ is differentiable and there exist 
	constants $S_H \ge 0$ and $\Hsmooth \in [0.5, 1]$ such that $\forall\, y_1, y_2 \in \RB^{d_y}$,
	\begin{equation*}
		\| H(y_1) -  H(y_2) - \nablaH(y_2) (y_1-y_2)  \| \le  S_H  \|y_1-y_2\|^{1+\Hsmooth}.
	\end{equation*}
\end{assump}

\begin{assump}[Nested local linearity of $F$ and $G$ up to order ($1+\Fsmooth, 1+\Gsmooth$)]
	\label{assump:near-linear}
	There exist matrices $B_1, B_2, B_3$ with compatible dimensions, constants $\SBF, \SBG \ge 0$ and $\Fsmooth, \Gsmooth \in (0,1]$ such that
	\begin{align}
		\|F(x, y) - B_1(x {-} H(y))\| &\le \SBF \left( \|x {-} H(y)\|^{1+\Fsmooth} +  \|y {-} y^{\star}\|^{1+\Fsmooth} \right),
		\label{assump:near-linear-F}
		\\
		\|G(x, y) - B_2(x {-} H(y))- B_3(y {-} y^{\star}) \| &\le \SBG \left( \|x {-} H(y)\|^{1+\Gsmooth} + \|y {-} y^{\star}\|^{1+\Gsmooth} \right). 
		\label{assump:near-linear-G}
	\end{align}
	Furthermore, $\|B_1\| \le L_F$, $\| B_2\| \le \LGx$, and $\|B_3\|\le \LGy$. 
\end{assump}
Assumptions~\ref{assump:smooth} -- \ref{assump:near-linear}, inspired by \cite{han2024finite}, impose regularity conditions on $F$ and $G$.
The nested structure of Assumptions~\ref{assump:smooth} and \ref{assump:near-linear} aligns with the double-loop framework of \eqref{eq:two-loop-approx}.
Notably, differing from the original formulation in \cite{han2024finite}, we additionally require $\Hsmooth \ge 0.5$ in Assumption~\ref{assump:smoothH}.
This modification ensures the validity of the decoupled convergence rates in Proposition~\ref{prop:mart-decouple-rate}, which helps control higher-order residual terms in Assumptions~\ref{assump:smoothH} and \ref{assump:near-linear}.

\begin{rema}\label{remar:fclt:B-form}
    We have the following discussion about Assumption~\ref{assump:near-linear}.
    \begin{enumerate}[(i)] 
        \item  Under the Lipschitz conditions in Assumption~\ref{assump:smooth}, we only require Assumption~\ref{assump:near-linear} holds in a neighborhood of $(x^\star, y^\star)$. See Proposition~\ref{prop:near-linear-tolocal} in Appendix~\ref{sec:fclt:proof:assump}. 
        \item Under Assumptions~\ref{assump:smooth} and \ref{assump:smoothH}, we can obtain Assumption~\ref{assump:near-linear} from the standard local linearity around $(x^\star, y^\star)$, i.e., $F$ and $G$ can be approximated by linear functions of $x - x^\star$ and $y - y^\star$. See Proposition~\ref{prop:local-linear-local} in Appendix~\ref{sec:fclt:proof:assump}. 
        \item Under Assumption~\ref{assump:near-linear}, if $F$ and $G$ are differentiable at $(x^\star, y^\star)$ and $\nabla_x F(x^{\star}, y^{\star}) $ is invertible, then we have
	\begin{align*}
		B_1 & = \nabla_x F(x^{\star}, y^{\star}), \quad
		B_2 =  \nabla_x G(x^{\star}, y^{\star}), \\
		B_3 & = \nabla_y G(x^{\star}, y^{\star}) - \nabla_x G(x^{\star}, y^{\star}) [\nabla_x F(x^{\star}, y^{\star})]^{-1} \nabla_y F(x^{\star}, y^{\star}).
	\end{align*}
    Here, $B_3$ is the Schur complement of $B_1$ in the Jacobian matrix
    $\begin{pmatrix}
        \nabla_x F(x^{\star}, y^{\star}) &  \nabla_y F(x^{\star}, y^{\star}) \\
        \nabla_x G(x^{\star}, y^{\star}) &  \nabla_y G(x^{\star}, y^{\star})) 
    \end{pmatrix} $.
    In particular, if $G(H(\cdot), \cdot)$ is differentiable, then $B_3$ is its Jacobian matrix at $y^\star$.
    For the special linear case with $F(x,y) = A_{11}x + A_{12}y + b_1$ and $G(x,y) = A_{21}x + A_{22} y + b_2$, $B_1 = A_{11}$ and $B_3 = A_{22} - A_{21} A_{11}^{-1} A_{12}$.
    \end{enumerate}
\end{rema}


\begin{assump}[Conditions on step sizes]\label{assump:stepsize-twotime} The step sizes $\{ \alpha_{n} \}_{n=0}^\infty$ and $\{ \beta_{n} \}_{n=0}^\infty$ together with their ratios $\{ \kappa_{n} = \beta_{n} / \alpha_{n} \}_{n=0}^\infty$ satisfy the following conditions.
	\begin{enumerate}[(i)]
		\item \label{assump:stepsize-twotime-tozero} As $n \to \infty$, $\alpha_{n}, \beta_{n}, \kappa_{n} \rightarrow 0$ and $\alpha_{n}$, $\beta_{n}$ are decreasing in $n$.
		\item \label{assump:stepsize-twotime-ratiob} $\frac{\beta_{n-1}}{\beta_{n}} = 1 + \OM(\beta_{n})$ and $\invdiffslow := \lim_{n \to \infty} (\beta_{n+1}^{-1} - \beta_{n}^{-1} )  \in [0, \infty) $. 
		\item \label{assump:stepsize-twotime-ratioa} $\frac{\alpha_{n-1}}{\alpha_{n}} = 1 + \OM(\beta_{n})$ and $\frac{\kappa_{n-1}}{\kappa_{n}} = 1 + \OM\big(\beta_{n}\big)$.
		\item \label{assump:stepsize-twotime-nanb} $n \alpha_{n}$ and $n \beta_{n}$ are non-decreasing in $n$.
		\item \label{assump:stepsize-twotime-moreab} $\beta_{n}^{(1+\Hsmooth) / 2 } = \OM(\alpha_{n})$ and $ \alpha_{n}^{1+\Gsmooth}, \alpha_{n}^{1+\Fsmooth} = o(\beta_{n}) $, with $\Hsmooth$, $\Fsmooth$ and $\Gsmooth$ defined in Assumptions~\ref{assump:smoothH} and \ref{assump:near-linear}.
		
	\end{enumerate}
\end{assump}

Conditions~\ref{assump:stepsize-twotime-tozero} -- \ref{assump:stepsize-twotime-ratioa}
require the step sizes and their ratios to decrease (monotonically) to $0$ at an appropriate rate, ensuring convergence by satisfying the condition that the sum of step sizes diverges to infinity (see Remark~\ref{rema:step-size}).
These conditions are standard in the analysis of two-time-scale SA~\citep{konda2004convergence, mokkadem2006convergence, faizal2023functional, kaledin2020finite}.
The parameter $\invdiffslow$, defined as the limit of $\beta_{n+1}^{-1} - \beta_{n}^{-1}$,  plays a key role in the subsequent analysis.
Condition~\ref{assump:stepsize-twotime-nanb}, inspired by \cite{li2023online}, is an additional regularity requirement designed to facilitate theoretical analysis by preventing extreme variations in the step-size sequences.
Condition~\ref{assump:stepsize-twotime-moreab} ensures that the higher-order residual terms arising from Assumptions~\ref{assump:smoothH} and \ref{assump:near-linear} remain negligible compared to the dominant linear terms. It requires the distinction between the two time scales to fall within an appropriate range, as illustrated in the following example.

\begin{exam}[Polynomially diminishing step sizes]\label{exmp:fclt-stepsize}
	For 
	$\alpha_{n} = \alpha_{0} (n+1)^{-a}$ and $\beta_{n} = \beta_{0} (n+1)^{-b}$ with $a, b \in (0, 1]$ and $k \ge 1$, 
	we have $\invdiffslow = \beta_{0}^{-1}$ if $b = 1$ and $\invdiffslow = 0$ otherwise.
	Assumption~\ref{assump:stepsize-twotime} requires $1 < \frac{b}{a} < 1 + \min\{ \Fsmooth, \Gsmooth \} $ and $\frac{b}{a} \ge \frac{2}{1+\Hsmooth} $. 
	If $\Hsmooth=1$,  $\frac{b}{a} \ge \frac{2}{1+\Hsmooth} $ naturally holds.

    Since our FCLTs rely on the decoupled convergence rates $\EB \|y_n-y^\star\|^2=\OM(\beta_n)$ and $\EB \|x_n-H(y_n)|^2=\OM(\alpha_n)$, the condition $\frac{b}{a}\le 1+\Gsmooth$ is necessary for the decoupling phenomenon, as shown by the phase-transition mechanism in \citep{sarkar2026two}. We impose the strict inequality $\frac{b}{a}<1+\Gsmooth$ to ensure that higher-order error terms do not affect the leading term. The other conditions, $\frac{b}{a}<1+\Fsmooth$ and $\frac{b}{a}\ge\frac{2}{1+\Hsmooth}$, are technical, and whether they can be relaxed is left for future work.
\end{exam}

\begin{assump}[Hurwitz matrices]\label{assump:hurwitz}
	The matrices $B_1$ and $B_3$ defined in Assumption~\ref{assump:near-linear} and the parameter $\invdiffslow$ defined in Assumption~\ref{assump:stepsize-twotime} satisfy that both $- B_1 $ and $- \left( B_3 - \frac{\invdiffslow}{2} \mI \right)$ are Hurwitz matrices, i.e, the real parts of their eigenvalues are negative.
\end{assump}

Assumption~\ref{assump:hurwitz} guarantees that operators $F$ and $G$ exhibit desirable stability properties near the root $(x^\star, y^\star)$, a fundamental criterion in the SA literature~\citep{polyak1992acceleration, konda2004convergence, mokkadem2006convergence, li2023online}.
This condition is weaker than requiring the symmetric parts to be negative definite, and it is also weaker than the strong monotonicity condition commonly used in non-asymptotic analyses~\cite{doan2022nonlinear, han2024finite}. 
As shown in Example~\ref{exmp:fclt-stepsize}, with $\beta_n = \frac{\beta_0}{n+1}$, we have $\invdiffslow = \beta_0^{-1} > 0$.
To guarantee that Assumption~\ref{assump:hurwitz} holds, the initial step size must be chosen carefully and cannot be too small. This requirement is consistent with non-asymptotic SA analyses, where sufficiently large initial step sizes are used to ensure the effect of initialization decays rapidly \citep{moulines2011non}.

The following assumption is about the noise.
Define filtration $( \FM_{n} )_{n \ge 1}$ as 
\begin{equation*}
	\FM_{n} = \sigma ( x_{0},y_{0},\xi_{0},\psi_{0},\xi_{1},\psi_{1},\ldots,\xi_{n-1},\psi_{n-1} ).
\end{equation*}

\begin{assump}[Conditions on the noise]\label{assump:noise-fclt}
	The noise sequences $\{ \xi_{n} \}_{n=0}^\infty$ and $\{ \psi_{n} \}_{n=0}^\infty$ are martingale difference sequences satisfying $\EB[ \xi_{n} \,|\, \FM_{n}] = 0$ and $\EB [ \psi_{n} \,|\, \FM_{n}] = 0$.
	Their fourth moments are uniformly bounded. That is,
	\begin{equation*}
		\sup_{n \ge 0} \EB \| \xi_{n} \|^4 + \sup_{n \ge 0} \EB \| \psi_{n} \|^{4} < \infty.
	\end{equation*}
	Moreover, as $n \to \infty$, 
	\begin{equation*}
		\EB \left[ \begin{pmatrix}
			\xi_{n} \xi_{n}^\top & \xi_{n} \psi_{n}^\top \\
			\psi_{n} \xi_{n}^\top & \psi_{n} \psi_{n}^\top
		\end{pmatrix}  \Bigg|\, \FM_{n} \right]
		\overset{p}{\to} 
		\begin{pmatrix}
			\Sigma_{\xi} & \Sigma_{\xi, \psi} \\
			\Sigma_{\xi, \psi}^\top & \Sigma_{\psi}
		\end{pmatrix}.
	\end{equation*}
\end{assump}

The boundedness of fourth moments is assumed to manage the higher-order residual terms arising from Assumptions~\ref{assump:smooth} and \ref{assump:near-linear}~\citep{han2024finite}. The convergence in probability of the conditional second moment is required to accurately capture the asymptotic behavior~\citep{polyak1992acceleration, mokkadem2006convergence}.

Some intermediate results related to the assumptions are deferred to Appendix~\ref{sec:fclt:proof:assump}.

\section{Main Results}
\label{sec:fclt:main}
In this section, we present our main results. Section~\ref{sec:fclt:construct} details the construction of continuous-time trajectories. Section~\ref{sec:fclt:theo} formally states our decoupled FCLTs and discusses their implications. Finally, Section~\ref{sec:fclt:exam} showcases several examples illustrating applications of our theorems.

\subsection{Construction of the Continuous Trajectories}\label{sec:fclt:construct}

In this subsection, we construct the stochastic processes from the original iterates $ \{ x_{n} \}_{n=0}^\infty $ and $\{ y_{n} \}_{n=0}^\infty $.
Since two-time-scale SA can be viewed as an approximation of the procedure in \eqref{eq:two-loop-approx}, previous work studies the following error terms~\citep{kaledin2020finite, doan2022nonlinear, han2024finite}
\begin{equation}
	\begin{aligned}
		\xhat_{n} = x_{n} - H(y_{n}), \quad 
		\yhat_{n} = y_{n} - y^{\star}.
	\end{aligned}    \label{alg:xyhat}
\end{equation}
Under Assumption~\ref{assump:smooth}, $\| x_{n} - x^\star \|  = \| \xhat_{n} + H(y_{n}) - H(y^\star) \| \le \| \xhat_{n} \|  + L_H \| \yhat_{n} \|$.
Consequently, it suffices to focus on the error terms in \eqref{alg:xyhat}.
In particular, under the assumptions in Section~\ref{sec:fclt:assump}, the following non-asymptotic convergence rates hold.

\begin{prop}[Non-asymptotic convergence rates, informal, {\cite[Theorem~3]{han2024finite}}]\label{prop:mart-decouple-rate}
	Under Assumptions~\ref{assump:smooth} -- \ref{assump:noise-fclt} and some additional regular conditions,
	we have
	$\EB \| \xhat_{n} \|^2 = \OM(\alpha_{n}) $, $\EB \| \yhat_{n} \|^2 = \OM(\beta_{n})$ and $\EB \| \xhat_{n} \|^4 + \EB \| \yhat_{n} \|^4 = \OM (\alpha_{n}^2)$.
\end{prop}
This proposition lays the groundwork for our asymptotic analysis. 
The additional regularity conditions consist of the strong monotonicity assumptions and mild additional requirements on the step sizes. 
The formal presentation is deferred to Appendix~\ref{append:construct-non-asymp}.
Because our primary objective is to establish the asymptotic FCLT based on non-asymptotic rates, we assume these rates as given from this point onward.
\begin{assump}[Non-asymptotic decoupled convergence]\label{assump:mart-decouple-rate}
	The convergence rates in Proposition~\ref{prop:mart-decouple-rate} hold.
\end{assump}

To achieve the goal of ``decoupling'',
we rescale the two sequences $\{ \xhat_{n} \}_{n=1}^\infty$ and $\{ \yhat_{n} \}_{n=1}^\infty$ by the square root of their respective step sizes.
For $n \ge 1$, we define
\begin{equation}\label{eq:xycheck-def}
    \xcheck_{n} := \frac{\xhat_{n}}{ \sqrt{ \alpha_{n-1}} } = \frac{ \xx_n - H(\yy_n) }{ \sqrt{\alpha_{n-1}} }, \quad 
    \ycheck_{n} := \frac{ \yhat_{n} }{ \sqrt{\beta_{n-1}}  } = \frac{ \yy_{n} - \yy^\star }{ \sqrt{ \beta_{n-1} } }.
\end{equation}
Here we divide the errors by $\sqrt{\alpha_{n-1}}$ and $\sqrt{\beta_{n-1}}$ is to simplify the calculation. Under Assumption~\ref{assump:stepsize-twotime}, it makes no difference to divide the errors by  $\sqrt{\alpha_{n-1}}$ and $\sqrt{\beta_{n-1}}$ or $\sqrt{\alpha_{n}}$ and $\sqrt{\beta_{n}}$.
Assumption~\ref{assump:mart-decouple-rate} implies $ \EB \| \xcheck_{n} \|^2 + \EB \| \ycheck_{n} \|^2 = \OM(1) $, indicating that $\xcheck_{n}, \ycheck_{n} = \OM_p(1)$.

Now we construct the continuous trajectories $\{ \xbar_{n}(\cdot) \}_{n=1}^\infty$ and $\{ \ybar_{n}(\cdot) \}_{n=1}^\infty$ by linearly interpolating between the terms of the rescaled sequence $\{ \xcheck_{n} \}_{n=1}^\infty$ and $\{ \ycheck_{n} \}_{n=1}^\infty$:
\begin{align}
	\label{eq:xbar-short}
	\xbar_{n}(t) & = \begin{cases}
		\xcheck_{n}, &  t = 0; \\
		\xcheck_{m} + \frac{t - \sum_{i=n}^{m-1} \alpha_{i} }{\alpha_{m}} (\xcheck_{m+1} - \xcheck_{m}), & t \in \left( \sum_{i=n}^{m-1} \alpha_{i}, \sum_{i=n}^{m} \alpha_{i} \right] \text{ for } m \ge n;
	\end{cases} \\
	\label{eq:ybar-short}
	\ybar_{n}(t) & = \begin{cases}
		\ycheck_{n}, &  t = 0; \\
		\ycheck_{m} + \frac{t - \sum_{i=n}^{m-1} \beta_{i} }{\beta_{m}} (\ycheck_{m+1} - \ycheck_{m}), & t \in \left( \sum_{i=n}^{m-1} \beta_{i}, \sum_{i=n}^{m} \beta_{i} \right] \text{ for } m \ge n.
	\end{cases} 
\end{align}
\Figref{fig:xybar} illustrates this construction intuitively.
This linear interpolation, also employed in \cite{faizal2023functional}, ensures that
$\xbar_{n}(t)$ and $\ybar_{n}(t)$ are continuous w.r.t. $t$. 
Consequently, the analysis can focus on weak convergence within the space of continuous functions. 
Moreover, in constructing $\xbar_{n}(\cdot)$ and $\ybar_{n}(\cdot)$, we place $\xcheck_{n+k}$ and $\ycheck_{n+k}$ with different time increments. This choice aligns with the different step sizes, so that the rescaled sequences can be interpreted as discretizations of their respective limiting processes. As shown in Remark~\ref{rema:separate-converge}, our results are separate-convergence results, and the FCLTs for each time scale are established individually.
Thus, using $t$ as a generic time index does not induce ambiguity.  
Additional details of the construction are provided in Appendix~\ref{append:construct-details}.

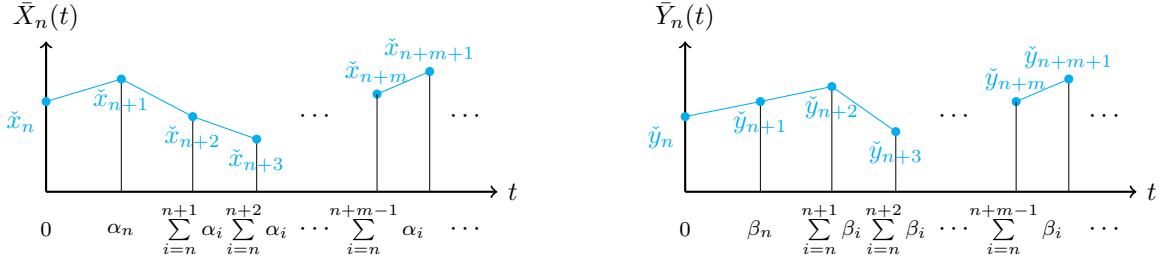
\begin{figure}[t]
	\centering
	\begin{minipage}[b]{0.45\textwidth} 
		\centering
	\begin{tikzpicture}
	\draw[->,thick] (0,0) -- (6,0) node[right] {\small $t$};
	\draw[->,thick] (0,0) -- (0,2) node[above] {\small $\bar{X}_{n}(t)$};
	
	\coordinate (x0) at (0, 1.2);
	\node (0) at (0, -0.5) { \scriptsize $0$};
	\coordinate (x1) at (1.0, 1.5);
	\coordinate (t1) at (1.0, 0);
	\node (alphan) at (1.0, -0.5) { \scriptsize $\alpha_{n}$};
	\coordinate (x2) at (1.95, 1.0);
	\coordinate (t2) at (1.95, 0);
	\coordinate (x3) at (2.8, 0.7);
	\coordinate (t3) at (2.8, 0);
	\node (dotsx) at (3.6, 1.0) {\small $\cdots$};
	\node (dotst) at (3.6, -0.5) {\small $\cdots$};
	\coordinate (xm) at (4.4, 1.3);
	\coordinate (tm) at (4.4, 0);
	\coordinate (xm1) at (5.1, 1.6);
	\coordinate (tm1) at (5.1, 0);
	\coordinate (xm2) at (5.75, 0.8);
	\coordinate (tm2) at (5.75, 0);
	\node (dotsx2) at (5.6, 1.0) {\small $\cdots$};
	\node (dotst2) at (5.6, -0.5) {\small $\cdots$};

	\filldraw [cyan] (x0) circle (1.5pt) node[below left] {\small $\xcheck_{n}$};
	
	\draw[very thin] (x1) -- (t1); 
	\draw[-, cyan] (x0) -- (x1);
	\filldraw[cyan] (x1) circle (1.5pt) node[below] {\small $\check{x}_{n+1}$};
	
	\draw[very thin] (x2) -- (t2) node[below] {\scriptsize $\sum\limits_{i=n}^{n+1}\alpha_{i}$};
	\draw[-, cyan] (x1) -- (x2);
	\filldraw[cyan] (x2) circle (1.5pt) node[below] {\small $\check{x}_{n+2}$};
	
	\draw[very thin] (x3) -- (t3) node[below] {\scriptsize $\sum\limits_{i=n}^{n+2}\alpha_{i}$};
	\draw[-, cyan] (x2) -- (x3);
	\filldraw[cyan] (x3) circle (1.5pt) node[below] {\small $\check{x}_{n+3}$};
	
	\filldraw [cyan] (xm) circle (1.5pt) node[above] {\small $\xcheck_{n+m}$};
	\draw[very thin] (xm) -- (tm) node[below] {\scriptsize $\sum\limits_{i=n}^{n+m-1}\alpha_{i}$};
	
	\draw[very thin] (xm1) -- (tm1); 
	\draw[-, cyan] (xm) -- (xm1);
	\filldraw[cyan] (xm1) circle (1.5pt)
	node[above] {\small $\check{x}_{n+m+1}$};
	
	
\end{tikzpicture}
	\end{minipage}
	\hspace{0.05\textwidth} 
	\begin{minipage}[b]{0.45\textwidth}
		\centering
	\begin{tikzpicture}
	\draw[->,thick] (0,0) -- (6,0) node[right] {\small $t$};
	\draw[->,thick] (0,0) -- (0,2) node[above] {\small $\bar{Y}_{n}(t)$};
	
	\coordinate (x0) at (0, 1.0);
	\node (0) at (0, -0.5) { \scriptsize $0$};
	\coordinate (x1) at (1.0, 1.2);
	\coordinate (t1) at (1.0, 0);
	\node (alphan) at (1.0, -0.5) { \scriptsize $\beta_{n}$};
	\coordinate (x2) at (1.95, 1.4);
	\coordinate (t2) at (1.95, 0);
	\coordinate (x3) at (2.8, 0.8);
	\coordinate (t3) at (2.8, 0);
	\node (dotsx) at (3.6, 1.0) {\small $\cdots$};
	\node (dotst) at (3.6, -0.5) {\small $\cdots$};
	\coordinate (xm) at (4.4, 1.2);
	\coordinate (tm) at (4.4, 0);
	\coordinate (xm1) at (5.1, 1.5);
	\coordinate (tm1) at (5.1, 0);
	\coordinate (xm2) at (5.75, 0.9);
	\coordinate (tm2) at (5.75, 0);
	\node (dotsx2) at (5.6, 1.0) {\small $\cdots$};
	\node (dotst2) at (5.6, -0.5) {\small $\cdots$};

	\filldraw [cyan] (x0) circle (1.5pt) node[below left] {\small $\ycheck_{n}$};
	
	\draw[very thin] (x1) -- (t1); 
	\draw[-, cyan] (x0) -- (x1);
	\filldraw[cyan] (x1) circle (1.5pt) node[below] {\small $\check{y}_{n+1}$};
	
	\draw[very thin] (x2) -- (t2) node[below] {\scriptsize $\sum\limits_{i=n}^{n+1}\beta_{i}$};
	\draw[-, cyan] (x1) -- (x2);
	\filldraw[cyan] (x2) circle (1.5pt) node[below] {\small $\check{y}_{n+2}$};
	
	\draw[very thin] (x3) -- (t3) node[below] {\scriptsize $\sum\limits_{i=n}^{n+2}\beta_{i}$};
	\draw[-, cyan] (x2) -- (x3);
	\filldraw[cyan] (x3) circle (1.5pt) node[below] {\small $\check{y}_{n+3}$};
	
	\filldraw [cyan] (xm) circle (1.5pt) node[above] {\small $\ycheck_{n+m}$};
	\draw[very thin] (xm) -- (tm) node[below] {\scriptsize $\sum\limits_{i=n}^{n+m-1}\beta_{i}$};
	
	\draw[very thin] (xm1) -- (tm1); 
	\draw[-, cyan] (xm) -- (xm1);
	\filldraw[cyan] (xm1) circle (1.5pt)
	node[above] {\small $\check{y}_{n+m+1}$};
	
\end{tikzpicture}
	\end{minipage}
	\caption{Construction of $\xbar_{n}(\cdot)$ and $\ybar_{n}(\cdot)$}
	\label{fig:xybar}
\end{figure}

\subsection{Decoupled Functional Central Limit Theorems}

\label{sec:fclt:theo}

With $\{ \xbar_{n}(\cdot) \}_{n=1}^\infty$ and $\{ \ybar_{n}(\cdot) \}_{n=1}^\infty$ defined in \eqref{eq:xbar-short} and \eqref{eq:ybar-short}, we could establish the following decoupled FCLT.

\begin{thm}[Decoupled functional central limit theorems]\label{thm:fclt-twotime}
	Suppose that Assumptions~\ref{assump:smooth} -- \ref{assump:noise-fclt} and \ref{assump:mart-decouple-rate} hold.
	\begin{enumerate}[(i)]
		\item \label{thm:fclt-x} The stochastic process $\xbar_{n}(\cdot)$ defined in \eqref{eq:xbar-short} converges weakly to the stationary solution of the following SDE 
		\begin{equation}\label{eq:fclt-twotime:x-sde}
			d \ermX(t) = - B_1 \ermX(t) dt + \Sigma_{\xi}^{1/2} d \ermW^{d_x}(t).
		\end{equation}
		The rescaled iterate $ \xcheck_{n} $ defined in \eqref{eq:xycheck-def} converges weakly to the invariant distribution of \eqref{eq:fclt-twotime:x-sde}, i.e., $\NM(0, \Sigma_x)$, where $\Sigma_x$ satisfies the following Lyapunov equation 
		\begin{equation}\label{eq:fclt-twotime:x-var}
			B_1 \Sigma_x + \Sigma_x B_1 = \Sigma_{\xi}.
		\end{equation}

		\item \label{thm:fclt-y} The stochastic process $\ybar_{n}(\cdot)$ defined in \eqref{eq:ybar-short} converges weakly to the stationary solution of the following SDE 
		\begin{equation}\label{eq:fclt-twotime:y-sde}
			d \ermY(t) = - \left(B_3 - \frac{\invdiffslow \mI}{2}  \right) \ermY(t) dt + \widetilde{\Sigma}_{\psi}^{1/2} d \ermW^{d_y}(t),
		\end{equation}
		where the matrix $\widetilde{\Sigma}_{\psi}$ is defined as 
		\begin{equation}\label{eq:cov-z}
			\widetilde{\Sigma}_{\psi} := \Sigma_{\psi} - B_2 B_1^{-1} \Sigma_{\xi,\psi} - \Sigma_{\xi, \psi}^\top B_1^{-\top} B_2^\top + B_2 B_1^{-1} \Sigma_{\xi} B_1^{-\top} B_2^\top.
		\end{equation}
		The rescaled iterate $ \ycheck_{n}$ defined in \eqref{eq:xycheck-def} converges weakly to the invariant distribution of \eqref{eq:fclt-twotime:y-sde}, i.e., $\NM(0, \Sigma_y)$, where $\Sigma_y$ satisfies the following Lyapunov equation
        \begin{equation}\label{eq:fclt-twotime:y-var}
			\left(B_3 - \frac{\invdiffslow \mI}{2} \right) \Sigma_y + \Sigma_y \left( B_3 - \frac{\invdiffslow \mI}{2} \right) = \widetilde{\Sigma}_{\psi}.
		\end{equation}
	\end{enumerate}
\end{thm}
In Theorem~\ref{thm:fclt-twotime}, the ``stationary solution'' refers to a solution that is a stationary process. Specifically, this corresponds to a solution where the initial distribution matches the invariant distribution. Under Assumption~\ref{assump:hurwitz}, the solutions to \eqref{eq:fclt-twotime:x-var} and \eqref{eq:fclt-twotime:y-var}  are unique and have the following explicit forms \citep[Theorem~7.5]{antsaklis2005linear}
\begin{align*}
    \Sigma_x &= \int_0^\infty \exp(-B_1 t) \Sigma_{\xi} \exp(- B_1^\top t)  dt
    \\
    \Sigma_y &= \int_0^\infty \exp \left( - \Big(B_3 - \frac{\invdiffslow I}{2} \Big) t \right) \widetilde{\Sigma}_{\psi} \exp \left( - \Big(B_3^\top - \frac{\invdiffslow I}{2} \Big) t \right) dt.
\end{align*} 
The presence of $\invdiffslow$ 
highlights the impact of step sizes on asymptotic behavior. As illustrated in Example~\ref{exmp:fclt-stepsize}, for $\beta_{n} = \Theta(n^{-b})$ with $b \in (0,1]$, $\invdiffslow = \beta_{0}^{-1} \neq 0$ only when $b=1$, which requires the initial step size to be sufficiently large. This aligns with the convergence analysis for standard SA~\citep{moulines2011non}.
Due to the assumption $\beta_{n} / \alpha_{n} = o(1)$, the choice of $\alpha_n$ does not influence the asymptotic behavior for the fast time scale in \eqref{eq:fclt-twotime:x-sde} (see Remark~\ref{rema:step-size}).

For Theorem~\ref{thm:fclt-twotime}, we also have the following explanations.

\begin{rema}[Effect of interdependence]\label{rema:effect-inter}
The limiting process for the fast time scale coincides with that of standard SA applied to the operator $F(\cdot,y^*)$ driven by the original noise $\xi_n$.
This implies the slow-time-scale update does not affect the fast-time-scale limiting dynamics. 
Intuitively, because $x_n$ is updated with the larger step size, the evolution of $y_n$ is asymptotically negligible from the fast-time-scale perspective.

In contrast, the limiting process for the slow time scale is influenced by the fast-time-scale update. As shown in Remark~\ref{remar:fclt:B-form},
$B_3$ is the Jacobian matrix of $G(H(\cdot), \cdot)$ evaluated at $y^\star$.
Meanwhile, $\widetilde{\Sigma}_{\psi}$ is the asymptotic covariance of $\breve{\psi}_n := \psi_n - B_2 B_1^{-1} \xi_n$.
Thus, asymptotically, the slow iterate behaves like a standard SA iterate for the effective operator $G(H(\cdot), \cdot)$, but with a modified noise term $\breve{\psi}_n$ that incorporates contributions from both time scales.
This matches the approximation viewpoint in \eqref{eq:two-loop-approx} and confirms that $G(H(\cdot),\cdot)$ is the intrinsic objective operator for the outer loop.

Next we provide an intuitive illustration in the linear case
$F(x,y) = A_{11}x + A_{12} y$ and $G(x,y) = A_{21} x + A_{22} y$, where $A_{11} $ and $A_{22}$ are invertible.
Then $x^\star=0$, $y^\star$ = 0, $H(y) = - A_{11}^{-1} A_{12} y$, and $G(H(y), y) = (A_{22} - A_{21} A_{11}^{-1} A_{12} ) y = B_3 y$.
Because $\alpha_n \gg \beta_n$, the sequence $y_n$ is quasi-static from the fast-time-scale perspective (i.e., $y_{n+1} \approx y_n$).
With $\xhat_{n} = x_{n} - H(y_n) = x_{n} + A_{11}^{-1} A_{12} y_n $,
We have $\xhat_{n+1} - \xhat_{n} \approx x_{n+1} - x_{n}$ and  $A_{11}x_n + A_{12}y_n = A_{11} \xhat_{n}$.
So
        $\xhat_{n+1}
        \approx \xhat_{n} - \alpha_n (A_{11} \xhat_{n} + \xi_n)$,
which is approximately a standard SA iteration for $A_{11} x = 0$.
Because $-A_{11}$ is Hurwitz, the effect of initialization decays at the rate $\prod_{i=0}^n(1 - c \alpha_i)$ for some $c > 0$, 
and for large $n$ the iterate $\xhat_{n}$ is dominated by its accumulated noise; heuristically,
$\xhat_{n} = x_{n} + A_{11}^{-1} A_{12} y_n \approx \text{noise}$.
Substituting this behavior into the slow update gives
    \begin{equation*}
        y_{n+1} 
        \approx y_n - \beta_{n} \Big[ \left(  A_{22} - A_{21} A_{11}^{-1} A_{12} \right) y_n + \underbrace{ \psi_n + A_{21} \cdot \text{noise brought by}\ x_n }_{\approx \breve{\psi}_n} \Big].
    \end{equation*}
Detailed calculation yields that the original slow-time-scale noise $\xi_n$, together with the noise transmitted through the fast-scale iterate $x_n$, combines to yield (approximately) the modified noise~$\breve{\psi}_n$.
\end{rema}

\begin{rema}[Separate and decoupled convergence]\label{rema:separate-converge}
The asymptotic statements in Theorem~\ref{thm:fclt-twotime} are formulated in terms $\xbar_{n}(\cdot)$ or $\ybar_{n}(\cdot)$ separately,
rather than as a joint convergence result. 
This ``separate convergence'' also provides an additional interpretation of decoupled convergence: we can derive an FCLT for each time scale individually, and the influence of the other time scale appears only through the coefficient matrices in the corresponding limiting SDE. 
This implies that, at the level of the rescaled limits, each time scale can be analyzed in a standard SA fashion.

This perspective also highlights a benefit of adopting strict two-time-scale step sizes when compared with the single-time-scale regime. To illustrate that, consider the case $\beta_n / \alpha_n = \kappa $ and $ \invdiffslow=0$, assume $F$ and $G$ are linear, i.e., $F(x,y) = A_{11} x + A_{12} y$ and $G(x,y) = A_{21} x + A_{22} y$.
Then the two-time-scale recursion can be rewritten as a single-time-scale SA iteration
\begin{equation*}
        u_{n+1} = u_n - \alpha_n \big(\widetilde{A} u_n + \phi_n\big),
    \end{equation*}
where $u_n^\top =\left(x_n^\top,\  y_n^\top\right)
$, 
$\widetilde{A} = \left(
A_{11},\   A_{12};\ \  \kappa A_{21},\ \kappa A_{22} \right)$,
and $\phi_n^\top = \left(\xi^\top_n,\ \kappa\psi^\top_n \right)$.
Hence the asymptotic theory follows from standard SA arguments \cite{borkar2009stochastic}. In particular, if we construct $\bar{U}_n(\cdot)$ by linearly interpolating $\frac{u_n - u^\star}{\sqrt{\alpha_{n-1}}} $ with $(u^\star)^\top = \left( (x^\star)^\top,\  (y^\star)^\top \right)$ in the same way as $\xbar_{n}(\cdot)$ and $\ybar_{n}(\cdot)$,
then $\bar{U}_n(\cdot)$ converges weakly to the stationary solution of
\begin{equation*}
    d \ermU(t) = - \widetilde{A} \ermU(t) dt + \Sigma_{\phi}^{1/2} d \ermW^{d_x+d_y}(t),
\end{equation*}
where $\Sigma_{\phi} =\left( \Sigma_\xi, \  \kappa \Sigma_{\xi,\psi}; \ \ 
\kappa \Sigma^\top_{\xi, \psi}, \ \kappa^2\Sigma_{\psi} \right) $.
In general, neither $\widetilde{A}$ nor $\Sigma_{\phi}$ is block diagonal,
so the $x$- and $y$-components of the limiting dynamics are intrinsically coupled. The invariant distribution is  $\NM(0, \Sigma_u)$ where $\Sigma_u$ solves $\widetilde{A} \Sigma_u + \Sigma_u (\widetilde{A})^\top = \Sigma_\phi$.
As a result, if one is interested only in the $x$-part or only in the $y$-part, extracting an explicit asymptotic covariance is typically nontrivial. 
\end{rema}

\begin{rema}[Comparison with previous work]
Compared to the FCLT presented in~\cite{faizal2023functional}, which rescales the error terms $\xhat_{n}$ and $\yhat_{n}$ using the square root of the fast-time-scale step size, we apply distinct rescalings to $\xhat_{n}$ and $\yhat_{n}$ based on the square root of their respective step sizes.
This refinement enables us to derive decoupled FCLTs under local linearity conditions, offering a more precise characterization of the asymptotic behavior in two-time-scale SA. Furthermore, our FCLTs can recover the classical CLTs from \cite{konda2004convergence, mokkadem2006convergence}, as demonstrated in Corollary~\ref{cor:clt}.
\end{rema}

\begin{cor}[Central limit theorems]\label{cor:clt}
	Under the conditions of Theorem~\ref{thm:fclt-twotime}, we have 
	$\alpha_{n}^{-1/2} (\xx_{n} - H(\yy_{n})) \weakconverge \NM(0, \Sigma_x)$, 
	$ \beta_{n}^{-1/2} (\yy_{n} - \yy^\star) \weakconverge \NM(0,\Sigma_{y})$ and
	$ \alpha_{n}^{-1/2} (\xx_{n} - \xx^\star) \weakconverge \NM(0, \Sigma_{x})$.
\end{cor}
Moreover, our theorem has the following implications.

\begin{rema}[Convergence rates to the invariant distribution]\label{rema:rate-to-invar}
Our FCLT suggests that, for sufficiently large $n$,
the rescaled iterate sequences $\{ \xcheck_{n+k} \}_{k=0}^\infty$ and $\{ \ycheck_{n+k} \}_{k=0}^\infty$
can be interpreted as discretizations of the limiting SDEs \eqref{eq:fclt-twotime:x-sde} and \eqref{eq:fclt-twotime:y-sde}, respectively. 
Although Theorem~\ref{thm:fclt-twotime} identifies a stationary limiting process, stationarity arises in the asymptotic regime $n \to \infty$.
For finite $n$, the initial distribution of $\xbar_{n}$ (resp. $\ybar_{n}$) is the distribution
of $\xcheck_{n}$ (resp. $\ycheck_{n}$), which need not coincide with the invariant distribution of the corresponding SDE. Nevertheless, $\xbar_{n}(\cdot)$ and $\ybar_n(\cdot)$ can still be viewed as approximations to non-stationary Ornstein-Uhlenbeck processes.

To illustrate, consider $\ybar_{n}(\cdot)$ as approximating the SDE \eqref{eq:fclt-twotime:y-sde} with initial distribution given by that of $\ycheck_{n}$. The solution admits the explicit representation
$\ermY(t) = \exp(-\widetilde{B}_3 t) \ermY(0) + \int_0^t \exp(-\widetilde{B}_3 (t-s) ) \widetilde{\Sigma}_\psi^{1/2} d \ermW^{d_y}(t)$,
where$\widetilde{B}_3 := B_3 - \frac{\invdiffslow \mI}{2}$.
If $\ermY(0)$ is distributed according to the invariant distribution, then $\ermY(\cdot)$ is stationary; otherwise, the rate at which $\ermY(t)$ converges in distribution to the invariant distribution is governed by the spectral properties of $\widetilde{B}_3$ (see Proposition~\ref{prop:hurwitz-expbound} in Appendix~\ref{append:fclt-interpre}).
Intuitively, larger real parts of the eigenvalues of $\widetilde{B}_3$ yield faster convergence to the invariant distribution.

This observation also suggests practical accelerations. For example, one may introduce a (typically diagonal) preconditioner $P$ and consider
$y_{n+1} = y_{n} - \beta_{n}P\left(G(x_{n},y_{n}) + \psi_{n}\right)$
so that the effective drift matrix becomes $P\widetilde{B}_3$ with potentially improved spectral properties. 
More generally, one can employ an adaptive preconditioner $P_n$, in the spirit of adaptive methods such as Adam \citep{kingma2014adam}.
\end{rema}

\begin{rema}[Simultaneous convergence over a range of iterates]\label{rema:simul-converge}
The FCLT captures trajectory-level behavior and therefore provides finer information than single-point convergence. We illustrate this via a simple example on the slow time scale. 

Let $\ermY_\infty(\cdot)$ denote a solution to the limiting SDE \eqref{eq:fclt-twotime:y-sde}. 
For any fixed \(T>0\), define \(h:C([0,\infty);\RB^d)\to\RB\) by \(h(f):=\sup_{0\le t\le T}\norm{f(t)}\), and let \(J\) be a positive integer such that \(J\ge T\). 
With $\rho_J$ and $\rho$ defined in \eqref{eq:metric},
for any \(\epsilon>0\), if \(\rho(f,g)<2^{-J}\epsilon/(1+\epsilon)\), then \(\rho_J(f,g)<\epsilon\), because \(\rho(f,g)\ge 2^{-J}\rho_J(f,g)/(1+\rho_J(f,g))\). Moreover, \(\left|h(f)-h(g)\right|\le \sup_{0\le t\le T}\norm{f(t)-g(t)}\le \rho_J(f,g)<\epsilon\). Thus, \(h\) is continuous with respect to \(\rho\). Since \(\ybar_n\weakconverge\ermY_\infty\) in \(C([0,\infty);\RB^d)\), the continuous mapping theorem yields \(\sup_{0\le t\le T}\norm{\ybar_n(t)}=h(\ybar_n)\weakconverge h(\ermY_\infty)=\sup_{0\le t\le T}\norm{\ermY_\infty(t)}\).
For any $\delta > 0$, there exists $M > 0$ such that $\PB(\sup_{0 \le t \le T} \norm{\ermY_\infty(t)} \ge M) \le \delta$ with the derivation deferred to Appendix~C.2 in the supplementary material.
It follows that $\limsup_{n \to \infty} \PB(\sup_{0 \le t \le T} \|\ybar_{n}(t) \| \ge M) \le \delta$.

We now translate this trajectory-level bound back to the original discrete iterates. Let $N^\beta(n,T)$ be the largest integer $n'$ such that $\sum_{k=n}^{n'-1} \beta_k \le T$. 
For $\beta_n =\Theta(n^b)$, we have $N^\beta(n,T) - n = \Theta(n^b) $ (see Proposition~\ref{prop:aux:step-size} in Appendix~\ref{append:construct-details}).
By the definition $\ybar_{n}(\cdot)$,
we have
\begin{equation*}
    \sup_{0 \le t \le T} \|\ybar_{n}(t) \| 
    \ge \sup_{n \le k \le N^\beta(n,T)} \norm{ \ycheck_{k} }
    = \sup_{n \le k \le N^\beta(n,T)} \frac{\norm{ y_k - y^\star }}{\sqrt{\beta_{k-1} }} 
    \ge \frac{1}{\sqrt{\beta_{n-1} }} \sup_{n \le k \le N^\beta(n,T)} \norm{ y_k - y^\star }.
\end{equation*}
Consequently,
\begin{equation*}
    \limsup_{n \to \infty}\PB\left(\sup_{n \le k \le N^\beta(n,T)} \norm{ y_k - y^\star } \ge M \sqrt{\beta_{n-1}}\right) 
  \le \limsup_{n \to \infty} \PB\left(\sup_{0 \le t \le T} \|\ybar_{n}(t) \| \ge M\right) \le \delta.
\end{equation*}
In words, with probability at least $1-\delta$ for sufficiently large $n$, all iterates  $\{y_k\}_{k=n}^{N^\beta(n,T)}$ remain in a neighborhood of $y^\star$ of radius on the order $\OM(\sqrt{\beta_{n}} )$.
Compared with single-point convergence of $y_n$, the FCLT yields a uniform guarantee over an entire window of iterates whose length is determined by the step-size selection. 
Such uniform-in-iteration guarantees are useful in online settings for mitigating peeking issues in sequential monitoring~\citep{johari2017peeking,howard2021time}.
In addition to the supremum functional, one may also consider averaging-type functionals, which can potentially reduce the variance; a rigorous characterization is left for future work.
\end{rema}

\begin{rema}[Inspiration for algorithm design]\label{rema:alg-design}
Our results suggest additional flexibility in designing two-time-scale schemes. Suppose the goal is to solve $\widetilde{G}(y) = 0$, and we have found $F(x,y)$ and $G(x,y)$ such that the inner solution map $H(y)$ satisfies $\widetilde{G}(y) = G(H(y),y)$.
Then we may replace $G$ by an another operator $G'(x,y) = G(x,y) + h(F(x,y))$ as long as $G'(H(y),y) = \widetilde{G}(y)$.
For example, consider a linear choice $h(z) = Az$, i.e., $G'(x,y) = G(x,y) + AF(x,y)$ for some matrix $A$.
Under local linearity, we have $F(x,y) \approx B_1(x-H(y))$ and $G(x,y) \approx B_2(x-H(y)) + B_3(y-y^\star) $ around $(x^\star, y^\star)$, and hence $G'(x,y) \approx (B_2 + AB_1) (x - H(y)) + B_3(y-y^\star)$.
Thus, the local linearity is preserved, with $B_2$ replaced by $B_2 + AB_1$ and $B_3$ unchanged.

Importantly, this modification does not alter the effective slow-scale noise. Indeed, because evaluating $F(x_n, y_n)$ 
incurs the noise $\xi_n$, the $y$-update under $G'$ becomes
\begin{equation*}
    y_{n+1}
    = y_{n} - \beta_n (G'(x_n, y_n) + \psi_{n} + A \xi_n).
\end{equation*}
Recall that in Remark~\ref{rema:effect-inter}, the effective noise for the slow-time-scale recursion is $\breve{\psi}_n := \psi_n - B_2 B_1^{-1} \xi_n$.
After replacing $G$ by $G'$, the effective noise becomes $\psi_n +A\xi_n - (B_2 + AB_1)B_1^{-1}\xi_n  = \breve{\psi}_n$,
which is unchanged. Consequently, decomposing the original problem into a two-loop form yields a family of algorithms parameterized by $A$, all of which share the same asymptotic limit under our theorem.
For further discussion and concrete examples, see Examples~\ref{exmp:momentum} and~\ref{exmp:gtd}.
\end{rema}

\subsection{Examples}
\label{sec:fclt:exam}

In this subsection, we present several examples of two-time-scale SA and demonstrate the application of our FCLTs to these cases. 
The operator-related assumptions are verified within each example. A brief discussion of how to ensure the remaining assumptions is deferred to Appendix~\ref{append:exmp}.

In the following two examples, we consider the problem of minimizing a strongly convex function $f$ with a unique minimizer $x_o^\star := \arg\min_{x} f(x)$, using noisy gradient observations.

\begin{exam}[SGD with Polyak-Ruppert averaging]\label{exmp:PRave}
	Stochastic gradient descent (SGD), introduced by \citep{robbins1951stochastic}, iteratively updates the iterate $x_t$ by~\eqref{eq:SGD}.
	To improve the convergence of SGD, an additional averaging step~\eqref{eq:SGD-averaging} is often used~\citep{polyak1992acceleration,ruppert1988efficient}
	\begin{subequations} \label{eq:SGD-both}
		\begin{align}
			x_{n+1} &= x_{n} - \alpha_n(\nabla f(x_{n}) + \xi_{n})~\text{with}~\alpha_n = \frac{\alpha_0}{ (n+1)^a }~\Big(\frac{1}{2}  < a < 1\Big)  \label{eq:SGD}, \\
			y_{n+1} &= \frac{1}{n+1}\sum_{\tau=0}^n x_{\tau} = y_n - \beta_n(y_n-x_n)~\text{with}~\beta_n = \frac{1}{n+1}. \label{eq:SGD-averaging}
		\end{align}
	\end{subequations}
	This update is a special case 
	of \eqref{alg:xy} with $F(x, y) = \nabla f(x)$, $G(x, y) = y -x$, and $H(y) \equiv x^{\star} = x_o^{\star}$.
	It follows that $G(H(y), y) = y - x_o^{\star}$, $y^{\star} 
    = x_o^{\star}$ and $\invdiffslow=1$. 
    If we assume $f$ is twice differentiable at $x_o^\star$ and $\nabla f$ is Lipschitz continuous, then Assumptions~\ref{assump:smooth} -- \ref{assump:hurwitz} holds with $\Fsmooth = \Gsmooth = \Hsmooth = 1$. One can check 
	$B_1 = \nabla^2 f(x_o^\star)$, $B_2 = - \mI$, $B_3 - \frac{\invdiffslow }{2} \mI =  \frac{1}{2} \mI $, $\Sigma_{\psi} = 0$, $\Sigma_{\xi,\psi} = 0$ and $\widetilde{\Sigma}_\psi = [ \nabla^2 f(x_o^\star) ]^{-1} \Sigma_{\xi} [ \nabla^2 f(x_o^\star) ]^{-1} $.
    A specialness of this example is that the slow-time-scale update is deterministic.
    Moreover, \citet[Theorems~1 and 3]{moulines2011non} establish that $\EB \norm{\xhat_n}^2 = \OM(\alpha_n) = \OM(n^{-\alpha})$ and $\EB\norm{\yhat_n}^2 = \OM(\beta_n) = \OM(1/n)$ when $\nabla^2 f$ is Lipschitz continuous.
    Eq.~(35) in their proof further implies $\EB\norm{\xhat_n}^4 = \OM(\alpha_n^2) = \OM(n^{-2\alpha})$.
    Given $\EB\norm{\xhat_n}^4 = \OM(\alpha_n^2)$, it is natural to expect $\EB \norm{\yhat_n}^4 = \OM(\alpha_n^2)$.
    In particular, by Minkowski's inequality,
    \begin{equation*}
        \EB \norm{\yhat_n}^4
        \le n^{-4} \left( \sum_{i=0}^{n-1} \left(\EB\norm{\xhat_{i}}^4 \right)^{1/4} \right)^4
        \lesssim n^{-4} n^{4(1-\alpha/2)}
        \lesssim n^{-2\alpha} = \OM(\alpha_n^2).
    \end{equation*}
    Consequently, Assumption~\ref{assump:mart-decouple-rate} still hold.

	By Theorem~\ref{thm:fclt-twotime}, the solution to \eqref{eq:fclt-twotime:x-var} is $ \Sigma_x = \int_0^\infty \exp(-\nabla^2 f(x_o^\star) t) \Sigma_{\xi} \exp(- \nabla^2 f(x_o^\star) t)  dt$, and the solution to \eqref{eq:fclt-twotime:y-var} is $\Sigma_{y} = \widetilde{\Sigma}_{\psi} = [ \nabla^2 f(x_o^\star) ]^{-1} \Sigma_{\xi} [ \nabla^2 f(x_o^\star) ]^{-1}$.
    Then $\alpha_n^{-1/2} (x_n - x_o^\star) \weakconverge \NM(0, \Sigma_x) $ and $\sqrt{n} (y_n - x_o^\star) \weakconverge \NM(0, \widetilde{\Sigma}_{\psi})$ and
	we recover the CLT for the averaged iterate from \cite{polyak1992acceleration}.
    The FCLT for the fast time scale is similar to the FCLT for standard SA in \cite[Section~II.4.5.3]{benveniste2012adaptive} and \cite[Section~8.4]{borkar2009stochastic}, 
    whereas the FCLT for the slow time scale additionally yields a trajectory-level characterization of the averaged iterate. In particular, it shows that the rescaled averaged iterate admits a diffusion approximation, and can be interpreted as a discretization of the Ornstein--Uhlenbeck SDE
    $d \ermY(t) = - \frac{1}{2}  \ermY(t) dt + \widetilde{\Sigma}_{\psi}^{1/2} d \ermW^{d_y}(t)$.
    Moreover, by Remark~\ref{rema:simul-converge}, for some fixed $c>1$, with high probability we have
    $\max_{n\le k\le \lfloor cn\rfloor}\|y_k-x_o^\star\|=\OM(n^{-1/2})$, i.e., the averaged iterates stay uniformly within an $\OM(n^{-1/2})$-neighborhood of $x_o^\star$ over the time window $[n,cn]$.
    In this example, while the update for the slow iterate $y_{n}$ is deterministic, the limiting trajectories remain stochastic, reflecting the influence of the fast-time-scale updates on the slow-time-scale dynamics.

    Moving forward, we may also consider a weighted average by taking
    $\beta_n=\frac{\beta_0}{n+1}$.
    Since
    $y_{n+1}=(1-\beta_n)y_n+\beta_n x_n$,
    a larger $\beta_n$ assigns more weight to the most recent iterate $x_n$.
    To ensure that $(B_3-\invdiffslow I/2)$ is Hurwitz, we require
    $\invdiffslow=\beta_0^{-1}<2$, i.e., $\beta_0>1/2$.
    Under this condition, the solution to \eqref{eq:fclt-twotime:y-var} is
    $\Sigma_y=\frac{1}{2-\invdiffslow}\widetilde{\Sigma}_\psi$.
    With
    $\ycheck_n=(y_n-y^\star)/\sqrt{\beta_{n-1}}
    =\invdiffslow^{1/2}\sqrt{n}\,(y_n-y^\star)$,
    we obtain
    $\sqrt{n}\,(y_n-y^\star)\ \weakconverge\
    \NM\!\left(0,\ \frac{1}{2\invdiffslow-\invdiffslow^2}\,\widetilde{\Sigma}_\psi\right)$.
    The factor $\frac{1}{2\invdiffslow-\invdiffslow^2}$ (for $0<\invdiffslow<2$) is minimized at
    $\invdiffslow=\beta_0^{-1}=1$.
    Therefore, among this family of weighted averages, the vanilla average ($\beta_0=1$) is asymptotically optimal in the sense of minimizing the asymptotic covariance of the estimator of $x_o^\star$.
   
    The same discussion can be equivalently phrased by keeping the original step size
    $\beta_n=\frac{1}{n+1}$ and modifying the operator to
    $G(x,y)=\beta_0(y-x)$.
    In this case, $B_2=-\beta_0\mI$ and $B_3-\frac{\invdiffslow}{2}\mI=\frac{2\beta_0-1}{2}\mI$,
    while $\widetilde{\Sigma}_\psi
    =\beta_0^2\,[\nabla^2 f(x_o^\star)]^{-1}\Sigma_\xi[\nabla^2 f(x_o^\star)]^{-1}$.
    Substituting these quantities into \eqref{eq:fclt-twotime:y-var} yields the same limit
    $\sqrt{n}\,(y_n-y^\star)\ \weakconverge\
    \NM\!\left(0,\ \frac{1}{2\beta_0^{-1}-\beta_0^{-2}}
    [\nabla^2 f(x_o^\star)]^{-1}\Sigma_\xi[\nabla^2 f(x_o^\star)]^{-1}\right)$.
    Note that replacing $G(x,y)=y-x$ by $\beta_0(y-x)$ does not change the fixed point $y^\star$.
    This illustrates that the framework allows some flexibility in the choice (or scaling) of the operators, and that asymptotic analysis can be used to identify an optimal choice.
\end{exam}

\begin{exam}[SGD with momentum]\label{exmp:momentum}
    SGD with momentum is ubiquitous in machine learning. We consider a special case of quasi-hyperbolic momentum \citep{gitman2019understanding}
    \begin{equation}\label{eq:SHB}
		\begin{aligned}
			x_{n+1} &= x_{n} - \alpha_n (x_n - \nabla f(y_n) - \xi_n),\\
			y_{n+1} &= y_{n} - \beta_n  [\nu x_{n} + (1-\nu) (\nabla f(y_n) + \xi_n) ].
		\end{aligned}
	\end{equation}
    Here $\nu\in[0,1]$ is an interpolation parameter: when $\nu=1$, the method becomes a ``normalized'' version of stochastic heavy ball (SHB)~\citep{gadat2018stochastic,gupal1972stochastic}; when $\nu=0$, it reduces to SGD.
    This update is a special case of 
	\eqref{alg:xy} with $F(x, y) = x-\nabla f(y), G(x, y) = \nu x + (1-\nu) \nabla f(y)$, $H(y) = \nabla f(y)$, and $\psi_n = - (1-\nu) \xi_n$.
    Note that $G(x,y) = \nabla f(y) + \nu(x - \nabla f(y))$.
    It follows that $G(H(y), y) =  \nabla f(y)$,
	$y^{\star} = x_o^{\star}$ and $x^{\star} = 0$.
    One may interpret $x_n$ as a (stochastic) search direction combining the current stochastic gradient $\nabla f(y_n)+\xi_n$ and the previous direction $x_n$.
    The fast-time-scale update aims to track the gradient map $H(y)=\nabla f(y)$ at (approximately) fixed $y$, while the slow loop updates $y$ toward the stationary point (and minimizer) $x_o^\star$.
    If we assume $f$ is twice differentiable at $x_o^\star$ and $\nabla f$ is Lipschitz continuous, then Assumptions~\ref{assump:smooth} -- \ref{assump:hurwitz} holds with $\Fsmooth = \Gsmooth = \Hsmooth = 1$.
    One can check
	$B_1 = \mI$, $B_2 = \nu \mI$, $B_3 - \frac{\invdiffslow}{2} \mI = \nabla^2 f(x_o^\star) - \frac{\invdiffslow}{2} \mI$, $\Sigma_\psi = (1-\nu)^2 \Sigma_\xi$, $\Sigma_{\xi,\psi} = - (1-\nu) \Sigma_\psi$ and $\widetilde{\Sigma}_{\psi} =  \Sigma_{\xi}$.
    
    By Theorem~\ref{thm:fclt-twotime}, the solution to \eqref{eq:fclt-twotime:x-var} is $\Sigma_x=\Sigma_\xi/2$, and the solution to \eqref{eq:fclt-twotime:y-var} is $
    \Sigma_y=\int_0^\infty \exp\!\Big(-\big[\nabla^2 f(x_o^\star)-\invdiffslow \mI/2\big]t\Big)\,
    \Sigma_\xi\,
    \exp\!\Big(-\big[\nabla^2 f(x_o^\star)-\invdiffslow \mI/2\big]t\Big)\,dt,$
    which does not depend on $\nu$. This is consistent with Remark~\ref{rema:alg-design} by noting $G(x,y) = \nabla f(y) + \nu F(x,y)$.
    Thus, for a strongly convex objective, interpolating between SGD and normalized SHB does not change the asymptotic covariance.
    This observation is consistent with prior results showing that, under constant step sizes, momentum methods have the same leading-order asymptotic covariance as SGD \citep{gitman2019understanding}, and that momentum can even hurt convergence in certain regimes \citep{kidambi2018insufficiency,liu2021diffusion}.
    The benefits of momentum are instead associated with escaping saddle points in nonconvex optimization \citep{liu2021diffusion} and with faster transient (early-stage) convergence toward a neighborhood of the optimum in the strongly convex case \citep{tang2023acceleration}.
    Our result complements these works by ruling out asymptotic gains from momentum in the strongly convex setting with diminishing step sizes. Turning to transient performance, \citet{ma2019quasi} reported improved training performance of QHM across several deep-learning tasks, providing empirical evidence that different choices of $\nu$ in $G(x,y)=\nabla f(y)+\nu F(x,y)$ can improve transient behavior, although the constant-step-size setting considered there falls outside our assumptions.
    
    Moreover, as discussed in Remark~\ref{rema:simul-converge}, our FCLT implies that, with high probability, all iterates between $y_n$ and $y_{N^\beta(n,T)}$ remain within an $\OM(\sqrt{\beta_n})$-neighborhood of $x_o^\star$.
    This also complements the FCLT in \cite{gadat2017stochastic}, which is proved for quadratic objectives and under a constant step-size ratio $\beta_n/\alpha_n$.

    Returning to \eqref{eq:SHB}, it is arguably more natural to update the slow iterate using the fresh direction $x_{n+1}$ as $
    y_{n+1}=y_n-\beta_n\big[\nu x_{n+1}+(1-\nu)\big(\nabla f(y_n)+\xi_n\big)\big]$.
    This can be rewritten in the form of \eqref{eq:SHB} by allowing the interpolation parameter to vary with $n$, as in \cite{gitman2019understanding}. Indeed,
    $y_{n+1}
    = y_n-\beta_n\big[\nu_n x_n+(1-\nu_n)\big(\nabla f(y_n)+\xi_n\big)\big]$ with $\nu_n=\nu(1-\alpha_n)$,
    so the induced operator $G_n(x,y)=\nu_n x+(1-\nu_n)\nabla f(y)$ becomes iteration-dependent.
    However, the composite operator $G(H(\cdot),\cdot)$ remains $\nabla f(\cdot)$, which is the object that drives the slow-time-scale limit.
    Our analysis extends to this case with only minor modifications.
    Finally, as shown in \cite[Appendix~A]{gitman2019understanding}, setting $\alpha_n\equiv \nu$ recovers a stochastic variant of Nesterov's accelerated gradient (NAG).
    Although a constant step size violates Assumption~\ref{assump:stepsize-twotime}, we expect analogous diffusion limits to hold, leading to a similar conclusion: the advantage of momentum/acceleration is largely offset by stochastic noise asymptotically, so variance reduction becomes the dominant concern near the optimum.
\end{exam}

\begin{exam}[Gradient-based temporal-difference learning]\label{exmp:gtd}
	Gradient-based temporal-difference (TD) algorithms are widely used in off-policy reinforcement learning (RL), such as policy evaluation with function approximation~\citep{sutton2008convergent, sutton2009fast, maei2009convergent}.
	We consider two specific algorithms: GTD2 and TDC with linear function approximation~\citep{sutton2009fast}. These algorithms serve as typical examples of two-time-scale SA~\citep{dalal2018finite}. For clarity, we concentrate on the specific update rules while omitting the details of the RL setup. A more comprehensive explanation is provided in Appendix~\ref{append:exmp}.

	In this problem, the goal is to estimate the value function $V(s)$ for a given policy using a linear parameterization $V(s; y) = \phi(s)^\top y$, where $\phi(s)$ represents the feature vector of state $s$, and $y$ denotes the parameter vector.
	Given i.i.d.\ samples $\{ (s_n, r_n, s'_n) \}_{n=0}^\infty$~\citep{sutton2008convergent}, where $s_n$ is the current state, $r_n$ is the reward and $s'_n$ is the next state, the optimal parameter $y$ can be estimated iteratively via the following update rules
	\begin{align}
		x_{n+1} &= x_{n} - \alpha_{n} \big( \phi(s_{n})^{\top} x_{n} - \rho_{n} \delta_{n} \big) \phi(s_{n}), \label{eq:gtd-x}\\
		y_{n+1} &= y_{n} - \beta_{n} \rho_{n} \big( \gamma \phi(s'_{n}) - \phi(s_{n}) \big) \phi(s_{n})^\top x_{n}, \label{eq:gtd2-y}\  \qquad \mathrm{ (GTD2) }   \\
		y_{n+1} &= y_{n} - \beta_{n} \rho_{n} \big( \gamma \phi(s'_n) \phi(s_n)^\top x_{n} - \delta_{n} \phi(s_{n}) \big), \label{eq:tdc-y} \ \quad \,  \mathrm{ (TDC) }
	\end{align}
	where $\gamma$ is the discount factor, $\rho_{n}$ is the importance weight with its explicit form deferred to Appendix~\ref{append:exmp}, and $\delta_{n}$ is the TD error defined as $	\delta_{n} = r_{n} + ( \gamma \phi(s'_{n}) - \phi(s_n) )^\top y_{n} $.
	Here, the fast iterate $x_{n}$ is an auxiliary variable.
	Both GTD2 and TDC share the same fast time-scale update but differ in their slow time-scale updates.
	
	To contextualize these methods within the two-time-scale SA~\eqref{alg:xy}, we define $A_n =  \rho_{n} \phi(s_n) ( \phi(s_n) - \gamma \phi(s'_n) )^\top $, $b_n = \rho_{n} r_{n} \phi(s_n) $, $C_n = \phi(s_{n}) \phi(s_n)^\top $, $D_n = \gamma \rho_{n} \phi(s_n) \phi(s'_n)^\top$, and their expectations $A= \EB[A_n]$, $b = \EB[b_n]$, $C=\EB[C_n]$, and $D=\EB[D_n]$.
	It can be checked that \eqref{eq:gtd-x} -- \eqref{eq:tdc-y} are an instance of \eqref{alg:xy} with
	\begin{equation*}
		F(x,y) = Cx +Ay-b,\ \ 
		\Ga(x,y) = -A^\top x,\ \ 
		\Gb (x,y) = D^\top x + Ay - b.
	\end{equation*}
	It follows that $H(y) = -C^{-1}(Ay-b)$, implying
	$\Ga(H(y), y) = A^\top C^{-1} (Ay- b)$ and $\Gb(H(y), y) = (C - D^\top) C^{-1} (Ay - b)$.
	Although these expressions appear different, it can be verified that $C - D^\top = A^\top$ \citep[Equation~(68)]{li2024high}, leading to $\Ga (H(y), y) \equiv \Gb(H(y),y)$.
	Therefore, GTD2 and TDC have the same optimal solution, $y^\star = A^{-1} b$ and $x^\star = H(y^\star) = 0$. 
    Since $F$, $G$ and $H$ are linear operators, Assumptions~\ref{assump:smooth} -- \ref{assump:near-linear} holds with $\Fsmooth = \Gsmooth = \Hsmooth = 1$.
    For GTD2 and TDC, the matrices $B_1 = C$ and $B_3 = A^\top C^{-1} A$ coincide,
    whereas $B_2$ differs: $B_2^{\mathrm{GTD2}}=-A^\top$ and $B_2^{\mathrm{TDC}}=D^\top$.
    If the features $\phi(s)$ are chosen appropriately, Assumption~\ref{assump:hurwitz} also holds; see \cite{dalal2018finite}.
	
	For the noise terms, both the methods share the same noise for the fast time scale $\xi_{n} = C_n x_{n} + A_n y_n - b_n - F(x_n, y_n) $.
	On the slow time scale, however, the noise differs slightly. For GTD2, the noise is expressed as $\psia_{n} = - A_{n}^\top x_n - \Ga(x_n, y_n) $, whereas for TDC, it is given by $\psib_{n} = D_n^\top x_{n} + A_n y_n - b_n - \Gb(x_n, y_n)$.
	Despite these differences, one can check $\psia_{n} - B_2^{\mathrm{GTD2}} B_1^{-1} \xi_{n} = \psib_{n} - B_2^{\mathrm{TDC}} B_1^{-1} \xi_{n} + \oprob(1)$, leading to the same $\widetilde{\Sigma}_{\psi}$ in \eqref{eq:cov-z}. Calculation details can be found in Appendix~\ref{append:exmp}.
    Thus, although GTD2 and TDC differ in form and may display distinct numerical behaviors~\citep{dann2014policy, ghiassian2020gradient}, their asymptotic dynamics are identically characterized by \eqref{eq:fclt-twotime:y-sde}, aligning with the results of~\citep{hu2024central}.
    Compared to their work, our results provide a more detailed characterization of trajectory properties for i.i.d.\ samples.
    As discussed in Remark~\ref{rema:simul-converge}, we can strengthen the convergence statement to a trajectory level: with high probability, all iterates between $y_n$ and $y_{N^\beta(n,T)}$ remain within an $\OM(\sqrt{\beta_n})$-neighborhood of $y^\star$.

    Similar to Remark~\ref{rema:alg-design} and Example~\ref{exmp:momentum}, different choices of $G$ can induce the same composite operator $G(H(\cdot),\cdot)$.
    Indeed, using the identity $\Gb(x,y)=\Ga(x,y)+F(x,y)$, we may consider the one-parameter family
    $G_\nu(x,y)=\Ga(x,y)+\nu F(x,y)$ with $\nu\in[0,1]$, which interpolates between GTD2 and TDC.
    This highlights the flexibility in algorithm design afforded by the framework~\eqref{eq:two-loop-approx}.
    Although all these algorithms share the same slow-scale asymptotic dynamics under our theorem, their non-asymptotic behavior may differ.
    In particular, TDC has been reported to outperform or match GTD2 in empirical comparisons, providing a concrete illustration that the modification in Remark~\ref{rema:alg-design} may improve transient performance without changing the asymptotic limit \citep{ghiassian2020gradient}.
\end{exam}

\section{Framework of Proof}
\label{sec:fclt:proof-frame}

In this section, we outline the framework of the proof, which is divided into four key steps: (i) deriving the one-step recursion; (ii) establishing the tightness of the sequences constructed in Section~\ref{sec:fclt:construct}; (iii) demonstrating that these sequences approximately solve the corresponding martingale problems; and (iv) combining the results from the previous steps to derive the decoupled FCLT.
The main content for these steps is detailed in Sections~\ref{sec:fclt:one-step} -- \ref{sec:fclt:intergrate}. 

Notably, in Section~\ref{sec:fclt:one-step}, we introduce an auxiliary sequence $\{ \zcheck_{n} \}_{n=1}^\infty$ and construct the associated continuous-time trajectories $\{ \zbar_{n}(\cdot) \}_{n=1}^\infty$. The introduction of this auxiliary sequence represents one of the primary novelties of our work and serves as a cornerstone of the proof. \Figref{fig:proof-sketch} offers a visual summary of the proof framework. The blue, green, yellow, and white blocks correspond to the four steps outlined above, providing a clear and intuitive depiction of the overall structure.

\begin{figure}[t]\centering
	\begin{tikzpicture}[>=stealth,every node/.style={shape=rectangle,draw,rounded corners, minimum width=1cm, node distance=0.8cm},]
				
		\node[fill=blue!20] (xcheck) at (-4.5,2) {\small \begin{tabular}{c}
				Recursion of $\xcheck_{n}$ \\
				(Lemma~\ref{lem:check-twotime-xy})
			\end{tabular} };
		\node[fill=blue!20] (zcheck) at (-4.5,0) 
		{\small \begin{tabular}{c}
				Recursion of $\zcheck_{n}$ \\
				(Lemma~\ref{lem:check-twotime-z})
			\end{tabular} };
		\node[fill=blue!20] (ycheck) at (-4.5,-2)
		{\small \begin{tabular}{c}
				Recursion of $\ycheck_{n}$ \\
				(Lemma~\ref{lem:check-twotime-xy})
			\end{tabular} };
		
		\draw[->, line width=.3mm] (xcheck) to (zcheck);
		\draw[->, line width=.3mm] (ycheck) to (zcheck);
		
		\node[fill=yellow!20] (martprob) at (0,0)
		{\small \begin{tabular}{c}
				Approximately solving \\ 
				of martingale problems \\
				(Lemma~\ref{lem:mart-prob-twotime})
		\end{tabular} };
		\node[fill=green!20] (xztight) at (0,2)
		{\small \begin{tabular}{c}
				Tightness of $\{ \xbar_{n}(\cdot) \}_{n=1}^\infty$ \\
				 and $\{ \zbar_{n}(\cdot) \}_{n=1}^\infty$ \\
				(Lemma~\ref{lem:tight-twotime})
			\end{tabular}};
		
		\node[fill=green!20] (ytight) at (0,-2)
		{\small \begin{tabular}{c}
				Tightness of $\{ \ybar_{n}(\cdot) \}_{n=1}^\infty$ \\
				(Lemma~\ref{lem:tight-twotime})
			\end{tabular} };
		
		\draw[->, line width=.3mm] (xcheck) to  (xztight);
		\draw[->, line width=.3mm] (zcheck) to  (xztight);
		\draw[->, line width=.3mm] (xcheck) to (martprob);
		\draw[->, line width=.3mm] (zcheck) to (martprob);
		\draw[->, line width=.3mm] (ycheck) to (ytight);
		
		\node (zfclt) at (4.5, 0)
		{\small \begin{tabular}{c}
				FCLT for $\{ \zbar_{n}(\cdot) \}_{n=1}^\infty$ \\
				(Corollary~\ref{cor:fclt-z})
		\end{tabular}  };
		\node (xfclt) at (4.5,2)
		{\small \begin{tabular}{c}
				FCLT for $\{ \xbar_{n}(\cdot) \}_{n=1}^\infty$ \\
				(Theorem~\ref{thm:fclt-twotime}~\ref{thm:fclt-x})
			\end{tabular}  };
		\node (yfclt) at (4.5, -2)
		{\small \begin{tabular}{c}
				 FCLT for $\{ \ybar_{n}(\cdot) \}_{n=1}^\infty$ \\
				 (Theorem~\ref{thm:fclt-twotime}~\ref{thm:fclt-y})
			\end{tabular}};
		
		\draw[->, line width=.3mm] (xztight) to (xfclt);
		\draw[->, line width=.3mm] (xztight) to (zfclt);
		\draw[->, line width=.3mm] (martprob) to (xfclt);
		\draw[->, line width=.3mm] (martprob) to (zfclt);
		\draw[->, line width=.3mm] (ytight) to (yfclt);
		\draw[->, line width=.3mm] (zfclt) to (yfclt);
	\end{tikzpicture}
	\caption{Illustration for the framework of proof.}
	\label{fig:proof-sketch}
\end{figure}
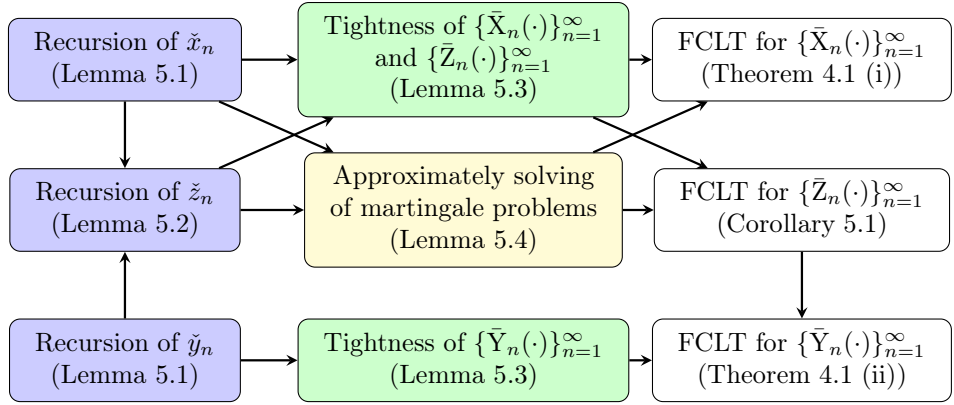

\subsection{Derivation of One-Step Recursions and Introduction of the Auxiliary Sequence}
\label{sec:fclt:one-step}

With the rescaled sequence defined in \eqref{eq:xycheck-def}, we derive their one-step recursion formulas as follows.

\begin{lem}[One-step recursions for $\xcheck_{n}$ and $\ycheck_{n}$]\label{lem:check-twotime-xy}
	Suppose that Assumptions~\ref{assump:smooth} -- \ref{assump:stepsize-twotime} and \ref{assump:mart-decouple-rate} hold.
	Define $\tilde{\xi}_{n} := \xi_{n} - \kappa_{n} H^\star \psi_{n}$ with  $H^\star := \nabla H(\yy^\star)$. 
	The rescaled sequences $\{ \xcheck_{n} \}_{n=0}^\infty$ and $\{ \ycheck_{n} \}_{n=0}^\infty$ defined in \eqref{eq:xycheck-def} 
	satisfy the following relationships  
	\begin{align}
		\xcheck_{n+1} & = (\mI - \alpha_{n} B_1) \xcheck_{n} - \sqrt{\alpha_{n}} \tilde{\xi}_{n} + R_{n}^x, \label{eq:xcheck-twotime} \\
		\ycheck_{n+1}
		& = \left( \mI - \beta_{n} B_3 + \frac{ \invdiffslow \beta_{n} }{2} \mI \right) \ycheck_{n} - \sqrt{\beta_{n} \alpha_{n} } B_2 \xcheck_{n} - \sqrt{\beta_n} \psi_{n} + R_{n}^y. \label{eq:ycheck-twotime}
	\end{align}
	For any $2 < p \le \frac{4}{1 + ( \Hsmooth \vee \Fsmooth \vee \Gsmooth )/2} $, 
	the residual terms $R_{n}^x$ and $R_{n}^y$ 
	have the following properties:
		(i) $\EB \| \EB [ R_{n}^x | \FM_n ] \| = o(\sqrt{\beta_{n} \alpha_{n}}) $, $\EB \| R_{n}^x \|^2 = \OM(\beta_{n} \alpha_{n})$ and $\EB \| R_{n}^x \|^p = o(\alpha_{n})$;
		(ii) $\EB \| R_{n}^y \|^2 = o(\beta_{n}^2)$ and $\EB \| R_{n}^y \|^p = o(\beta_{n})$.
\end{lem}

Under the local linearity assumption, $\xcheck_{n+1}$ can be decomposed into three components: a linear transformation of the previous iterate $\xcheck_{n}$,
a noise term $\sqrt{\alpha_{n}} \tilde{\xi}_{n}$ and a high-order residual term $R_n^x$.
The influence of $\ycheck_{n}$ on \eqref{eq:xcheck-twotime} is negligible for two reasons: (i) the noise term is adjusted to $\tilde{\xi}_{n}$, where $\tilde{\xi}_{n} = \xi_{n} + o_p(1)$; (ii) the residual term $R_n^x$ is asymptotically negligible.
As a result, the behavior of the sequence $\{ 
\xcheck_{n} \}_{n=0}^\infty$ closely resembles that of standard SA.
Ignoring residual terms, recursively applying \eqref{eq:xcheck-twotime} shows that each $\xcheck_{n}$ is a weighted sum of noise terms and an initialization term.
Under Assumption~\ref{assump:hurwitz}, $\mI - \alpha_{n} B_1$ is a contraction operator, ensuring that the influence of initialization term diminishes asymptotically.

For \eqref{eq:ycheck-twotime}, however, there is an additional term, $\sqrt{ \beta_{n} \alpha_{n} } B_2 \xcheck_{n}$, reflecting the influence of $\xcheck_{n}$ on $\ycheck_{n}$.
This term is of order $\OM_p( \sqrt{\beta_{n} \alpha_{n}} )$, which dominates the term $ \beta_{n} (B_3 - \invdiffslow \mI / 2) \ycheck_{n} $ (of order $\OM_p(\beta_{n}) $)
because $\beta_{n} / \alpha_{n} = o(1)$.
This dominance complicates the analysis for the slow-time-scale.

To address this, we introduce the auxiliary sequence
\begin{equation}\label{eq:zcheck-def}
	\zcheck_{n} = \ycheck_{n} - \sqrt{ \kappa_{n-1} } B_2 B_1^{-1} \xcheck_{n}
	= \frac{ \yhat_{n} - \kappa_{n-1} B_2 B_1^{-1} \xhat_{n} }{ \sqrt{\beta_{n-1}} }, \ n \ge 1,
\end{equation}
which subtracts the correction term $\sqrt{\kappa_{n-1}} B_2 B_1^{-1} \xcheck_{n}$ from $\ycheck_{n}$.
Since $\kappa_{n} = o(1)$, it holds that $\EB \| \zcheck_{n} \|^2 = \OM(1)$, and for sufficiently large $n$, $\zcheck_{n} = \ycheck_{n} + \oprob(1)$.
The following lemma establishes that the one-step recursion for $\zcheck_{n}$ takes a form similar to \eqref{eq:xcheck-twotime}.

\begin{lem}[One-step recursion for $\zcheck_{n}$]\label{lem:check-twotime-z}
	Suppose that Assumptions~\ref{assump:smooth} -- \ref{assump:stepsize-twotime} and \ref{assump:mart-decouple-rate} hold.
	Define 
	$\tilde{\psi}_{n} := \psi_{n} - B_2 B_1^{-1} \tilde{\xi}_n$ with $\tilde{\xi}_{n}$ defined in Lemma~\ref{lem:check-twotime-xy}. The auxiliary sequences $\{ \zcheck_{n} \}_{n=1}^\infty$ defined in \eqref{eq:zcheck-def} satisfy the following relationships  
	\begin{equation}
		\zcheck_{n+1} = \left( \mI - \beta_{n} B_3 + \frac{\invdiffslow \beta_{n}}{2} \mI \right) \zcheck_{n} - \sqrt{\beta_{n}} \tilde{\psi}_{n} + R_{n}^z. \label{eq:zcheck-twotime}
	\end{equation}
	For any $2 < p \le \frac{4}{1 + ( \Hsmooth \vee \Fsmooth \vee \Gsmooth )/2} $, 
	the residual terms $R_{n}^z$ have the following properties:
		$\EB \| \EB [ R_{n}^z | \FM_n ] \| = o(\beta_{n} ) $, $\EB \| R_{n}^z \|^2 = \OM(\beta_{n}^2 )$ and $\EB \| R_{n}^z \|^p = o(\beta_{n})$.
\end{lem}

Comparing \eqref{eq:ycheck-twotime} and \eqref{eq:zcheck-twotime}, the introduction of  $\zcheck_{n}$ effectively eliminates the primary impact of the fast-time-scale update on the slow-time-scale update. While prior works~\citep{konda2004convergence, kaledin2020finite} have addressed mitigating the influence of slow-time-scale updates on fast-time-scale dynamics, the reverse direction has largely been unexplored. Thus, the introduction of $\zcheck_{n}$  is both innovative and has the potential to inspire future research on two-time-scale stochastic approximation.

To establish Lemmas~\ref{lem:check-twotime-xy} and \ref{lem:check-twotime-z}, Condition~\ref{assump:stepsize-twotime-moreab} in Assumption~\ref{assump:stepsize-twotime} is critical for deriving upper bounds for the residual term $R_{n}^x$, $R_{n}^y$, and $R_{n}^z$.
These bounds play a key role in ensuring that the nonlinear components are negligible compared to the linear terms in the subsequent analysis. Notably, obtaining these bounds requires detailed and intricate calculations. The full proofs of Lemmas~\ref{lem:check-twotime-xy} and \ref{lem:check-twotime-z} are provided in Appendix~\ref{sec:fclt:proof:construct}.

For the subsequent analysis, we will utilize the corrected sequence $\{ \zcheck_{n} \}_{n=1}^\infty$ as a foundational step toward establishing the FCLT for the slow time scale. Analogous to the construction of   $\xbar_n(\cdot)$ and $\ybar_{n}(\cdot)$, we construct $\zbar_{n}(\cdot)$ as follows
\begin{align}
	\label{eq:zbar-short}
	\zbar_{n}(t) & = \begin{cases}
		\zcheck_{n}, &  t = 0; \\
		\zcheck_{m} + \frac{t - \sum_{i=n}^{m-1} \beta_{i} }{\beta_{m}} (\zcheck_{m+1} - \zcheck_{m}), & t \in \left( \sum_{i=n}^{m-1} \beta_{i}, \sum_{i=n}^{m} \beta_{i} \right] \text{ for } m \ge n.
	\end{cases} 
\end{align}

\subsection{Establishment of Tightness}
\label{sec:fclt:tight}

In this subsection, we aim to establish the tightness of the trajectory sequences $\{ 
\xbar_{n}(\cdot) \}_{n=1}^\infty$, $\{ 
\ybar_{n}(\cdot) \}_{n=1}^\infty$, and $\{ \zbar_{n} (\cdot) \}_{n=1}^\infty$.
From our construction, $\xbar_{n}(\cdot)$ is a random function in $C([0,\infty); \RB^{d_x})$ while $\ybar_{n}(\cdot)$, $\zbar_{n}(\cdot)$ are random functions in $C([0, \infty); \RB^{d_y})$.
To establish tightness, we apply the following sufficient conditions,
which is a corollary of \cite[Theorem~7.3]{billingsley2013convergence} and \cite[Corollary~5]{whitt1970weak}.

\begin{prop}
	[Sufficient conditions for tightness]
	\label{prop:tight-suff-slow}
	Suppose that $\{ \uu_{n}(\cdot) \}_{n=0}^\infty$ is a sequence of random functions in $C([0, \infty); \RB^{d})$.
	Then $\{ \uu_{n}(\cdot) \}_{n=0}^\infty$ is tight 
	if the following two conditions hold:
	\begin{enumerate}[(i)]
		\item For each positive $\eta$, there exists a positive number $a$ and a positive integer $n_0$ such that 
		\begin{equation*}
			\PB \left( \| \uu_{n}(0) \| \ge a \right) \le \eta, \quad \forall\, n \ge n_0.
		\end{equation*}
		\item For each 
		positive $T$, there exist positive numbers $\gamma_1, \gamma_2, c_1, c_2, K, n_0$ such that
		\begin{equation}\label{eq:tight-suff-slow}
			\PB \left(  \| \uu_{n}(s) {-} \uu_{n}(t) \| \ge \eps  \right) \le K \left( \frac{|s{-}t |^{1 {+} \gamma_1} }{\eps^{c_1} } + \frac{ | s{-}t |^{1 {+} \gamma_2}}{\eps^{c_2}} \right), \  \forall\, s,t \in [0, T], \forall \, n \ge n_0,
			\forall\, \eps > 0.
		\end{equation}
	\end{enumerate}
\end{prop} 

In Proposition~\ref{prop:tight-suff-slow}, the first condition ensures the tightness of the initial distributions, while the second condition guarantees a form of uniform equicontinuity for $\{ \uu_{n} (\cdot) \}_{n=0}^\infty$ over compact time intervals.
We take two different values of 
$\gamma$ and $c$ because the deterministic and stochastic components in the one-step recursions are controlled by different arguments.

\begin{lem}[Tightness]\label{lem:tight-twotime}
	Suppose that Assumptions~\ref{assump:smooth} -- \ref{assump:stepsize-twotime}, \ref{assump:noise-fclt}, and \ref{assump:mart-decouple-rate} hold.
	The sequences $\{ \xbar_{n}(\cdot) \}_{n=1}^\infty$, $\{ \ybar_{n}(\cdot) \}_{n=1}^\infty$,  
	and $\{ \zbar_{n}(\cdot) \}_{n=1}^\infty$ constructed in \eqref{eq:xbar-short}, \eqref{eq:ybar-short} and \eqref{eq:zbar-short}
	are tight.
\end{lem}

By Proposition~\ref{prop:tight}, once tightness is established, proving the weak convergence of the random functions reduces to examining their finite-dimensional distributions.
From the definition of $\zcheck_{n}$ in \eqref{eq:zcheck-def} and the construction of $\ybar_{n}(\cdot)$ and $\zbar_{n}(\cdot)$ in \eqref{eq:ybar-short} and \eqref{eq:zbar-short}, it is straightforward to verify that the finite-dimensional distributions of $\ybar_{n}(\cdot)$ have the same weak convergence limits as those of $\zbar_{n}(\cdot)$.
Consequently, determining the limit of $\ybar_{n}(\cdot)$ now reduces to analyzing $\zbar_{n}(\cdot)$.

While Lemma~\ref{lem:tight-twotime} is highly powerful, establishing it is challenging, particularly for the tightness of $\{ \ybar_{n}(\cdot) \}_{n=1}^\infty$. This difficulty arises from the presence of the $\sqrt{ \beta_{n} \alpha_{n} } B_2 \xcheck_{n}$ term in \eqref{eq:ycheck-twotime}. The detailed proof of this lemma can be found in Appendix~\ref{sec:fclt:proof:tight}.

\subsection{Approximately Solving of Martingale Problems}\label{sec:fclt:mart-prob}

In this subsection, we demonstrate that 
the sequence 
$\{ \xbar_{n}(\cdot) \}_{n=1}^\infty$
\textit{approximately solves} the martingale problem for $(\AM^x,C^\infty_c (\RB^{d_x}) )$, where
$\AM^x$ is the generator defined as
\begin{equation}\label{eq:generator-twotime:Lx}
	\AM^x f(\xx) = \inner{- B_1 \xx }{ \nabla f(\xx) } + \frac{1}{2} \tr( \nabla^2 f(\xx)\, \Sigma_{\xi} ), \ \forall f \in C^\infty_c (\RB^{d_x}).
\end{equation}
Meanwhile, the sequence $\{ \zbar_{n}(\cdot) \}_{n=1}^\infty$
\textit{approximately solves} the martingale problem for $(\AM^y, C^\infty_c (\RB^{d_y}) )$, where
    \begin{equation}\label{eq:generator-twotime:Ly}
	\AM^y g(\yy) = \inner{- \Big( B_3 - \frac{\invdiffslow \mI}{2} \Big) \yy }{ \nabla g(\yy) } + \frac{1}{2} \tr( \nabla^2 g(\yy)\, \tilde{\Sigma}_{\psi} ), \ \forall g \in C^\infty_c (\RB^{d_y}) .
\end{equation}
Here the matrix $\widetilde{\Sigma}_{\psi}$ is defined in \eqref{eq:cov-z}.
Clearly, $\AM_x$ and $\AM_y$ are the generators associated to the SDEs in \eqref{eq:fclt-twotime:x-sde} and \eqref{eq:fclt-twotime:y-sde}, respectively.
The formal statement of this approximately solving of the martingale problem is as follows.

\begin{lem}[Approximately solving of martingale problems]\label{lem:mart-prob-twotime}
Suppose that Assumptions~\ref{assump:smooth} -- \ref{assump:stepsize-twotime}, \ref{assump:noise-fclt}, and \ref{assump:mart-decouple-rate} hold.
For each $n$, define the following filtrations
\begin{align}\label{eq:filtration-main}
	\begin{split}
	\FM^\xbar_{n,t} & = \begin{cases}
		\FM_{n}, &  t = 0; \\
		\FM_{m+1}, & t \in \left( \sum_{i=n}^{m-1} \alpha_{i}, \sum_{i=n}^{m} \alpha_{i} \right] \text{ for } m \ge n;
	\end{cases} \\
	\FM^\zbar_{n,t} & = \begin{cases}
		\FM_{n}, &  t = 0; \\
		\FM_{m+1}, & t \in \left( \sum_{i=n}^{m-1} \beta_{i}, \sum_{i=n}^{m} \beta_{i} \right] \text{ for } m \ge n.
	\end{cases} 
	\end{split}
\end{align}

\begin{enumerate}[(i)]
    \item  For any $f \in C^\infty_c(\RB^{d_x})$ and $t \ge 0$, with $\AM^x$ defined in \eqref{eq:generator-twotime:Lx}, we have the following decomposition
    \begin{equation*}
    f(\xbar_{n}(t)) - f(\xbar_{n}(0)) - \int_0^{t} \AM^x f(\xbar_{n}(s))\, \d s
    = \MM_{n}^f(t) + \RM_{n}^f(t), 
    \end{equation*}
    where for each $n$, $\MM_{n}^f(t)$ is a martingale w.r.t. $(\FM^\xbar_{n,t} )_{t \ge 0}$ and for each $t$, $\EB | \RM_{n}^f(t) | \to 0$ as $n \to \infty$. 
    \item For any $g \in C^\infty_c (\RB^{d_y})$ and $t \ge 0$, with $\AM^y$ defined in \eqref{eq:generator-twotime:Ly}, we have the following decomposition
    \begin{equation*}
    g(\zbar_{n}(t)) - g(\zbar_{n}(0)) - \int_0^{t} \AM^y f(\zbar_{n}(s))\, \d s
    = \MM_{n}^g(t) + \RM_{n}^g(t), 
    \end{equation*}
    where for each $n$, $\MM_{n}^g(t)$ is a martingale w.r.t. $( \FM^\zbar_{n,t})_{t \ge 0}$ and for each $t$, $\EB | \RM_{n}^g(t) | \to 0$ as $n \to \infty$.
\end{enumerate}
\end{lem}
The definition of filtrations in \eqref{eq:filtration-main} ensures that  each $\xbar_{n}(t)$ is adapted to $(\FM^\xbar_{n.t})_{t \ge 0}$ and each $\zbar_{n}(t)$ is adapted to $(\FM^\zbar_{n,t})_{t \ge 0}$.
The term \textit{approximately solving} indicates that the residual terms $\RM_{n}^f(t)$ and $\RM_{n}^g(t)$ converging to $0$ in the $L_1$ sense as $n \to \infty$.
The derivation of this lemma heavily relies on the one-step recursions presented in Lemmas~\ref{lem:check-twotime-xy} and \ref{lem:check-twotime-z}. The detailed proof of Lemma~\ref{lem:mart-prob-twotime} is provided in Appendix~\ref{sec:proof:fclt:mart-prob}.

\subsection{Integration of Pieces}
\label{sec:fclt:intergrate}

In this subsection, we combine the results in the previous subsections to derive the FCLTs in Theorem~\ref{thm:fclt-twotime}.
The combination is based on the following proposition.

\begin{prop}[{\citealp[Chapter~4, Theorem~8.10]{ethier2009markov}}]\label{prop:mart-prob-approx}
	Suppose that the generator $\AM$ is defined in \eqref{eq:generator-general}, $U(\cdot)$ is the solution to the martingale problem for $(\AM, C_c^\infty(\RB^d)))$, and the following conditions are met. 
	\begin{enumerate}[(i)]
		\item \label{prop:final:condition:unique} Given the initial distribution, the martingale problem for $(\AM, C_c^\infty(\RB^d) )$ has a unique solution.
		\item \label{prop:final:condition:tight} $\{ \uu_{n}(\cdot) \}_{n=0}^\infty$ is a tight sequence in $C([0, \infty); \RB^d)$.
		\item \label{prop:final:condition:initial}  $\uu_{n}(0) \weakconverge \uu(0)$.
		\item \label{prop:final:condition:mart-prob} For any $k \ge 1$, $0 \le t_1 < t_2 < \dots < t_k \le t \le t+s $, $f \in C^\infty_c (\RB^d)$ and $h_1, h_2, \dots, h_k \in C_c (\RB^d)$, 
		\begin{equation}\label{eq:generator-condition}
			\lim_{n \to \infty} \EB  \bigg[ \Big( f(\uu_{n}(t+s)) - f(\uu_{n}(t)) - \int_t^{t+s} \AM f(\uu_{n}(r)) \,\d r \Big) \prod_{i=1}^k h_i(\uu_{n}(t_i)) \bigg]  = 0.
		\end{equation}
	Then we have $\uu_{n}(\cdot) \weakconverge \uu(\cdot)$.
	\end{enumerate}
\end{prop}
In this proposition, the choice of $h_i \in C_c(\RB^d)$ follows from \cite[Chapter~3, Proposition~4.4]{ethier2009markov}.

\begin{figure}[t]\centering
	\begin{tikzpicture}[>=stealth,every node/.style={shape=rectangle,draw,rounded corners, minimum width=1cm, node distance=0.8cm},]
		
		\node (xtight) at (0,1.5) 		{\small \begin{tabular}{c}
				Tightness of $\{ \xbar_{n}(\cdot) \}_{n=1}^\infty$ \\
				(Lemma~\ref{lem:tight-twotime})
		\end{tabular}};
		\node (martprob) at (0,-1.6)
		{\small \begin{tabular}{c}
				Approximately solving of \\
				martingale problems \\
				(Lemma~\ref{lem:mart-prob-twotime})
		\end{tabular} };
		
		\node (subseq) at (0,0)
		{\small \begin{tabular}{c}
				FCLT for convergent \\ subsequence 
				$\{ \xbar_{n_k}(\cdot) \}_{k=1}^\infty$
		\end{tabular} };
	
		\draw[->, line width=.3mm, color=cyan] (xtight) to (subseq);
		\draw[->, line width=.3mm, color=cyan] (martprob) to (subseq);

		\node (xinitial) at (3.8,0)
		{\small \begin{tabular}{c}
				$\xbar(0) \weakconverge \pi^\star$
		\end{tabular}};
		
		\node (xfclt) at (7.6,0)
		{\small \begin{tabular}{c}
				FCLT for $\{ \xbar_{n}(\cdot) \}_{n=1}^\infty$ \\
				(Theorem~\ref{thm:fclt-twotime}~\ref{thm:fclt-x})
		\end{tabular} };
		
		\draw[->, line width=.3mm, color=red] (subseq) to  (xinitial);
		\draw[->, line width=.3mm, color=red] (xtight) to  (xinitial);
		\draw[->, line width=.3mm] (xinitial) to  (xfclt);
		\draw[->, line width=.3mm] (xtight) to (xfclt);
		\draw[->, line width=.3mm] (martprob) to (xfclt);
		
	\end{tikzpicture}
	\caption{Illustration for the last step of proof.}
	\label{fig:proof-intergrate}
\end{figure}
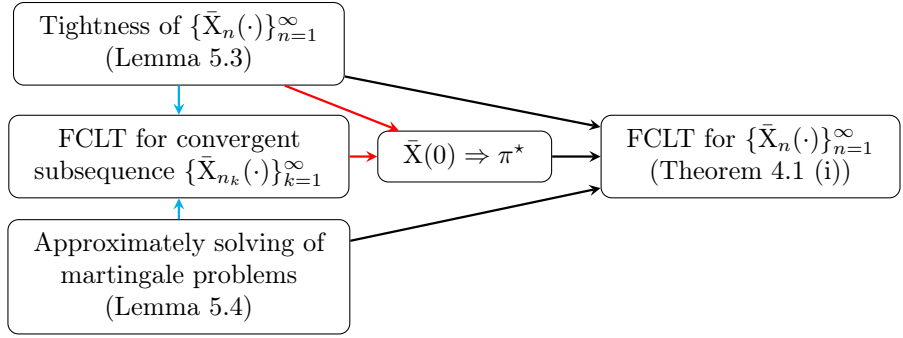

We now briefly illustrate how Proposition~\ref{prop:mart-prob-approx} is applied with $\AM = \AM_x$ and $\uu_{n}(\cdot) = \xbar_{n}(\cdot)$ to derive the FCLT for $\xbar_{n}(\cdot)$ in Theorem~\ref{thm:fclt-twotime}. The process can be divided into three parts, corresponding to the cyan, red, and black arrows in \Figref{fig:proof-intergrate}.

\begin{enumerate}[(i)]
	\item \textbf{FCLT for convergent subsequences}:
	Conditions~\ref{prop:final:condition:unique}, \ref{prop:final:condition:tight}, and \ref{prop:final:condition:mart-prob} are verified in Section~\ref{sec:fclt:prelim}, Lemmas~\ref{lem:tight-twotime} and \ref{lem:mart-prob-twotime}, respectively. Tightness ensures the existence of a convergent subsequence, and applying Proposition~\ref{prop:mart-prob-approx} shows that the subsequence converges to a solution of \eqref{eq:fclt-twotime:x-sde}.

	\item \textbf{Convergence of initial distributions}: 
	Assumption~\ref{assump:hurwitz} guarantees both the uniqueness of the invariant distribution $\pi^\star$ and the geometric ergodicity for \eqref{eq:fclt-twotime:x-sde} (see Proposition~\ref{prop:ergo}). 
     These properties ensure that all convergent subsequences of $\xbar_{n}(0)$ have the same limit $\pi^\star$, which implies the convergence of $\xbar_{n}(0)$.
		
	\item \textbf{FCLT for the entire sequence}: Applying Proposition~\ref{prop:mart-prob-approx} to the full sequence establishes the desired FCLT.
\end{enumerate}
Similarly, we could obtain the following FCLT for $\zbar_{n}(\cdot)$.
\begin{cor}[FCLT for the auxiliary sequence]\label{cor:fclt-z}
	Under the conditions of Theorem~\ref{thm:fclt-twotime},
	the stochastic process $\zbar_{n}(\cdot)$ defined in \eqref{eq:zbar-short} converges weakly to the stationary solution of the SDE in \eqref{eq:fclt-twotime:y-sde}.
\end{cor}

From the construction of $\ybar_{n}(\cdot)$ and $\zbar_{n}(\cdot)$ in \eqref{eq:ybar-short} and \eqref{eq:zbar-short}, it follows that their finite-dimensional distributions share the same weak convergence limits. Consequently, by Proposition~\ref{prop:tight}, $\ybar_{n}(\cdot)$ and $\zbar_{n}(\cdot)$ converge to the same limit. The detailed proof for this subsection and Corollary~\ref{cor:clt} is deferred to the Appendix~\ref{sec:fclt:proof:intergrate}.

\section{Concluding Remarks}\label{sec:fclt:conclude}

In this paper, we have established decoupled FCLTs for two-time-scale SA, representing a significant advancement over prior single-point decoupled convergence results. Our results provide a more comprehensive characterization of the interdependence between different time scales. While the decoupling primarily pertains to the convergence rates, 
our separate limit theorems for the two time scales also offer an alternative perspective on decoupled convergence: one can derive an FCLT for each time scale in isolation, and the other time scale affects the limit at most through the coefficient matrices of the corresponding limiting SDE.

To achieve decoupled FCLTs, we have introduced a novel auxiliary sequence that effectively mitigates the primary influence of the fast-time-scale update on the slow-time-scale dynamics. We would believe this auxiliary sequence can serve as a useful tool for further investigations into two-time-scale algorithms.

Despite the progress presented in this work, several promising directions for future research remain open. First, extending our results to accommodate more general noise settings, such as Markovian or state-dependent noise, would be of significant interest. Second, establishing FCLTs for partial-sum processes of two-time-scale SA and leveraging these results for online inference presents another compelling avenue. Lastly, exploring scenarios involving more complex or heterogeneous time scales could further broaden the applicability of our framework.

We expect that our results can be extended to Markovian noise.
The main challenge is that Markovian noise typically exhibits strong temporal dependence and need not be conditionally unbiased given the history.
When the noise arises from a possibly non-stationary Markov chain, a standard approach is to use the Poisson equation to decompose the noise term into a martingale-difference term and a non-martingale remainder term.
The martingale-difference term can be handled using standard martingale arguments, as in our paper. The remainder term can be handled using telescoping arguments and the smoothness of the operators, together with the small one-step increments and the slow variation of the step sizes.
Such Poisson-equation-based decompositions have been widely used in the analysis of SA with Markovian noise and usually require suitable ergodicity conditions on the Markov chain \citep{li2023online,li2026convergence,kaledin2020finite,srikant2025rates,haque2025stochastic,wang2026steady,yang2026periodic}.

%% file: tex/append.tex
\input{tex/append_assump}

\input{tex/append_results}

\section{Omitted Proofs for Auxiliary Results in Appendices}

\subsection{Proofs for Appendix~\ref{sec:fclt:proof:assump}
}
\label{sec:fclt:proof:assump:proof}

\begin{proof}[Proof of Proposition~\ref{prop:near-linear-tolocal}]
    We first prove the inequality for $F$. \\
    Suppose  that \eqref{assump:near-linear-F-local} hold for $\max\{ \| x - H(y) \|, \| y-y^* \|\} \le \eps$.
    If $\max\{ \| x - H(y) \|, \| y-y^* \|\} > \eps$, we can divide into the following two cases.
    \begin{enumerate}[(i)]
        \item $\| x - H(y) \| > \eps$: From the definition of $H$, we have $F(H(y),y) = 0$. Then
        we have    
        \begin{align*}
        \| F(x,y) - B_1 (x - H(y)) \|
        & \le \| F(x,y) - F(H(y),y) \| + \|B_1 \| \| x-H(y) \| \\
        & \overset{(a)}{\le} (L_F + \| B_1 \|) \| x - H(y) \|
        \le \frac{L_F + \| B_1 \|}{\eps^\Fsmooth} \| x - H(y) \|^{1+\Fsmooth},
    \end{align*}
    where (a) applies the Lipschitz condition of $F$ in Assumption~\ref{assump:smooth}.
    \item $\| x - H(y) \| < \eps$ and $\| y - y^* \| \ge \eps$: Similar to the former case, we have
    \begin{align*}
        \| F(x,y) - B_1 (x - H(y)) \|
        & \le (L_F + \| B_1 \|) \| x - H(y) \|
        \le (L_F + \| B_1 \|) \| y - y^* \| \\
        & \le \frac{L_F + \| B_1 \|}{\eps^\Fsmooth} \| y - y^\star \|^{1+\Fsmooth}.
    \end{align*}
    \end{enumerate}
    In sum, if \eqref{assump:near-linear-F-local} hold for $\max\{ \| x - H(y) \|, \| y-y^* \|\} \le \eps$, we can remove this requirement by replace $\SBF$ with $\max\left\{ \SBF,  \frac{L_F + \| B_1 \|}{\eps^\Fsmooth} \right\}$. \\
    Next, we focus the inequality for $G$. Suppose  that \eqref{assump:near-linear-G-local} hold for $\max\{ \| x - H(y) \|, \| y-y^* \|\} \le \eps$.
    Since $x^* = H(y^*)$, we have $G(H(y^*), y^* = 0)$. Then we have
    \begin{align*}
        & \quad \|G(x, y) - B_2(x {-} H(y))- B_3(y {-} y^{\star}) )\| \\
        & \le \| G(x,y) - G(H(y), y) \| + \| G(H(y), y) - G(H(y^*), y^*) \| \\
        & \quad \ + \| B_2\| \| x - H(y)| \| + \| B_3 \| \| y-y^* \| \\
        & \overset{(b)}{\le} (\LGx+ \| B_2 \|) \| x - H(y) \| + (\LGy + \|B_3 \|) \| y-y^* \|,
    \end{align*}
     where (b) applies the Lipschitz condition of $F$ in Assumption~\ref{assump:smooth}.
     If $\max\{ \| x - H(y) \|, \| y-y^* \|\} > \eps$, we can divide into the following three cases.
     \begin{enumerate}[(i)]
         \item $\| x - H(y) \| > \eps$ and $\| y-y^* \| \le \eps$: We have
         \begin{align*}
             \|G(x, y) - B_2(x {-} H(y))- B_3(y {-} y^{\star}) )\|
             & \le (\LGx+ \| B_2 \| + \LGy + \|B_3 \|) \| x - H(y) \| \\
             & \le \frac{\LGx+ \| B_2 \| + \LGy + \|B_3 \|}{\eps^\Gsmooth} \| x - H(y) \|^{1+\Gsmooth}.
         \end{align*}
         \item $\| x - H(y) \| \le \eps$ and $\| y-y^* \| > \eps$: Similar to the former case, we can obtain 
        \begin{align*}
             \|G(x, y) - B_2(x {-} H(y))- B_3(y {-} y^{\star}) )\|
             & \le \frac{\LGx+ \| B_2 \| + \LGy + \|B_3 \|}{\eps^\Gsmooth} \| y-y^* \|^{1+\Gsmooth}.
         \end{align*}
         \item $\| x - H(y) \| > \eps$ and $\| y-y^* \| > \eps$: we have
        \begin{align*}
             & \quad \|G(x, y) - B_2(x {-} H(y))- B_3(y {-} y^{\star}) )\| \\
             & \le \frac{\LGx+ \| B_2 \|}{\eps^\Gsmooth} \| x-H(y) \|^{1+\Gsmooth} + \frac{\LGy + \|B_3 \|}{\eps^\Gsmooth} \| y-y^* \|^{1+\Gsmooth}.
         \end{align*}
     \end{enumerate}
    In sum, if \eqref{assump:near-linear-G-local} hold for $\max\{ \| x - H(y) \|, \| y-y^* \|\} \le \eps$, we can remove this requirement by replace $\SBG$ with $\max\left\{ \SBG,  \frac{\LGx+ \| B_2 \| + \LGy + \|B_3 \|}{\eps^\Gsmooth} \right\}$.
\end{proof}
\begin{proof}[Proof of Proposition~\ref{prop:local-linear-local}]
For $F$,
Since 
$ \| x-H(y) \| \le \| x - x^* \| + \| H(y^*) - H(y) \| \le \| x - x^* \| + L_H \| y - y^* \| $ and $F(H(y),y)=0$, we have
\begin{align*}
    & \quad  \|F(x, y) - A_{11}(x -x^{\star}) - A_{12} (y-y^{\star}) \| \\
    & \le \| F(x,y) - F(H(y),y) \| + \| A_{11} \| \| x - x^* \| + \| A_{12} \| \| y - y^* \| \\
    & \le (L_F + \| A_{11} \|) \| x-x^* \| + (L_F L_H + \|A_{12} \|) \| y-y^* \|.
\end{align*}
For $G$, since $G(H(y^*),y^*)=0$, we can also obtain 
\begin{align*}
    & \quad  \|G(x, y) - A_{21}(x -x^{\star}) - A_{22} (y-y^{\star}) \| \\
    & \le \| G(x,y) - G(H(y),y) \| + \|G(H(y),y) - G(H(y^*),y^*) \| \\
    & \quad \ + \| A_{11} \| \| x - x^* \| + \| A_{12} \| \| y - y^* \| \\
    & \le (\LGx + \| A_{21} \|) \| x-x^* \| + (\LGx L_H + \LGy + \|A_{22} \|) \| y-y^* \|.
\end{align*}
Then similar to the proof of Proposition~\ref{prop:near-linear-tolocal}, we can remove the requirement  $\max\{ \| x - x^* \|, \| y-y^* \|\} \le \eps$ by replacing $S_{A,F} $ with $\tilde{S}_{A,F} := \max\left\{ S_{A,F}, \frac{L_F + \| A_{11} \| + L_F L_H + \|A_{12} \| }{\|\eps\|^{\Fsmooth}} \right\}$
and 
replacing $S_{A,G} $ with $\tilde{S}_{A,G} := \max\left\{ S_{A,G}, \frac{\LGx + \| A_{21} \| + \LGx L_H + \LGy + \|A_{22} \| }{\|\eps\|^{\Gsmooth}} \right\}$.
For the remaining proof, please refer to \cite[Proposition~2]{han2024finite} and the proof therein.
\end{proof}

\begin{proof}[Proof of Proposition~\ref{prop:H-reduce-order}]
	Combining \eqref{assump:smooth:FH:ineqH} and Assumption~\ref{assump:smoothH}, we have
	\begin{equation*} 
		\| H(y_1) -  H(y_2) - \nablaH(y_2) (y_1-y_2)  \| \le S_H \|y_1-y_2\| \cdot \min\left\{\|y_1-y_2\|^\Hsmooth, R_H \right\}
	\end{equation*}
	with $R_H = \frac{2L_H}{S_H}$.
	For any $\delta'_H \in [0, \Hsmooth]$,
	if $\| \yy_1 - \yy_2 \|^\Hsmooth \le R_H$, then we have 
	\begin{equation*}
		\| H(\yy_1) - H(\yy_2) - \nablaH(\yy_2) (\yy_1-\yy_2) \| \le S_H R_H^{\Hsmooth - \delta'_H} \| \yy_1 - \yy_2 \|^{1+\delta'_H};
	\end{equation*}
	otherwise, we have
	\begin{equation*}
		\| H(\yy_1) - H(\yy_2) - \nablaH(\yy_2) (\yy_1-\yy_2) \| \le S_H R_H \| \yy_1 - \yy_2 \| \le S_H R_H / R_H^{\delta'_H}  \| \yy_1 - \yy_2 \|^{1+\delta'_H}.
	\end{equation*}
	Thus, by replacing $S_H$ with $\max\left\{S_H R_H^{\Hsmooth - \delta'_H}, S_H R_H /  R_H^{\delta'_H} \right\}$,
	Assumption~\ref{assump:smoothH} holds up to order $1+\delta'_H$.
\end{proof}

\begin{proof}[Proof of Proposition~\ref{prop:FG-reduce-order}]
	For ease of exposition,  we let $\xhat = \xx - H(\yy)$ and $\yhat = \yy - \yy^\star$.
	First, these assumptions imply
	\begin{align*}
		\|F(\xx, \yy) - B_1 \xhat \| 
		& = \|F(\xx, \yy) - F(H(\yy), \yy) - B_1 \xhat  \| \le 2 L_F \| \xhat \| + \eps \cdot \| \yhat \|, \\
		\| G(\xx, \yy) - B_2 \xhat - B_3 \yhat ) \| 
		& = \| G(\xx, \yy) {-} G(H(\yy), \yy) + G(H(\yy), \yy) {-} G(H(\yy^\star), \yy^\star) {-} B_2 \xhat {-} B_3 \yhat \| \\
		& \le 2 \LGx \| \xhat \| + 2 \LGy \| \yhat \|,
	\end{align*} 
	where $\eps$ is an arbitrary positive constant.
	In the following, we only prove that \eqref{assump:near-linear-G} holds after  replacing $\Gsmooth$ with $\delta'_G \in (0, \Gsmooth)$.
	The proof for \eqref{assump:near-linear-F} is similar.
	
	If (i) $ \SBG \| \xhat \|^{1+\Gsmooth} \le 2 \LGx \| \xhat \| $ and $\SBG \| 
	\yhat \|^{1+\Gsmooth} \le 2 \LGy \| \yhat \| $ 
	or (ii) $ \SBG \| \xhat \|^{1+\Gsmooth} > 2 \LGx \| \xhat \| $ and $\SBG \| 
	\yhat \|^{1+\Gsmooth} > 2 \LGy \| \yhat \| $,
	we have
	\begin{align*}
		\| G(\xx, \yy) - B_2 \xhat - B_3 \yhat ) \|
		& \le \min \left\{ \SBG \| \xhat \|^{1+\Gsmooth} + \SBG \| \yhat \|^{1+\Gsmooth}, 2 \LGx \| \xhat \| + 2 \LGy \| \yhat \| \right\} \\
		& = \min \left\{ \SBG \| \xhat \|^{1+\Gsmooth}, 2 \LGx \| \xhat \| \right\} + \min \left\{ \SBG \| \yhat \|^{1+\Gsmooth}, 2 \LGy \| \yhat \|  \right\} \\
		& = \SBG \| \xhat \| \cdot \min\left\{ \| \xhat \|^\Gsmooth, \frac{2 \LGx}{\SBG} \right\} + \SBG \| \yhat \| \cdot \min \left\{ \| \yhat \|^{\Gsmooth}, \frac{2\LGy}{\SBG} \right\}.
	\end{align*}
	Then the remaining proof is similar to that of Proposition~\ref{prop:H-reduce-order}.
	
	If neither of (i) and (ii) holds, without loss of generality we assume $ \SBG \| \xhat \|^{1+\Gsmooth} \le 2 \LGx \| \xhat \| $ and $\SBG \| 
	\yhat \|^{1+\Gsmooth} > 2 \LGy \| \yhat \| $.
	This implies $ \| \xhat \| \le c_x := (2 \LGx / \SBG)^{1/\Gsmooth}  $ and $\| \yhat \| > c_y := ( 2 \LGy / \SBG )^{1/\Gsmooth}$.
	Then we have 
	\begin{align*}
		\| G(\xx, \yy) - B_2 \xhat - B_3 \yhat ) \|
		& \le \min \left\{ \SBG \| \xhat \|^{1+\Gsmooth} + \SBG \| \yhat \|^{1+\Gsmooth}, 2 \LGx \| \xhat \| + 2 \LGy \| \yhat \| \right\} \\
		& \le 2 \LGx \| \xhat \| + 2 \LGy \| \yhat \| 
		\le  2 \LGx\, c_x + 2 \LGy \| \yhat \| \\
		& \le \left( \frac{2 \LGx c_x}{ c_y^{1+\delta'_G} } + \frac{2 \LGy}{c_y^{\delta'_G}} \right) \| \yhat \|^{1+\delta'_G}.
	\end{align*}
	Thus \eqref{assump:near-linear-G} holds with $\Gsmooth$ and $\SBG$ replaced with $\delta'_G$ and $\frac{2 \LGx c_x}{ c_y^{1+\delta'_G} } + \frac{2 \LGy}{c_y^{\delta'_G}}$ respectively.
\end{proof}

\begin{proof}[ Proof of Proposition~\ref{prop:operator-decom} ]
We first prove Part~(i).
By Assumption~\ref{assump:near-linear} and the definition of $\xhat_{n}$ and $\yhat_n$ in \eqref{alg:xyhat}, we have
    $\|F(x_n, y_n) - B_1 \xhat_{n}\| \le \SBG ( \| \xhat_{n} \|^{1+\Fsmooth} + \| \yhat_{n} \|^{1+\Fsmooth} )$.
Setting $R^F_n = F(x_{n}, y_{n}) - B_1 \xhat_{n}$ yields that $\|R^F_n \| = \OM(\| \xhat_{n} \|^{1+\Fsmooth} + \| \yhat_{n} \|^{1+\Fsmooth})$. By Proposition~\ref{prop:FG-reduce-order}, for any $\delta'_F \in (0,\Fsmooth]$, we also have $\| R^F_n \| = \OM(\| \xhat_{n} \|^{1+\delta'_F} + \| \yhat_{n} \|^{1+\delta'_F} ) $. The proof of Part~(ii) is similar to that of Part~(i) by setting $R^G_n = G(x_n, y_n) - B_2 \xhat_{n} - B_3 \yhat_{n}$.

For Part~(iii), we first set $R^H_n = H(y_{n+1}) - H(y_{n}) - \nabla H(y_{n}) (y_{n+1} - y_{n})$. Then by Assumption~\ref{assump:smoothH} and Proposition~\ref{prop:H-reduce-order}, for any $\delta'_H \in (0, \Hsmooth]$, we have $\|R^H_n \| = \OM( \| y_{n+1} - y_{n} \|^{1+\delta'_H} ) $.
With  $R_{n}^{\nablaH} =  (\nablaH(\yy_{n}) - \nablaH(\yy^{\star}))(\yy_{n+1}-\yy_{n})$, we have $H(y_{n+1}) - H(y_{n}) = H^\star(y_{n+1} - y_{n}) + R_{n}^{\nablaH} + R_n^H $.
\end{proof}

\subsection{Proofs for Appendix~\ref{append:construct-details}}
\label{append:construct-details-proof}

\begin{proof}[Proof of Lemma~\ref{lem:aux:bounded}]
	We first prove the result for the sequence $\{ \beta_{n} \}$.
	From the definition of $N^\beta(n,t)$ in Definition~\ref{defn:time-inter}, we must have $N^\beta(n,t) \ge n$.
	Then the monotonicity of $n \beta_{n}$ in Assumption~\ref{assump:stepsize-twotime} implies $\frac{\beta_{n}}{ \beta_{ N^\beta(n,t) } } \le \frac{ N^\beta(n,t) }{n}$. It remains to derive the upper bound of $\frac{ N^\beta(n,t) }{n} $.
	
	From Assumption~\ref{assump:stepsize-twotime}, we have 
	$\frac{ \beta_{n-1} }{ \beta_{n} } = 1 + \OM (\beta_{n}) = 1 + \OM (\beta_{n-1}) $.
	It follows that $\frac{1}{\beta_{n}} = \frac{1}{\beta_{n-1}} + \OM(1) = \dots = \OM(n) $, implying $\beta_{n} = \Omega \big( \frac{1}{n} \big)$.
	Consequently, 
	there exists $c_\beta$ such that $\beta_{n} \ge \frac{c_\beta}{n} $ for any $n \in \NB$.
	Meanwhile, for a fixed $n$, there exists a unique integer $k(n)$ such that $ 2^{k(n) - 1} \le n < 2^{k(n)} $.
	Then we have that for any $k' \ge k(n)$,
	\begin{equation*}
		\beta_{2^{k'} + 1} + \dots + \beta_{ 2^{k' + 1} }
		\ge c_\beta \left( \frac{1}{ 2^{k'} + 1 } + \frac{1}{ 2^{k'} + 2 } + \dots \frac{1}{ 2^{k'+1} } \right)
		\ge \frac{ c_\beta }{2}.
	\end{equation*}
	It follows that with $m(n) := 2^{ k(n) + \left\lceil 2t/ c_\beta \right\rceil } $,
	\begin{equation*}
		\Gamma^\beta_{n, m(n)+1}
		= \ssum{l}{n}{m(n)} \beta_{l}
		> \ssum{l}{2^{k(n)}+1}{2^{k(n) + \left\lceil 2t/c_\beta \right\rceil } } \beta_{l}
		\ge \left\lceil \frac{2t}{c_\beta} \right\rceil \frac{c_\beta}{2} \ge t.
	\end{equation*}
	From Definition~\ref{defn:time-inter}, we have $N^\beta(n,t) \le m(n) = 2^{ k(n) + \left\lceil 2t/c_\beta \right\rceil }$.
	Recall that $n \ge 2^{k(n)-1}$. It follows that $\frac{ N^\beta(n,t) }{n} \le 2^{ \left\lceil 2t/c_\beta \right\rceil + 1 }$.
	
	For the sequence $\{ \alpha_{n} \}$, Assumption~\ref{assump:stepsize-twotime} guarantees that $n \alpha_{n}$ is non-decreasing and $\frac{\alpha_{n-1}}{\alpha_{n}} = 1 + o(\beta_{n}) = 1 + \OM(\alpha_{n})$.
	Then following similar procedure, we have that there exists $c_\alpha > 0$ such that $ \frac{ \alpha_{n} }{ \alpha_{ N^\alpha(n,t) } } \le \frac{ N^\alpha(n,t) }{n} \le 2^{ \left\lceil 2t/c_\alpha \right\rceil + 1 }$.
\end{proof}

\begin{proof}[Proof of Proposition~\ref{prop:aux:step-size}]
    The proof for $N^\alpha(n,t)$ and $N^\beta(n,t)$ is identical and without loss of generality, we focus on the latter.
    Recall the definition of $\Gamma^\beta_{n, m}$ in Definition~\ref{defn:time-inter}, for $\beta_n = \frac{\beta_0}{(n+1)^\beta}$, we have $\Gamma^\beta_{n, m} = \sum_{k=n+1}^{m} \frac{\beta_0}{k^\beta} $.
    
    If $\beta = 1$, then $\Gamma^\beta_{n, m} = \sum_{k=n+1}^{m} \frac{\beta_0}{k} \in \left( \beta_0 \log \frac{m+1}{n+1}, \beta_0 \log \frac{m}{n}\right) $.
    From the definition of $N^\beta(n,t)$ in Definition~\ref{defn:time-inter}, we have $\Gamma^\beta_{n, N^\beta(n,t)} \le t$ and $\Gamma^\beta_{n, N^\beta(n,t)+1} \ge t$. 
    Then we have $\log \frac{N^\beta(n,t)+1}{n+1} \le t / \beta_0 \le \log \frac{N^\beta(n,t)+1}{n}  $ and consequently 
    \begin{equation*}
         n (e^{t/\beta_0}-1) - 1 \le N^\beta(n,t) - n \le n (e^{t/\beta_0}-1) + e^{t/\beta_0} - 1.
    \end{equation*}    
    Thus, $N^\beta(n,t) - n =  (e^{t/\beta_0}-1) n + o(n)$. 
    
    If $\beta < 1$, then $\Gamma^\beta_{n, m} = \sum_{k=n+1}^{m} \frac{\beta_0}{k^\beta} \in \left( \frac{\beta_0}{1-\beta} ((m+1)^{1-\beta} - (n+1)^{1-\beta} ), \frac{\beta_0}{1-\beta} ( m^{1-\beta}- n^{1-\beta} )   \right)  $.
    From $\Gamma^\beta_{n, N^\beta(n,t)} \le t$ and $\Gamma^\beta_{n, N^\beta(n,t)+1} \ge t$, we can obtain 
    \begin{equation*}
        \left[ (1-\beta) t / \beta_0 + n^{1-\beta} \right]^{1/(1-\beta)} -1 \le N^\beta(n,t) \le \left[ (1-\beta) t / \beta_0 + (n+1)^{1-\beta} \right]^{1/(1-\beta)} -1.
    \end{equation*}
    For the left-hand side, we have $\left[ (1-\beta) t / \beta_0 + n^{1-\beta} \right]^{1/(1-\beta)} -1 = n \left[1 +  (1-\beta) t n^{-(1-\beta)} / \beta_0 \right]^{1/(1-\beta)} -1 = n + t n^{\beta} / \beta_0 + o(n^\beta)$.
    For the right-hand side, similarly, we have $\left[ (1-\beta) t / \beta_0 + (n+1)^{1-\beta} \right]^{1/(1-\beta)} -1 = n  +  t n^\beta/\beta_0 + o(n^\beta)$.
    In sum, we have $N^\beta(n,t) - n = \Theta(n^\beta) $.
\end{proof}

\begin{proof}[Proof of Proposition~\ref{prop:hurwitz-expbound}]
    Since $-A$ is a Hurwitz matrix, the  Lyapunov equation
$A^{\top}P + P A = I_d$
has a unique symmetric positive definite solution $P\in\mathbb R^{d\times d}$ \citep[p.~475]{antsaklis2005linear}. Define
\(p_{\max}:=\lambda_{\max}(P), p_{\min}:=\lambda_{\min}(P), \kappa(P):=p_{\max}/p_{\min}.
\)
For any $x\in\mathbb R^d$ set $y(t):=e^{-At}x$. Differentiate
\[\frac{d}{dt}\big(y(t)^{\top}P y(t)\big) = -y(t)^{\top}(A^{\top}P+PA)y(t) = -\|y(t)\|^2.
\]
Using $\|y\|^2\ge\frac{1}{p_{\max}}y^{\top}P y$ gives
$\frac{d}{dt}\big(y(t)^{\top}P y(t)\big) \le -\frac{1}{p_{\max}}y(t)^{\top}P y(t)$.
Then Gronwall's inequality yields
$y(t)^{\top}P y(t) \le e^{-t/p_{\max}} y(0)^{\top}P y(0) \le e^{-t/p_{\max}} p_{\max} \| y(0) \|^2 $.
Hence, with $y(0)=x$,
\begin{equation*}
\|e^{-At}\|^2 = \sup_{\|x\|=1} \|e^{-At}x\|^2 \le \sup_{\|x\|=1} \frac{1}{p_{\min}} (e^{-At}x)^{\top}P(e^{-At}x) \le \frac{p_{\max}}{p_{\min}} e^{-t/p_{\max}} = \kappa(P) e^{-t/p_{\max}},
\end{equation*}
where the first inequality is because $P - p_{\min} \mI$ is positive semidefinite.
Then one can check the solution to $A^{\top}P + P A = I_d$ is $ P = \int_0^\infty e^{-A^\top t} e^{-At} dt$, which is well-defined \citep[Theorem~7.5]{antsaklis2005linear}.
Therefore we have
\begin{equation*}
\|e^{-At}\| \le M e^{-b t},\forall t\ge0,\ \text{with} \ M := \sqrt{\kappa(P)}, b := \frac{1}{2p_{\max}}.
\end{equation*}
\end{proof}

\textbf{
Derivation of $\PB(\sup_{0 \le t \le T} \norm{\ermY_\infty(t)} \ge M) \le \delta$ } 

For ease of notation, we omit the subscipt of $\ermY_\infty$ and assume $\ermY(\cdot)= (\ermY^{1}(\cdot), \ermY^{2}(\cdot), \dots, \ermY^{d_y}(\cdot)  )^\top$ is a stationary $d_y$-dimensional OU process. Then each component $\ermY^i(\cdot)$ is a mean-zero Gaussian process.
For a fixed $T$,
by Borell–TIS inequality~\citep[Theorem~2.1.1]{adler2007random}, we have that for all $u \ge 0$,
\begin{equation*}
    \PB\left(\sup_{0 \le t \le T} \ermY^{i}(t) \ge \EB \sup_{0 \le t \le T} \ermY^{i}(t) + u\right) \le \exp\left(-\frac{u^2}{2 \sup_{0 \le t \le T} E (\ermY^{i}(t))^2} \right).
\end{equation*}
Similarly, $\PB\left(-\inf_{0 \le t \le T} \ermY^{i}(t) \ge - \EB \inf_{0 \le t \le T} \ermY^{i}(t) + u\right) \le \exp\left(-\frac{u^2}{2 \sup_{0 \le t \le T} E (\ermY^{i}(t))^2 } \right)$.
Note that $\sup_{0 \le t \le T} |\ermY^{i}(t)| = \max\{\sup_{0 \le t \le T} \ermY^{i}(t), -\inf_{0 \le t \le T} \ermY^{i}(t)  \} $.
By Dudley integral inequality \citep[Theorem~8.1.3]{vershynin2018high}, we can obtain an upper bound for $ \EB \sup_{0 \le t \le T} \ermY^{i}(t)$ and $- \EB \inf_{0 \le t \le T} \ermY^{i}(t) $ and denote this bound as $A_T$.
Then we have
\begin{equation*}
    \PB\left(\sup_{0 \le t \le T} |\ermY^{i}(t)| \ge A_T + u\right) \le 2\exp\left(-\frac{u^2}{2\sup_{0 \le t \le T}  \EB (\ermY^{i}(t))^2} \right).
\end{equation*}
Let $M = \sqrt{d_y} (A_T + u)$.
Then we obtain
\begin{align*}
    \PB\left(\sup_{0 \le t \le T} \norm{\ermY(t)} \ge M\right)
    & = \PB\left(\sup_{0 \le t \le T} \sum_{i=1}^{d_y}{(\ermY^{(i)}(t))^2} \ge M^2\right)
    \le \PB\left(\sum_{i=1}^{d_y} \sup_{0 \le t \le T} (\ermY^{(i)}(t))^2 \ge M^2 \right) \\ 
    & \le \sum_{i=1}^{d_y} \PB\left( \sup_{0 \le t \le T} (\ermY^{(i)}(t))^2 \ge M^2/d_y^2 \right)
    = \sum_{i=1}^{d_y} \PB\left( \sup_{0 \le t \le T} |\ermY^{(i)}(t)| \ge M/\sqrt{d_y} \right) \\
    & \le 2d_y \exp\left( - \frac{u^2}{2\sigma_T^2}\right),
\end{align*}
where $\sigma_T^2 := \sup_{1\le i\le d_y}\sup_{0 \le t\le T} \EB ( \ermY^{i}(t))^2 $.
It suffices to set $\delta = 2d_y \exp\left( - \frac{u^2}{2\sigma_T^2}\right) $ and we obtain $M = \sqrt{d_y} (A_T + \sigma_T \sqrt{2 \log (2d_y / \delta)} )$.

\input{tex/append_proof}

%% file: tex/append_assump.tex
\section{Intermediate Results about the Assumptions }
\label{sec:fclt:proof:assump}

This section presents some intermediate results about the Assumptions, with their proofs deferred to Appendix~\ref{sec:fclt:proof:assump:proof}.

Assumption~\ref{assump:smoothH} is equivalent to the H\"older continuity of $\nabla H$.

\begin{assumpt}{\ref*{assump:smoothH}$\dagger$}[$\Hsmooth$-H\"older continuity of $\nabla H$]
	\label{assump:H:holder-assump}
	Assume that $H(\cdot)$ is differentiable and there exist constants $\Hholder \ge 0$ and $\Hsmooth \in [0,1]$ such that $\forall \, y_1, y_2 \in \RB^{d_y}$, 
	\begin{equation}
		\label{assump:H:holder}
		\| \nabla H(y_1) - \nabla H(y_2) \|
		\le \Hholder \| y_1 - y_2 \|^{\Hsmooth}.
	\end{equation}
\end{assumpt}

\begin{prop}[{\cite[Proposition~1]{han2024finite}}]\label{prop:holder-equiv}
	Suppose $\Hsmooth \in [0,1]$.
	Under Assumption~\ref{assump:H:holder-assump},
	Assumption~\ref{assump:smoothH} holds with $S_H = \frac{\Hholder}{1+\Hsmooth} $;
	under Assumption~\ref{assump:smoothH}, Assumption~\ref{assump:H:holder-assump} holds with $\Hholder = 2^{1-\Hsmooth} \sqrt{1+\Hsmooth} \left( \frac{1+\Hsmooth}{\Hsmooth} \right)^\frac{\Hsmooth}{2} S_H$
		(we make the contention that $\left( \frac{1+\Hsmooth}{\Hsmooth} \right)^\frac{\Hsmooth}{2} = 1$ if $\Hsmooth = 0$). 
\end{prop}

Under Assumption~\ref{assump:smooth}, Assumption~\ref{assump:near-linear} can be replaced by the following weaker condition, which requires local linearity only in a neighborhood of $(x^\star, y^\star)$.
\begin{assumpt}{3.3\dag}[Nested local linearity of $F$ and $G$ up to order ($1+\Fsmooth, 1+\Gsmooth$)]
	\label{assump:near-linear-local}
	There exist matrices $B_1, B_2, B_3$ with compatible dimensions, constants $\SBF, \SBG \ge 0$, $\eps > 0$, and $\Fsmooth, \Gsmooth \in (0,1]$ such that for any $(x,y)$ satisfying $\max\{ \| x - H(y) \|, \| y-y^* \|\} \le \eps$,
	\begin{align}
		\|F(x, y) - B_1(x {-} H(y))\| &\le \SBF \left( \|x {-} H(y)\|^{1+\Fsmooth} +  \|y {-} y^{\star}\|^{1+\Fsmooth} \right),
		\label{assump:near-linear-F-local}
		\\
		\|G(x, y) - B_2(x {-} H(y))- B_3(y {-} y^{\star} )\| &\le \SBG \left( \|x {-} H(y)\|^{1+\Gsmooth} + \|y {-} y^{\star}\|^{1+\Gsmooth} \right). 
		\label{assump:near-linear-G-local}
	\end{align}
\end{assumpt}
\begin{prop}\label{prop:near-linear-tolocal} 
    Suppose that Assumptions~\ref{assump:smooth} and \ref{assump:near-linear-local} hold. Then Assumption~\ref{assump:near-linear} hold with $\SBF$ replaced by $\max\left\{ \SBF,  \frac{L_F + \| B_1 \|}{\eps^\Fsmooth} \right\}$ and $\SBG$ replaced by $\max\left\{ \SBG,  \frac{\LGx+ \| B_2 \| + \LGy + \|B_3 \|}{\eps^\Gsmooth} \right\}$.
\end{prop}
Moreover, the nested local linearity can be derived from the following standard local linearity.
\begin{assumpt}{3.3\ddag}[Local linearity up to order ($1+\Fsmooth, 1+\Gsmooth$)]
	\label{assump:local-linear-local}
	There exists matrices $A_{11}, A_{12}, A_{21}, A_{22}$ with compatible dimensions, constants $S_{A,F}, S_{A,G} \ge 0$, $\eps > 0$ and $\Fsmooth, \Gsmooth \in (0,1]$ such that for any $(x,y)$ satisfying $\max\{ \| x - x^* \|, \| y-y^* \|\} \le \eps$,
	\begin{align}
		\|F(x, y) - A_{11}(x -x^{\star}) - A_{12} (y-y^{\star}) \| &\le S_{A,F} \left( \|x - x^{\star}\|^{1+\Fsmooth} + \|y-y^{\star}\|^{1+\Fsmooth} \right), \label{assump:local-linear-F} \\
		\|G(x, y) - A_{21}(x -x^{\star}) - A_{22}(y-y^{\star} )\| &\le S_{A,G} \left( \|x - x^{\star}\|^{1+\Gsmooth} + \|y-y^{\star}\|^{1+\Gsmooth} \right). \label{assump:local-linear-G}
	\end{align}
\end{assumpt}
\begin{prop}\label{prop:local-linear-local}
Suppose that Assumptions~\ref{assump:smooth}, \ref{assump:smoothH}, and \ref{assump:local-linear-local} hold and $A_{11} $ is invertible.
If we further assume $\Hsmooth \ge \Fsmooth \vee \Gsmooth$, then Assumption~\ref{assump:near-linear} holds with parameters
		\begin{align*}
			B_1 & = A_{11},\ \ B_2 = A_{21},\ \ B_3 = A_{22}-A_{21}A_{11}^{-1}A_{12}.
            ,\\
			S_{B,F} &= S_{A,F} + \frac{L_F + \| A_{11} \| + L_F L_H + \|A_{12} \| }{\|\eps\|^{\Fsmooth}} + L_F (S_H \vee 2 L_H), \\
			S_{B,G} & = S_{A,G} + \frac{\LGx + \| A_{21} \| + \LGx L_H + \LGy + \|A_{22} \| }{\|\eps\|^{\Gsmooth}} + \LGx (S_H \vee 2 L_H).
		\end{align*}
\end{prop}

The following propositions suggest that with the Lipschitz conditions, we can always select a lower order without compromising the validity of Assumptions~\ref{assump:smoothH} and \ref{assump:near-linear}.

\begin{prop}\label{prop:H-reduce-order}
	Suppose that \eqref{assump:smooth:FH:ineqH} holds.
	If Assumption~\ref{assump:smoothH} holds up to order $1 + \Hsmooth$, then it also holds up to order $1+\delta'_H$ for any $\delta'_H \in [0, \Hsmooth] $.
\end{prop}

\begin{prop}\label{prop:FG-reduce-order}
	Suppose that Assumption~\ref{assump:smooth} holds.
	If Assumption~\ref{assump:near-linear} holds up to order $(1+\Fsmooth, 1+\Gsmooth)$, then it also holds up to order $(1+\delta'_F, 1+\delta'_G)$ for any $\delta'_F \in (0, \Fsmooth)$ and $\delta'_G \in (0, \Gsmooth)$.
\end{prop}

With Assumptions~\ref{assump:smoothH} and \ref{assump:near-linear}, we can decompose these operators as linear parts and higher-order residual terms.

\begin{prop}[Operator decomposition]
	\label{prop:operator-decom}
	Suppose that Assumptions~\ref{assump:smooth} -- \ref{assump:near-linear} hold.
	With $\xhat_{n}$ and $\yhat_{n}$ defined in \eqref{alg:xyhat}, for any $\delta'_F \in (0, \Fsmooth]$, $\delta'_G \in (0, \Gsmooth]$ and $\delta'_H \in (0, \Hsmooth]$, we have the following results.
	\begin{enumerate}[(i)]
		\item $F(\xx_{n}, \yy_{n}) = B_1 \xhat_{n} + R_{n}^F$ with $\|R_{n}^F\| = \OM( \| \xhat_{n} \|^{1+\delta'_F} + \| \yhat_{n} \|^{1+\delta'_F} )
		$.
		\item $G(\xx_{n}, \yy_{n}) = B_2\xhat_{n} + B_3\yhat_{n} + R_{n}^{G}$ with $\|R_n^{G}\| = \OM( \| \xhat_{n} \|^{1+\delta'_G} + \| \yhat_{n} \|^{1+\delta'_G} )
		$. 
		\item $H(\yy_{n+1}) - H(\yy_{n}) = H^\star  (\yy_{n+1}-\yy_{n}) + R_{n}^{\nablaH} + R_{n}^H $ with $H^\star := \nabla H(\yy^\star)$, $\|R_{n}^H\| = \OM (\|\yy_{n+1}-\yy_{n}\|^{1+\delta'_H} )$ and $R_{n}^{\nablaH} =  (\nablaH(\yy_{n}) - \nablaH(\yy^{\star}))(\yy_{n+1}-\yy_{n})$.
	\end{enumerate}
\end{prop}

\begin{rema}[Implications of Assumption~\ref{assump:stepsize-twotime}]\label{rema:step-size}
	We can derive the following results from Assumption~\ref{assump:stepsize-twotime}.
	\begin{itemize}
		\item
		Assumption~\ref{assump:stepsize-twotime} implies $\beta_n = \Omega(1/n)$ and $\alpha_{n} / \beta_{n} = \omega(1)$. Then we have
			$\ssum{n}{0}{\infty} \alpha_{n} = \ssum{n}{0}{\infty} \beta_{n} = \infty$,
		which is a ubiquitous condition in the SA analysis to ensure convergence to the target solution \citep{konda2004convergence, doan2022nonlinear, faizal2023functional}.
		
		\item 
		 Note that $\alpha_{n+1}^{-1} - \alpha_{n}^{-1} = \beta_{n+1}^{-1} \kappa_{n+1} - \beta_{n}^{-1} \kappa_{n} = \beta_{n+1}^{-1} (\kappa_{n+1} - \kappa_{n}) + \kappa_{n} ( \beta_{n+1}^{-1} - \beta_{n}^{-1} ) = \beta_{n+1}^{-1} \kappa_{n+1} ( 1 - \kappa_{n} / \kappa_{n+1} ) + \kappa_{n} ( \beta_{n+1}^{-1} - \beta_{n}^{-1} ) $.
		If we define  $\invdifffast := \lim_{n \to \infty} (\alpha_{n+1}^{-1} - \alpha_{n}^{-1} )$ then we have $\invdifffast = 0$. 
		This is why the choice of step sizes does not affect the SDE in \eqref{eq:fclt-twotime:x-sde}.
		The parameters $\invdiffslow$ and $\invdifffast$ also appear frequently in previous analysis \citep{konda2004convergence, faizal2023functional}.
	\end{itemize}
\end{rema}

\begin{rema}[Linear case]
	If $F$ and $G$ are linear operators, implying $H$ is also linear. the residual terms in Proposition~\ref{prop:operator-decom} vanish.
	Under this scenario, for Lemmas~\ref{lem:check-twotime-xy} and \ref{lem:check-twotime-z}, as well as Theorem~\ref{thm:fclt-twotime}, to hold, Condition~\ref{assump:stepsize-twotime-moreab} in Assumption~\ref{assump:stepsize-twotime} is no longer necessary. Moreover, the requirement for the boundedness of fourth moments can be relaxed to $p$-th moments for any $p>2$.
	For 
	polynomially diminishing step sizes 
	$\alpha_{n} = \alpha_{0} (n+1)^{-a}$ and $\beta_{n} = \beta_{0} (n+1)^{-b}$ with $a, b \in (0, 1]$, 
	it suffices to ensure $b / a > 1$.
\end{rema}

%% file: tex/append_results.tex
\section{Omitted Details for Section~\ref{sec:fclt:main}}\label{append:construct}
In this section, we present more details for the construction in Section~\ref{sec:fclt:construct} and the examples in Section~\ref{sec:fclt:exam}.

\subsection{Non-asymptotic Convergence Rates}\label{append:construct-non-asymp}
We first give the formal statement of Proposition~\ref{prop:mart-decouple-rate}.

\begin{prop}[Non-asymptotic convergence rates, {\cite[Theorem~3]{han2024finite}}]\label{prop:mart-decouple-rate-formal}
	Suppose that Assumptions~\ref{assump:smooth} -- \ref{assump:noise-fclt} and the following strong monotonicity conditions hold:
	(i) there exists a constant $\mu_{F} > 0$ such that for any $y \in \RB^{d_y}$ and $x \in \RB^{d_x}$, 
		$\left\langle x - H(y), F(x,y) - F(H(y),y) \right\rangle \geq \mu_{F} \|x-H(y)\|^2$;
	(ii) there exists a constant $\mu_{G} > 0$ such that for any $ y\in\RB^{d_y}$,
		$\left\langle y - y^{\star}, G(H(y),y) - G(H(y^\star), y^\star) \right\rangle \geq \mu_{G} \|y - y^{\star}\|^2$.
	Then we have $\EB \| \xhat_{n} \|^4 + \EB \| \yhat_{n} \|^4 = \OM (\alpha_{n}^2)$.
	If we further assume $b/a \le 1 + \Fsmooth/2$ in the step size selection of Example~\ref{exmp:fclt-stepsize} and $\alpha_0$ and $\beta_0$ are properly chosen, we have $\EB \| \xhat_{n} \|^2 = \OM(\alpha_{n}) $ and $\EB \| \yhat_{n} \|^2 = \OM(\beta_{n})$.
\end{prop}
Regarding the comparison between strong monotonicity and Assumption~\ref{assump:hurwitz}, strong monotonicity is a non-asymptotic condition that guarantees convergence from arbitrary initializations, whereas Assumption~\ref{assump:hurwitz} is a local condition that guarantees only the asymptotic behavior.
As for the additional requirement for the step sizes, we guess that this is due to proof techniques in \cite{han2024finite} and is not inherent to the problem itself.
\
Although \cite[Assumption~7]{han2024finite} imposes an additional condition on the step sizes, for
$\alpha_{n}=\alpha_{0}(n+1)^{-a}$ and $\beta_{n}=\beta_{0}(n+1)^{-b}$ with suitably chosen $\alpha_0$ and $\beta_0$, this condition holds for all sufficiently large $n$ and therefore does not affect our asymptotic analysis.

Moreover, in the linear setting, \cite{haque2023tight} derives a sharper control of the slow-iterate error by working with the norm induced by the solution to an appropriate Lyapunov equation, rather than the Euclidean norm.
We expect that extending this technique to the nonlinear case could lead to tighter non-asymptotic convergence rates.

\subsection{Details for the Construction}\label{append:construct-details}
In this subsection, we present an equivalent expression of the construction in \eqref{eq:xbar-short}, \eqref{eq:ybar-short} and \eqref{eq:zbar-short} for ease of subsequent proof.
We first introduce some necessary notation inspired by \cite{gadat2018stochastic, liang2023asymptotic}.
\begin{defn}[Time interpolation]\label{defn:time-inter}
	For the step size sequences $\{ \alpha_{n} \}_{n=0}^\infty$ and $\{ \beta_{n} \}_{n=0}^\infty$, $n, n' \in \NB$ with $n' > n$ and $t \ge 0$, we define
	\begin{align*}
		\Gamma^\alpha_{n, n'} := \ssum{k}{n}{n'-1} \alpha_{k},\ 
		N^\alpha(n, t) := \max \left\{ m \ge n \colon \Gamma^\alpha_{n, m} \le t \right\}, 
		\ \text{ and } \
		\underline{t}^\alpha_{n} = \Gamma^\alpha_{n, N^\alpha(n,t)} 
		, \\
		\Gamma^\beta_{n, n'} := \ssum{k}{n}{n'-1} \beta_{k}, \ N^\beta(n, t) := \max \left\{ m \ge n \colon \Gamma^\beta_{n,m} \le t \right\}, \ 
		\text{ and } \
		\underline{t}^\beta_{n} = \Gamma^\beta_{n, N^\beta(n,t)}.
	\end{align*}
	For $n' \le n$, we stipulate $\Gamma^\alpha_{n,n'} = \Gamma^\beta_{n,n'}  := 0 $.
\end{defn}

\begin{fact}
	Under Definition~\ref{defn:time-inter}, we have the following properties. 
	\begin{enumerate}[(i)]
		\item For $n_1 \le n_2 \le n_3$, $\Gamma^\alpha_{n_1, n_2} + \Gamma^\alpha_{n_2, n_3} = \Gamma^\alpha_{n_1, n_3}$. 
		\item 
		For $m \ge n$, $N^\alpha(n, \Gamma^\alpha_{n,m}) = m $.
		\item $\underline{t}^\alpha_{n} \le t <  \underline{t}^\alpha_{n} + \alpha_{N^\alpha(n,t)} $. When $t = \Gamma^\alpha_{n, m}$ for some $m \ge n$, $t = \underline{t}^\alpha_{n} $.
		\item  As $n \to \infty$, if $\alpha_{n} \to 0 $, then $\underline{t}^\alpha_{n} \uparrow t$.
	\end{enumerate}
	Similar properties hold if the superscript is replaced by $\beta$.
\end{fact}

In other words,
$\Gamma^\alpha_{n, n'}$ denotes the accumulation of the step sizes from $n$-th iteration to $(n'-1)$-th iteration.
$N^\alpha(n, t)$ is the maximal integer such that $\Gamma^\alpha_{n, N^\alpha(n, t)} \le t $.
We define this value by $\underline{t}^\alpha_{n} $,
which can be viewed as an approximation of $t$ from the left.

The following lemma characterizes the magnitude of $N^\alpha(n,t)$ and $N^\beta(n,t)$. 
\begin{lem}\label{lem:aux:bounded}
	Suppose that Assumption~\ref{assump:stepsize-twotime} holds.
	There exist two positive constants $c_\alpha$ and $c_\beta$ such that for any $t>0$ and $n \in \NB$,
	\begin{equation*}
		n \le N^\alpha(n,t) \le 2^{ \left\lceil 2t/c_\alpha \right\rceil + 1} n \quad \text{ and } \quad 
		n \le N^\beta(n,t) \le 2^{ \left\lceil 2t/c_\beta \right\rceil + 1} n.
	\end{equation*}
	As a corollary, 
	$\frac{\alpha_{n}}{ \alpha_{N^\alpha (n, t) } } 
	\le 2^{ \left\lceil 2t/c_\alpha \right\rceil + 1}$ and $\frac{\beta_{n}}{ \beta_{N^\beta (n, t) } } 
	\le 2^{ \left\lceil 2t/c_\beta \right\rceil + 1}$.
\end{lem}

Lemma~\ref{lem:aux:bounded} shows that under our assumptions, both $N^\alpha(n,t)$ and $N^\beta(n,t)$ are of the order $\Theta(n)$. This property is useful for establishing the tightness in Section~\ref{sec:fclt:tight}. 
For the polynomially diminishing step sizes, 
we can further sharpen this description, as stated in the following proposition.
\begin{prop}\label{prop:aux:step-size}
    If $\alpha_n = \frac{\alpha_0}{(n+1)^\alpha}$ and $\beta_n = \frac{\beta_0}{(n+1)^\beta}$ with $0 < a,b \le 1$, then $N^\alpha(n,t) - n = \Theta(n^\alpha)$ and $N^\beta(n,t) - n = \Theta(n^\beta)$. 
\end{prop}

The proof of Lemma~\ref{lem:aux:bounded} and Proposition~\ref{prop:aux:step-size} is deferred to Appendix~\ref{append:construct-details-proof}.

\begin{fact}\label{fact2}
	Under Definition~\ref{defn:time-inter},
        $\xbar_{n}(\cdot)$ defined in \eqref{eq:xbar-short}
    satisfies the following properties.
	\begin{enumerate}[(i)]
		\item For $m \ge n$, $ \xbar_{n}(\Gamma^\alpha_{n, m} ) = \xcheck_{m}$.
		\item 	$\xbar_{n}(t) = \xbar_{n}( \underline{t}_n^\alpha ) + \frac{t - \underline{t}^\alpha_{n}}{ \alpha_{N^\alpha(n,t)} } \big(  \xbar_{n}( \underline{t}_n^\alpha + \alpha_{N^\alpha(n,t)} ) - \xbar_{n}( \underline{t}_n^\alpha ) \big)$.
		\item \label{fact2-3} For $m \ge n$ and $t \in [\Gamma^\alpha_{n,m}, \Gamma^\alpha_{n,m+1})$, $\xbar_{n}(\underline{t}^\alpha_{n}) = \xcheck_{m}$.
	\end{enumerate}
	Similar results also hold for $\ybar_{n}(\cdot)$ and $\zbar_{n}(\cdot)$ defined in \eqref{eq:ybar-short} and \eqref{eq:zbar-short}.
\end{fact}

\begin{rema}[Discussion on the interpolation]
	For the way of interpolation in our works, \cite{faizal2023functional} also adopts the linear interpolation while \cite{gadat2018stochastic, liang2023asymptotic, blanchet2024limit} uses two different interpolations for different parts.
	Since both $\alpha_{N^\alpha(n,t)}$ and $\beta_{N^\beta(n,t)}$ converge to $0$ as $n \to \infty$, different interpolation methods make no difference in the asymptotic sense in our setup. We choose simple linear interpolation for analytical simplicity such that each trajectory is continuous in time.
\end{rema}

\subsection{An Upper Bound for the Exponential of a Hurwitz Matrix}
\label{append:fclt-interpre}

\begin{prop}\label{prop:hurwitz-expbound}
Suppose that $-A$ is a Hurwitz matrix, i.e., the real parts of the eigenvalues of $-A$ are negative. Then there exist $M,b > 0$ such that $\|e^{-At}\| \le M e^{-b t},\forall t\ge0$.
In particular, define $P := \int_0^\infty e^{-A^\top t} e^{-At} dt $.
Then $M$ and $b$ can be set as
$M := \sqrt{\lambda_{\max}(P)/\lambda_{\min}(P)}$ and $b := 1/(2\lambda_{\max}(P))$, where $\lambda_{\max}(P)$ and $\lambda_{\min}(P)$ denote the largest and smallest eigenvalues of $P$, respectively.
\end{prop}

Proposition~\ref{prop:hurwitz-expbound} shows that the decay of $\|e^{-At}\|$ is governed by the condition number and the largest eigenvalue of the Lyapunov matrix $P$. By definition, $P$ is positive definite.
In the special case where $A$ is symmetric positive definite, we have
$\lambda_{\max}(P)=\frac{1}{2\lambda_{\min}(A)}$ and
$\lambda_{\min}(P)=\frac{1}{2\lambda_{\max}(A)}$.
Consequently,
$M=\sqrt{\lambda_{\max}(A)/\lambda_{\min}(A)}$ and $b=\lambda_{\min}(A)$.
Therefore, for any $\eps>0$, if
$t \ \ge\  \lambda_{\min}(A)^{-1}\,
\log\!\left(\frac{\sqrt{\lambda_{\max}(A)/\lambda_{\min}(A)}}{\eps}\right)$,
then $\|e^{-At}\|\le \eps$.
Thus, larger $\lambda_{\min}(A)$ (and a smaller condition number $\lambda_{\max}(A)/\lambda_{\min}(A)$) corresponds to faster convergence to stationarity.
The proof of Proposition~\ref{prop:hurwitz-expbound} is deferred to Appendix~\ref{append:construct-details-proof}.

\subsection{Details for Examples}
\label{append:exmp}

For Assumption~\ref{assump:noise-fclt}, the dominant source of stochasticity in the optimization/RL examples is mini-batch (or sample) noise, which generally depends on the current iterate $(x_n,y_n)$.
The moment conditions in Assumption~\ref{assump:noise-fclt} are ensured when the iterates remain in a bounded region; this is standard in stochastic approximation and can be enforced, when needed, by a projected (or truncated) variant of the algorithm.
Such a modification does not change the local asymptotic behavior as long as $(x^\star,y^\star)$ lies in the interior of the projection set, so the projection is inactive in a neighborhood of the limit.
Moreover, mean-square convergence implies $(x_n,y_n)\to(x^\star,y^\star)$ in probability; combined with standard continuity/uniform integrability conditions on the noise moments, this yields convergence in probability of the corresponding conditional covariance quantities.

For Assumption~\ref{assump:mart-decouple-rate}, the verification for Example~\ref{exmp:PRave} is provided in the main text.
For the Examples~\ref{exmp:momentum} and \ref{exmp:gtd}, the strong monotonicity conditions required in Proposition~\ref{prop:mart-decouple-rate-formal} are satisfied as well.
Moreover, as discussed after Proposition~\ref{prop:mart-decouple-rate-formal}, the conclusion continues to hold with appropriately chosen step sizes.

\subsubsection{Details for Example~\ref{exmp:gtd}}
\label{append:exmp-gtd}

Consider an infinite-horizon MDP $\MM=(\SM, \AM, \PM ,R, \gamma)$ with discounted rewards, where $\SM$ and $\AM$ denote the (finite) state space and action space, respectively, and $\gamma \in [0,1)$ is the discount factor.
The transition kernel $\PM \colon \SM \times \AM \rightarrow \Delta(\SM)$ specifies the probability distribution over the next states for each state-action pair $(s,a) \in \SM \times \AM$, where $\PM(\cdot | s,a) \in \Delta(\SM) $ denotes the distribution of transitions from state $s$ when action $a$ is executed. The (possibly random) reward function $R\colon \SM \times \AM \rightarrow [0,1]$ represents the immediate reward received from state $s$ when action $a$ is taken. 

A policy $\pi\colon \SM \rightarrow \Delta(\AM)$ maps each state to a probability distribution over actions. Its quality is assessed via the value function  $$V^\pi(s) := \EB \bigg[ \sum_{t=0}^\infty \gamma^t R(s_t, a_t) \,|\, s_0 = s,\ a_t \sim \pi(\cdot|a_t),\ s_{t+1} \sim \PM(\cdot|s_t,a_t) \bigg], $$
which gives the expected discounted cumulative reward for following policy $\pi$ from initial state $s$.  
The process of estimating $V^\pi$ for a fixed policy $\pi$ is known as \textit{policy evaluation}.
When the data is generated under a different policy $\pi_b$ to establish the value function $V^\pi$ of $\pi$, the task is referred to as \textit{off-policy evaluation}~\citep{uehara2022review}.
When $|\SM|$ is large, directly estimating $V^\pi$ becomes computationally infeasible. Instead, $V^\pi$ can be approximated using a parameterized function. A straightforward approach is linear function approximation, where
$V^\pi(s) \approx V(s;y) = \phi(s)^\top y$, with $\phi(s)$ representing the feature vector of state $s$~\citep{dann2014policy}.

To find the optimal parameter $y$, both GTD2 are TDC are derived by minimizing the
\textit{mean squared projected Bellman error} through SGD~\cite{sutton2009fast}.
In Example~\ref{exmp:gtd},
given i.i.d. samples $\{ (s_n, a_n, R(s_n, a_n), s'_{n}) \}_{n=0}^\infty$ with $a_n \sim \pi_b(\cdot | s_n)$ and $s'_{n} \sim \PM(\cdot|s_{n}, a_{n})$, we define $r_{n} = R(s_n, a_n)$, and $\rho_{n} = \frac{\pi(a_n|s_n)}{\pi_b(a_n|s_n)}$.
Note in this setting $s'_n$ does not necessarily equal $s_{n+1}$.
The importance weight $\rho_{n}$ adjusts for the mismatch between $\pi$ and $\pi_b$.
This weight does not appear in the original formulations of GTD2 and TDC in~\cite{sutton2009fast} due to the use of a sub-sampling procedure~\citep{maei2011gradient}.
We incorporate this weight in update rules \eqref{eq:gtd-x}--\eqref{eq:tdc-y} to align with subsequent works, e.g.,~\cite{xu2019two, li2024high}.

As discussed in Section~\ref{sec:fclt:main}, deriving $\widetilde{\Sigma}_{\psi}$ in \eqref{eq:fclt-twotime:y-sde} requires analyzing the asymptotic covariance of $\psi_{n} - B_2 B_1^{-1} \xi_{n} $.
Suppose $(x_n, y_n)$ converges to $(x^\star, y^\star)$ in probability, as guaranteed by the results in~\cite{maei2011gradient}. For GTD2, we have $B_1 = C$ and $B_2 = -A^\top$, leading to
\begin{align*}
	\psia_{n} - B_2 B_1^{-1} \xi_{n} 
	& = -(A_{n}^\top {-} A^\top) x_{n} + A^\top C^{-1} [ (C_n {-} C) x_{n} + (A_n {-} A)y_{n} - (b_{n} {-} b) ] \\
	& = A^\top C^{-1} ( A_{n} A^{-1} b - b_{n} ) + o_p(1)
\end{align*}

For TDC, $B_1 = C$ and $B_2 = D^\top = C-A^\top$.
Then
\begin{align*}
	\psib_{n} - B_2 B_1^{-1}\xi_{n} 
	& = (D_n^\top {-} D^\top) x_{n} + (A_{n} {-} A ) y_{n} - (b_{n} {-} b) \\
	& \quad  \ - (C - A^\top) C^{-1} [ (C_n {-} C) x_{n} + (A_n {-} A)y_{n} - (b_{n} {-} b) ] \\
	& = A^\top C^{-1} ( A_{n} A^{-1} b - b_{n} ) + o_p(1).
\end{align*}
Thus, both GTD2 and TDC share the same $\widetilde{\Sigma}_{\psi}$.

\subsection{Q-learning with Polyak-Ruppert averaging}
\label{append:exmp-Q}

We continue with the notation in Appendix~\ref{append:exmp} and restrict our attention on finite state space and action space.
For a fixed policy $\pi$, its Q-function is defined as
$$Q^\pi(s,a) := \EB \bigg[ \sum_{t=0}^\infty \gamma^t R(s_t, a_t) \,|\, s_0 = s,\ a_0=a,\ s_{t+1} \sim \PM(\cdot|s_t,a_t), \ a_{t+1} \sim \pi(\cdot|a_{t+1}) \bigg], $$
which is the expected discounted cumulative reward for following policy $\pi$ from initial state $s$ and initial action $a$.
For fixed $s$ and $a$, define the expectation of reward as $r(s,a) = \EB [R(s,a)]$.
The optimal Q-function is defined as $Q^\star = \max_{\pi} Q^\pi(s,a)$ and is the unique fixed-point of the Bellman optimality  equation
$$Q(s,a) = r(s,a) + \gamma \EB _{s' \sim \PM(\cdot|s,a)} \max_{a'\ \in \AM} Q(s',a')=:(\TM Q)(s,a).$$
Viewing $Q$ as a vector in $\RB^{|\SM||\AM|}$, we define the Bellman optimality operator $\TM\colon \RB^{|\SM||\AM|} \to \RB^{|\SM||\AM|}$ as the right-hand side, then $Q^\star$ is the unique fixed point of $\TM$, i.e., $Q^\star = \TM Q^\star$. 

Q-learning is a model-free algorithm to seek the optimal $Q^\star$~\citep{watkins1989learning}.  Similar to Example~\ref{exmp:PRave}, Polyak-Ruppert averaging could improve the convergence rate~\citep{li2021polyak}.
We focus on the synchronous setting where a generative model produces independent samples for all state-action pairs in every iteration~\citep{kearns2002sparse}.
Then tabular Q-learning with PR averaging can be written as
\begin{equation}
\label{eq:q-ave}
\begin{aligned}
    x_{n+1} &= x_n - \alpha_n\big(x_n - \TM x_n + \xi_n\big),
\qquad \alpha_n = \alpha_0 (n+1)^{-a},\ \ a\in(1/2,1),\\
y_{n+1} &= \frac{1}{n+1}\sum_{\tau=0}^{n} x_{\tau}
\;=\; y_n - \beta_n (y_n-x_n),
\qquad \beta_n = \frac{1}{n+1}.
\end{aligned}
\end{equation}
The recursion \eqref{eq:q-ave} is therefore an instance of two-time-scale SA with $F(x,y) = x -\TM x$ and $G(x,y) = y-x$.
Consequently, $x^\star=y^\star=Q^\star$, and the unique solution of $F(x,y)=0$ is given by
$H(y)\equiv Q^\star$, where $Q^\star$ is the unique fixed point of $\TM$.
Since $G$ is linear, it suffices to verify the local linearity condition for $F$. 

To this end, we introduce the greedy policy induced by $Q$: for any $Q$, define $\pi_Q(\cdot\mid s)$ by
$\pi_Q(a\mid s)=\mathbf{1}\big\{a\in\arg\max_{a'\in\AM} Q(s,a')\big\}$.
When the argmax is not unique, we break ties arbitrarily (e.g., at random) to obtain a deterministic policy.
Let $\pi^\star:=\pi_{Q^\star}$ denote the (deterministic) optimal policy.
Viewing $r$ as a vector in $\RB^{|\SM||\AM|}$, we can write
\[
\TM Q = r + \gamma P^{\pi_Q}Q,
\qquad
P^{\pi_Q}_{(s,a),(s',a')} := \PM(s'\mid s,a)\,\pi_Q(a'\mid s').
\]
If the optimal policy $\pi^\star$ is unique, then there exists $L>0$ such that, for any $Q$,
\[
\big\|(P^{\pi_Q}-P^{\pi^\star})(Q-Q^\star)\big\|_\infty
\le L\|Q-Q^\star\|_\infty^2
\quad \text{\cite[Lemma~B.1]{li2021polyak}}.
\]
Intuitively, the uniqueness of $\pi^\star$ implies that, whenever $Q$ is sufficiently close to $Q^\star$,
the greedy policy $\pi_Q$ coincides with $\pi^\star$; since the set of deterministic policies is finite,
the Bellman optimality operator becomes locally affine around $Q^\star$.

For $F(x,y)=x-\TM x$, we have $H(y)\equiv Q^\star$ and hence
\begin{equation*}
    F(x,y)
    = x-\TM x
     = (x-Q^\star) + (\TM Q^\star-\TM x) 
    = (x-Q^\star) + \gamma P^{\pi^\star}Q^\star - \gamma P^{\pi_x}x,
\end{equation*}
where $\pi_x$ denotes the greedy policy induced by $x$ (with a fixed tie-breaking rule), and we use $\TM Q = r + \gamma P^{\pi_Q}Q$.
Let $B_1:= I-\gamma P^{\pi^\star}$.
We claim that
\begin{equation}\label{eq:q-learning-local-linear}
    \|F(x,y)-B_1(x-Q^\star)\|_\infty \le \gamma L \|x-Q^\star\|_\infty^2.
\end{equation}
Since all norms are equivalent in finite dimensions, the same bound also holds for the Euclidean norm up to a dimension-dependent constant.
We defer the verification of \eqref{eq:q-learning-local-linear} later.

With \eqref{eq:q-learning-local-linear}, Assumptions~\ref{assump:smooth} -- \ref{assump:hurwitz} holds with $\Fsmooth = \Gsmooth = \Hsmooth = 1$, $B_1 = 1 - \gamma P^{{\pi^\star}}$, $B_2 = -\mI $ and $B_3 = \mI$.
Since $\gamma \in [0,1)$ and $P^{{\pi^\star}}$ is a transition matrix, the eigenvalues of $B_1$ have positive real parts and $-B_1$ is Hurwitz.
By \cite[Theorem~2]{li2023online}, we can obtain $\EB\| x_n - Q^\star \|^4 = \OM(\alpha_n^2)$ and conquently $\EB\| x_n - Q^\star \|^2 = \OM(\alpha_n)$.
A remaining issue is whether one can further show $\EB|y_n - Q^\star|^2 = \OM(1/n)$.
Since \cite{li2021polyak} establishes the asymptotic distribution of $\sqrt{n},(y_n - Q^\star)$ and the rate $\EB|y_n - Q^\star|_\infty = \OM(n^{-1/2})$,  we believe this gap can be filled.

This example is similar to Example~\ref{exmp:PRave}, except that $F(x,y)$ is not the gradient of an objective function and $B_1$ is not symmetric. Moreover, unlike \cite{li2021polyak}, which establishes the FCLT for the partial-sum process via a uniform rescaling, our results yield an alternative characterization based on an iteration-dependent scaling. Following the analysis in Example~\ref{exmp:PRave}, one can also obtain that, within the family of weighted averages with $\beta_n=\frac{\beta_0}{n+1}$, the vanilla average ($\beta_0=1$) is asymptotically optimal in the sense of minimizing the estimator’s asymptotic covariance.

Now we verify \eqref{eq:q-learning-local-linear}. Note that
\begin{equation*}
    F(x,y)-B_1(x-Q^\star)
    = \bigl[(x-\TM x) - (I-\gamma P^{\pi^\star})(x-Q^\star)\bigr] 
    = \gamma\bigl(P^{\pi^\star}-P^{\pi_x}\bigr)x.
\end{equation*}
For each $(s,a)$, we have
\begin{align*}
    \bigl[(P^{\pi^\star}-P^{\pi_x})x\bigr](s,a)
    &= \sum_{s'\in\SM}\sum_{a'\in\AM} \PM(s'\mid s,a)\,x(s',a')
       \bigl(\pi^\star(a'\mid s')-\pi_x(a'\mid s')\bigr)\\
    &= \sum_{s'\in\SM} \PM(s'\mid s,a)\Bigl(
        \sum_{a'\in\AM}\pi^\star(a'\mid s')x(s',a') - \max_{a'\in\AM}x(s',a')
       \Bigr)
    \le 0,
\end{align*}
by the definition of the greedy policy.
Therefore, $F(x,y)-B_1(x-Q^\star)\le 0$ elementwise.
Moreover, we can decompose
\[
    (P^{\pi^\star}-P^{\pi_x})x
    = (P^{\pi^\star}-P^{\pi_x})(x-Q^\star) + (P^{\pi^\star}-P^{\pi_x})Q^\star.
\]
By the same argument as above, $(P^{\pi^\star}-P^{\pi_x})Q^\star\ge 0$ elementwise.
Combining the two elementwise inequalities yields
\[
    \bigl|(P^{\pi^\star}-P^{\pi_x})x\bigr|
    \le (P^{\pi_x}-P^{\pi^\star})(x-Q^\star)
    \quad\text{elementwise},
\]
and hence
\[
    \|F(x,y)-B_1(x-Q^\star)\|_\infty
    \le \gamma\bigl\|(P^{\pi_x}-P^{\pi^\star})(x-Q^\star)\bigr\|_\infty
    \le \gamma L \|x-Q^\star\|_\infty^2,
\]
where the last inequality follows from \cite[Lemma~B.1]{li2021polyak}. 

Finally, we note that a closely related example is SSP Q-learning, which is designed for average-reward MDPs \citep{abounadi2001learning}. SSP Q-learning can also be interpreted as a two-time-scale SA scheme, where the fast iterate corresponds to the Q-function and the slow iterate tracks a reference (baseline) variable. This example has been analyzed within the two-time-scale SA framework in \cite{chandak2025finite}. Verifying whether SSP Q-learning satisfies our local linearity condition is left for future work. 

%% file: tex/append_proof.tex
\section{Omitted Proofs for the Main Theorem}
\label{append:proof-full}

\subsection{Proofs in Section~\ref{sec:fclt:one-step}}
\label{sec:fclt:proof:construct}
In this subsection, we give the proof of the one-step recursions in Section~\ref{sec:fclt:one-step}.
Since the proofs of both Lemmas~\ref{lem:check-twotime-xy} and \ref{lem:check-twotime-z} are closely related, we combine their proof together.
\begin{proof}[Proof of Lemmas~\ref{lem:check-twotime-xy} and \ref{lem:check-twotime-z}]
	The proof is structured into three parts:  we derive the one-step recursion first for $\ycheck_{n}$, followed by $\xcheck_{n}$, and finally $\zcheck_{n}$.
	To begin, we take the update rules for $\xhat_{n}$ and $\yhat_{n}$, as outlined below.
	\begin{align}
		\xhat_{n+1} & = \xhat_n - \alpha_n F(x_n, y_n) - \alpha_n \xi_n + H(y_n) - H(y_{n+1}), \label{eq:xhat} \\
		\yhat_{n+1} & = \yhat_n - \beta_n G(H(y_n), y_n) + \beta_n ( G(H(y_n), y_n) - G(x_n, y_n) ) - \beta_n \psi_n. \label{eq:yhat}
	\end{align}
	
	\textbf{Part 1: One-step recursion for $\ycheck_{n}$}.
	By Proposition~\ref{prop:operator-decom}, we can rewrite the update rule of $\yhat_{n}$ in \eqref{eq:yhat} as 
	\begin{equation}
		\label{eq:yhat-decom}
		\yhat_{n+1} = (I - \beta_{n} B_3) \yhat_{n} - \beta_{n} B_2 \xhat_{n} - \beta_n \psi_n - \beta_n R_n^G.
	\end{equation}
	Then from the definition of $\ycheck_{n}$ and \eqref{eq:yhat-decom}, we have 
	\begin{align}
		\ycheck_{n+1}
		& = \sqrt{ \frac{ \beta_{n-1} }{ \beta_{n} } } (\mI - \beta_{n} B_3) \ycheck_{n} - \sqrt{ \beta_{n} \alpha_{n-1} } B_2 \xcheck_{n} - \sqrt{\beta_{n}} \psi_{n} - \sqrt{\beta_{n}} R_{n}^G \nonumber \\
		& = \left( \mI - \beta_{n} B_3 + \frac{ \invdiffslow \beta_{n} }{2} \mI \right) \ycheck_{n} - \sqrt{\beta_{n} \alpha_{n} } B_2 \xcheck_{n} - \sqrt{\beta_n} \psi_{n} + R_{n}^y, \label{eq:lem:check-twotime:y-update}
	\end{align}
	where
	\begin{align}\label{eq:lem:check-twotime:Ry}
		\begin{split}
			R_{n}^y
			& =  \underbrace{ \left( \sqrt{ \frac{ \beta_{n-1} }{\beta_{n}} } - 1 - \frac{ \invdiffslow \beta_{n} }{2} \right) \ycheck_{n} }_{=: R_{n,1}^y}  
			+ \underbrace{ ( \beta_{n} - \sqrt{ \beta_{n} \beta_{n-1} } ) B_3 \ycheck_{n} 
				\vphantom{ \left( \sqrt{ \frac{ \beta_{n-1} }{\beta_{n}} } \right) } }_{ =:R_{n,2}^y } \\
			& \quad \  + \underbrace{
				( \sqrt{ 
					\beta_{n} 
					\alpha_{n} } - \sqrt{ 
					\beta_{n} 
					\alpha_{n-1} } ) B_2 \xcheck_{n}
				\vphantom{ \left( \sqrt{ \frac{ \beta_{n-1} }{\beta_{n}} } \right) } }_{=: R_{n,3}^y } 
			- \underbrace{ \sqrt{\beta_{n}} R_{n}^G 
				\vphantom{ \left( \sqrt{ \frac{ \beta_{n-1} }{\beta_{n}} } \right) } }_{ =: R_{n,4}^y }.
		\end{split}
	\end{align}
	By Assumption~\ref{assump:stepsize-twotime}, we have
	\begin{align}
			\sqrt{ \frac{ \beta_{n-1} }{ \beta_{n} } } 
			& = \sqrt{ 1 + \beta_{n-1} ( \beta_{n}^{-1} - \beta_{n-1}^{-1} ) } = \sqrt{1 + \invdiffslow \beta_{n-1}  + o(\beta_{n-1}) } \nonumber \\
			& = \sqrt{1 + \invdiffslow \beta_{n}  + o(\beta_{n}) }
			= 1 + \frac{ \invdiffslow \beta_{n} }{2} + o(\beta_{n}), \nonumber \\
			\sqrt{ \frac{ \alpha_{n-1} }{\alpha_{n}} } 
			& = 1 + \OM(\beta_n),
			\qquad 
			\sqrt{ \frac{ \kappa_{n-1} }{\kappa_{n}} } 
			= 1 + \OM(\beta_n). \label{eq:lem:check-twotime:sqrtalpha}
	\end{align}
	It follows that $\| R_{n,1}^y \| +  \| R_{n,2}^y \| = o(\beta_{n}) \| \ycheck_{n} \| = o(\beta_{n}^{1/2}) \| \yhat_{n} \|  $
	and  $\| R_{n,3}^y \| = \OM(\beta_{n}^{3/2} \alpha_{n}^{1/2}) \| \xcheck_{n} \| = o(\beta_{n}^{3/2}) \| \xhat_{n} \| $. 
	Meanwhile, by Proposition~\ref{prop:operator-decom}, we have $ \| R_{n,4}^y \| = \OM(\beta_{n}^{1/2}) ( \| \xhat_{n} \|^{1+\Gsmooth} + \| \yhat_{n} \|^{1+\Gsmooth} ) $.
	Note that in Assumption~\ref{assump:stepsize-twotime}, we require $\alpha_{n}^{1+\Gsmooth} = o(\beta_{n}) $.
	Then we have
	\begin{equation}\label{eq:lem:check-twotime:Ry-upper}
		\| R_{n}^y \| = o(\beta_{n}^{1/2}) \| \yhat_{n} \| + o(\beta_{n}^{3/2}) \| \xhat_{n} \| + 
		\OM(\beta_{n}) \,
		\frac{ \alpha_{n}^{ (1+\Gsmooth) / 2 } }{ \beta_{n}^{1/2} } \cdot
		\frac{ \| \xhat_{n} \|^{1+\Gsmooth} + \| \yhat_{n} \|^{1+\Gsmooth} } {\alpha_{n}^{(1+\Gsmooth)/2} }.
	\end{equation}
	By Assumption~\ref{assump:mart-decouple-rate}, it follows that $ \EB \| \EB[ R_{n}^y | \FM_{n} ] \| \le \EB \| R_{n}^y \| = o(\beta_{n})$ and $\EB \| R_{n}^y \|^2 = o(\beta_{n}^2)$.
	Moreover, by Proposition~\ref{prop:FG-reduce-order}, we have that \eqref{eq:lem:check-twotime:Ry-upper} holds with $\Gsmooth$ replaced by $\Gsmooth / 2$.
	For $2 \le p \le \frac{4}{1+\Gsmooth / 2}$, it follows that
	\begin{equation*}
		\EB \| R_{n}^y \|^p
		= o(\beta_{n})
		+ o( \beta_{n} \cdot  \beta_{n}^{p-1} / \alpha_{n}^{ p \Gsmooth / 2} )
		= o(\beta_{n}),
	\end{equation*}
	where the last step is due to $p-1 \ge p/2 \ge p \Gsmooth / 2$.
	\vspace{0.2cm}
	
	\textbf{Part 2: One-step recursion for $\xcheck_{n}$}.
	By Proposition~\ref{prop:operator-decom}, the update rule of $\xhat_{n}$ in \eqref{eq:xhat} can be rewritten as
	\begin{align}
		\xhat_{n+1} & = (\mI - \alpha_{n} B_1) \xhat_{n} + H^\star (\yy_{n} - \yy_{n+1}) - \alpha_{n} \xi_{n} - R_{n}^{\nablaH} - R_{n}^{H} - \alpha_{n} R_{n}^F \nonumber \\ 
		\begin{split}
			& \overset{\eqref{eq:yhat-decom}}{=} (\mI - \alpha_{n} B_1 + \beta_{n} H^\star B_2) \xhat_{n} + \beta_n H^\star B_3 \yhat_{n} - \alpha_n \xi_{n} + \beta_{n} H^\star \psi_{n} \\
			& \qquad - R_{n}^{\nablaH} - R_{n}^H - \alpha_{n} R_{n}^F + \beta_n H^\star R_{n}^G. 
		\end{split}
		\label{eq:xhat-decom}
	\end{align}
	Then from the definition of $\xcheck_{n}$, we have
	\begin{align}
		\xcheck_{n+1}
		& = \sqrt{ \frac{\alpha_{n-1}}{\alpha_{n}} } (\mI - \alpha_{n} B_1 + \beta_n H^\star B_2) \xcheck_{n} + \sqrt{ \frac{\beta_{n-1}}{\alpha_{n}} } \beta_{n} H^\star B_3 \ycheck_{n} - \sqrt{ \alpha_{n} } \xi_{n} + \frac{ \beta_{n} }{ \sqrt{ \alpha_{n} } } H^\star \psi_{n} \nonumber \\
		& \qquad - \frac{1}{ \sqrt{\alpha_{n}} } R_{n}^{\nablaH} - \frac{1}{\sqrt{\alpha_{n}}} R_{n}^{H} - \sqrt{\alpha_{n}} R_{n}^F + \frac{\beta_{n}}{\sqrt{ \alpha_{n} }} H^\star R_{n}^G \nonumber \\
		& = (\mI - \alpha_{n} B_1) \xcheck_{n} - \sqrt{\alpha_{n}} \tilde{\xi}_{n} + R_{n}^x, \label{eq:lem:check-twotime:x-update}
	\end{align}
	where $\tilde{\xi}_{n} := \xi_{n} - \kappa_{n} H^\star \psi_{n} = \xi_{n} - \frac{\beta_{n}}{\alpha_{n}} H^\star \psi_{n}$ and 
	\begin{align}\label{eq:lem:check-twotime:Rx}
		\begin{split}
			R_{n}^x
			& = \underbrace{ \left( \sqrt{ \frac{\alpha_{n-1}}{\alpha_{n}} } - 1 \right) \xcheck_{n} }_{=: R_{n,1}^x} 
			+ \underbrace{ (\alpha_{n} - \sqrt{ \alpha_{n} \alpha_{n-1} }) B_1 \xcheck_{n} 
				\vphantom{\left( \sqrt{ \frac{\alpha_{n-1}}{\alpha_{n}} } \right) } }_{=: R_{n,2}^x} 
			+ \underbrace{ \sqrt{\frac{ \alpha_{n-1} }{ \alpha_{n} } } \beta_{n} H^\star B_2 \xcheck_{n} 
				\vphantom{\left( \sqrt{ \frac{\alpha_{n-1}}{\alpha_{n}} } \right) }}_{=: R_{n,3}^x} \\
			& \qquad 
			+ \underbrace{ \sqrt{ \frac{\beta_{n-1}}{\alpha_{n}} } \beta_{n} H^\star B_3 \ycheck_{n} 
				\vphantom{\left( \sqrt{ \frac{\alpha_{n-1}}{\alpha_{n}} } \right) } }_{=: R_{n,4}^x} 
			- \underbrace{ \frac{1}{ \sqrt{\alpha_{n}} } R_{n}^{\nablaH} }_{=: R_{n,5}^x} 
			- \underbrace{ \frac{1}{\sqrt{\alpha_{n}}} R_{n}^{H} }_{=: R_{n,6}^x} 
			- \underbrace{ \sqrt{\alpha_{n}} R_{n}^F 
				\vphantom{\frac{\beta_{n}}{\sqrt{ \alpha_{n} } } } }_{=:R_{n,7}^x } 
			+ \underbrace{ \frac{\beta_{n}}{\sqrt{ \alpha_{n} } } H^\star R_{n}^G }_{=:R_{n,8}^x}.
		\end{split}
	\end{align}
	By \eqref{eq:lem:check-twotime:sqrtalpha}, we have
	$\| R_{n,1}^x \| + \| R_{n,2}^x \| = \OM(\beta_{n}) \| \xcheck_{n} \| = \OM(\beta_{n} \alpha_{n}^{-1/2} ) \| \xhat_{n} \| $.
	By Assumption~\ref{assump:stepsize-twotime}, we have
	$\| R_{n,3}^x \| + \| R_{n,4}^x \| = \OM(\beta_{n}) \| \xcheck_{n} \| + \OM ( \beta_{n}^{3/2} \alpha_{n}^{-1/2} ) \| \ycheck_{n} \|
	= \OM(\beta_{n} \alpha_{n}^{-1/2} ) ( \| \xhat_{n} \| + \| \yhat_{n} \| ) $.
	For $R_{n,5}^x$ and $R_{n,6}^x$,
	from the definition of $\FM_{n}$, Assumption~\ref{assump:smooth} and Proposition~\ref{prop:operator-decom}, we have
	\begin{align*}
		\| \EB [ R_{n}^{\nablaH} | \FM_{n} ] \| 
		& = \| (\nablaH (\yy_{n}) - \nablaH (\yy^\star)) \EB[ \yy_{n+1} - \yy_{n} | \FM_{n} ] \| \\
		& = \beta_{n} \| (\nablaH (\yy_{n}) - \nablaH (\yy^\star)) G(\xx_{n}, \yy_{n}) \| 
		\le \beta_{n} \Hholder \|  \yhat_{n} \|^\Hsmooth \| G(\xx_{n}, \yy_{n}) \| \\
		\| R_{n}^{\nablaH} \| 
		& \le \Hholder \| \yhat_{n} \|^{\Hsmooth}  \| \yy_{n+1} - \yy_{n} \|  , \\
		\| R_{n}^H \| 
		& = \OM ( \| \yy_{n+1} - \yy_{n} \|^{1+\Hsmooth} ), \\
		\| G(\xx_{n}, \yy_{n}) \|
		& = \|  G(\xx_{n}, \yy_{n}) - G( H(\yy_{n}), \yy_{n} ) + G( H(\yy_{n}), \yy_{n} ) - G( H(\yy^\star), \yy^\star ) \| \\
		& \le  \LGx \| \xhat_{n} \| +  \LGy \| \yhat_{n} \|, \\
		\| \yy_{n+1} - \yy_{n} \| 
		& \le \beta_{n} \| G(\xx_{n}, \yy_{n}) \| + \beta_{n} \| \psi_{n} \|.
	\end{align*}
	Then by Young's inequality and Jensen's inequality, we have
	\begin{align*}
		\| \EB[  R_{n,5}^x  | \FM_n ] \| & = \OM (\beta_{n} \alpha_{n}^{-1/2}) ( \| \xhat_{n} \|^{1+\Hsmooth} + \| \yhat_{n} \|^{1+\Hsmooth} ), \\
		\| R_{n,5}^x \| 
		& = \OM (\beta_{n} \alpha_{n}^{-1/2}) ( \| \xhat_{n} \|^{1+\Hsmooth} + \| \yhat_{n} \|^{1+\Hsmooth} + \| \yhat_{n} \|^\Hsmooth \| \psi_{n} \| ), \\
		\| R_{n,6}^x \| 
		& = \OM (\beta_{n}^{1+\Hsmooth} \alpha_{n}^{-1/2}) ( \| \xhat_{n} \|^{1+\Hsmooth} + \| \yhat_{n} \|^{1+\Hsmooth} + \| \psi_{n} \|^{1+\Hsmooth}  ).
	\end{align*}
	Moreover, Proposition~\ref{prop:operator-decom} also implies $\| R_{n,7}^x \| 
	= \OM(\alpha_{n}^{1/2}) ( \| \xhat_{n} \|^{1+\Fsmooth} + \| \yhat_{n} \|^{1+\Fsmooth} )$ and 
	$ \| R_{n,8}^x \| 
	= \OM (\beta_{n} \alpha_{n}^{-1/2}) ( \| \xhat_{n} \|^{1+\Gsmooth} + \| \yhat_{n} \|^{1+\Gsmooth} ) $.
	Summarizing the above results, we obtain
	\begin{align}\label{eq:lem:check-twotime:Rx-upper}
		\begin{split}
			\| \EB [ R_{n}^x | \FM_n ] \|
			& = \OM( \beta_{n} \alpha_{n}^{-1/2})  (\| \xhat_{n} \| + \| \yhat_{n} \| ) + \OM (\beta_{n} \alpha_{n}^{-1/2}) ( \| \xhat_{n} \|^{1+\Hsmooth} + \| \yhat_{n} \|^{1+\Hsmooth} ) \\
			& \quad \ + \OM (\beta_{n}^{1+\Hsmooth} \alpha_{n}^{-1/2}) \| \psi_{n} \|^{1+\Hsmooth} + \OM(\alpha_{n}^{1/2}) ( \| \xhat_{n} \|^{1+\Fsmooth} + \| \yhat_{n} \|^{1+\Fsmooth} ) \\
			& \quad \ + \OM (\beta_{n} \alpha_{n}^{-1/2}) ( \| \xhat_{n} \|^{1+\Gsmooth} + \| \yhat_{n} \|^{1+\Gsmooth} ), \\
			\| R_{n}^x \| 
			& = \OM( \| \EB [ R_{n}^x | \FM_n ] \| ) + \OM (\beta_{n} \alpha_{n}^{-1/2}) \| \yhat_{n} \|^\Hsmooth \| \psi_{n} \|.
		\end{split}
	\end{align}
	Note that Assumption~\ref{assump:stepsize-twotime} requires $\alpha_{n}^{1+\Fsmooth} = o (\beta_{n})$, $\beta_{n}^{(1+\Hsmooth) / 2 } = \OM(\alpha_{n})$.
	By Assumption~\ref{assump:mart-decouple-rate}, we have
	$ \EB \|  \EB [ R_{n}^x | \FM_n ] \| = \OM( \beta_{n} + \beta_{n}^{1+\Hsmooth} \alpha_{n}^{-1/2} + \alpha_{n}^{1 + \Fsmooth/2}) = o(\sqrt{\beta_{n} \alpha_{n}}) $.
	Meanwhile, since $\yhat_{n} \in \FM_{n}$, Assumption~\ref{assump:noise-fclt} implies $\EB \| \yhat_{n} \|^{2\Hsmooth} \| \psi_{n} \|^2 = \OM( \EB \| \yhat_{n} \|^{2\Hsmooth})$.
	By Assumption~\ref{assump:mart-decouple-rate}, we have
	$ \EB \| R_{n}^x \|^2 = \OM( \beta_{n}^2 + \beta_{n}^{2+\Hsmooth} \alpha_{n}^{-1} + \alpha_{n}^{2 + \Fsmooth}) = \OM(\beta_n \alpha_n) $.
	Moreover, by Propositions~\ref{prop:H-reduce-order} and \ref{prop:FG-reduce-order}, we have that \eqref{eq:lem:check-twotime:Rx-upper} holds after replacing $\Hsmooth, \Fsmooth, \Gsmooth$ with $\Hsmooth/2, \Fsmooth/2, \Gsmooth/2$ respectively.
	For $2 \le p \le \frac{4}{1 + (\Hsmooth \vee \Fsmooth \vee \Gsmooth) / 2 }$, one can check $\EB \| R_{n}^x \|^p = o(\alpha_{n}) $.
	\vspace{0.2cm}
	
	\textbf{Part 3: One-step recursion for $\zcheck_{n}$}.
	From the definition of $\zcheck_{n}$ and the update rules of $\ycheck_{n}$ and $\xcheck_{n}$ in \eqref{eq:lem:check-twotime:y-update} and \eqref{eq:lem:check-twotime:x-update} , we have
	\begin{align*}
		\zcheck_{n+1}
		& = \ycheck_{n+1} - \sqrt{\kappa_{n}} B_2 B_1^{-1} \xcheck_{n+1} \\
		& = \left( \mI - \beta_{n} B_3 + \frac{\invdiffslow \beta_{n}}{2} \mI \right) (\zcheck_{n} + \sqrt{\kappa_{n-1}} B_2 B_1^{-1} \xcheck_{n})
		- \sqrt{ \beta_{n} \alpha_{n} } B_2 \xcheck_{n} - \sqrt{ \beta_{n} } \psi_{n} + R_{n}^y \\
		& \quad - \sqrt{\kappa_{n}} B_2 B_1^{-1} \left[ (\mI - \alpha_{n} B_1 ) \xcheck_{n} - \sqrt{\alpha_{n}} \tilde{\xi}_{n} + R_{n}^x \right] \\
		& = \left( \mI - \beta_{n} B_3 + \frac{\invdiffslow \beta_{n}}{2} \mI \right) \zcheck_{n} - \sqrt{\beta_{n}} \tilde{\psi}_{n} + R_{n}^z,
	\end{align*}
	where $\tilde{\psi}_{n} := \psi_{n} - B_2 B_1^{-1} \tilde{\xi}_n = (\mI + \kappa_{n} B_2 B_1^{-1} H^\star) \psi_{n} - B_2 B_1^{-1} \xi_{n}$ and
	\begin{equation*}
		R_{n}^z
		= \underbrace{ (\sqrt{\kappa_{n-1}} {-} \sqrt{\kappa_{n}}) B_2 B_1^{-1} \xcheck_{n} }_{=: R_{n,1}^z} 
		- \underbrace{ \sqrt{\kappa_{n-1}} \beta_{n} (B_3 {-} \invdiffslow \mI / 2) B_2 B_1^{-1} \xcheck_{n} }_{=: R_{n,2}^z}
		+ R_{n}^{y} - \sqrt{\kappa_{n}} B_2 B_1^{-1} R_{n}^x.
	\end{equation*}
	By \eqref{eq:lem:check-twotime:sqrtalpha}, we have $\| R_{n,1}^z \| + \| R_{n,2}^z \| = o(\beta_{n}) \| \xcheck_{n} \|$.
	Moreover, from the definitions of $R_{n}^y$ and $R_{n}^x$ in \eqref{eq:lem:check-twotime:Ry} and \eqref{eq:lem:check-twotime:Rx} and the analysis in Parts~1 and 2, we can obtain
	$\EB \| \EB [ R_{n}^z | \FM_{n} ] \| = o(\beta_{n}) $, $\EB \| R_{n}^z \|^2 = \OM (\beta_{n}^2)$ and $\EB \| R_{n}^z \|^p = o(\beta_{n})$ for $2 \le p \le \frac{4}{1 + ( \Hsmooth \vee \Fsmooth \vee \Gsmooth )/2}$.
\end{proof}

\subsection{Proofs in Section~\ref{sec:fclt:tight}}
\label{sec:fclt:proof:tight}
In this subsection, we give the proof of Proposition~\ref{prop:tight-suff-slow} and Lemma~\ref{lem:tight-twotime} in Section~\ref{sec:fclt:tight}.
The main challenge lies in establishing the tightness of $\{ \ybar_{n}(\cdot) \}_{n=1}^\infty$, which is detailed in Part~3 of this proof.

First, we emphaze that 

With Definition~\ref{defn:time-inter}, the construction in \eqref{eq:xbar-short}, \eqref{eq:ybar-short} and \eqref{eq:zbar-short} can be equivalently expressed as follows 

\begin{align*}\label{eq:def:xbar}\tag{X}
	\begin{split}
		\xbar_{n}(t) 
		& \overset{\eqref{eq:xbar-short}}{=} \xcheck_{ N^\alpha(n,t) } + \frac{t - \underline{t}^\alpha_{n}}{ \alpha_{N^\alpha(n,t)} } \big( \xcheck_{ N^\alpha(n,t) + 1} - \xcheck_{ N^\alpha(n,t)  } \big) \\
		& \overset{\eqref{eq:xcheck-twotime}}{=} \xcheck_{ N^\alpha(n,t) } + (t - \underline{t}^\alpha_{n} ) \left( - B_1 \xcheck_{N^\alpha(n,t)} + \frac{R_{N^\alpha(n,t)}^x}{ \alpha_{N^\alpha(n,t)} } \right) - \frac{t - \underline{t}^\alpha_{n}}{ \sqrt{ \alpha_{N^\alpha(n,t)} } } 
		\tilde{\xi}_{N^\alpha(n,t)}.
	\end{split} \\
	\label{eq:def:ybar}\tag{Y}
	\begin{split}
		\ybar_{n}(t) 
		& \overset{\eqref{eq:ybar-short}}{=} \ycheck_{ N^\beta(n,t) } + \frac{t - \underline{t}^\beta_{n}}{\beta_{N^\beta(n,t)}} \big( \ycheck_{ N^\beta(n,t)+1 } - \ycheck_{ N^\beta(n,t) } \big) \\
		& \overset{\eqref{eq:ycheck-twotime}}{=} \ycheck_{ N^\beta(n,t) } + (t - \underline{t}^\beta_{n} ) \left[ \Big(- B_3 + \frac{\invdiffslow \mI}{2} 
		\Big) \ycheck_{N^\beta(n,t)} + \frac{R_{N^\beta(n,t)}^y }{ \beta_{N^\beta(n,t)} } \right] - 
		\frac{t - \underline{t}^\beta_{n}} {\sqrt{ \beta_{N^\beta(n,t)} } }
		\psi_{N^\beta(n,t)} 
		\\
		& \qquad - (t - \underline{t}^\beta_{n}) \sqrt{ \frac{ \alpha_{N^\beta(n,t)} }{ \beta_{N^\beta(n,t)} } } B_2 \xcheck_{N^\beta(n,t)}.
	\end{split}\\
	\label{eq:def:zbar}\tag{Z}
	\begin{split}
		\zbar_{n}(t) 
		& \overset{\eqref{eq:zbar-short}}{=} \zcheck_{ N^\beta(n,t) } + \frac{t - \underline{t}^\beta_{n}}{\beta_{N^\beta(n,t)}} \big( \zcheck_{ N^\beta(n,t)+1 } - \zcheck_{ N^\beta(n,t) } \big) \\
		& \overset{\eqref{eq:zcheck-twotime}}{=} \zcheck_{ N^\beta(n,t) } + (t - \underline{t}^\beta_{n} ) \left[ \Big(- B_3 + \frac{\invdiffslow \mI}{2} 
		\Big) \zcheck_{N^\beta(n,t)} + \frac{R_{N^\beta(n,t)}^z }{ \beta_{N^\beta(n,t)} } \right] -  \frac{t - \underline{t}^\beta_{n}}{ \sqrt{ \beta_{N^\beta(n,t)}} } 
		\tilde{\psi}_{N^\beta(n,t)}.
	\end{split}
\end{align*}

\begin{proof}[Proof of Proposition~\ref{prop:tight-suff-slow}]
For each $U_n(\cdot)$, if we restrict the domain as $[0, T]$ for a fixed $T > 0$, then $U_n(\cdot)|_{[0,T]}$ is also a random function in $C( [0,T]; \RB^d)$.
Equip $C( [0,T]; \RB^d)$ with the norm $\rho_{T} (f, g) := \max_{0 \le t \le T} \| f(t) - g(t) \|$, we can similarly define the tightness of $\left\{ U_n(\cdot)|_{[0,T]} \right\}_{n=0}^\infty$ as Definition~\ref{defn:tight}.
\cite[Corollary~5]{whitt1970weak} guarantees that $\{U_n(\cdot)\}_{n=0}^\infty$ is tight in $C( [0,\infty); \RB^d)$ if and only if for each integer $j \ge 1$, $\left\{ U_n(\cdot)|_{[0,j]} \right\}_{n=0}^\infty$ is tight in $C( [0,j]; \RB^d)$.
Then it suffices to establish the tightness of $\left\{ U_n(\cdot)|_{[0,j]} \right\}_{n=0}^\infty$ for each $j \ge 1$. 

To this end, we resort to \cite[Theorem~7.3]{billingsley2013convergence}.
Although this theorem only focus on $C([0,1]; \RB)$, we can generalize it effortlessly to $C([0,j]; \RB^{d})$.
Then it suffices to verify that for any $\eps, \eta>0$, there exists $\delta \in (0,1)$ such that 
    \begin{equation*}
        \PB\bigg(\sup_{s,t \in[0,j] \colon |s-t| \le \delta} \| U_n(s) - U_n(t) \| \ge \eps \bigg) \le \eta, \quad \forall\, n \ge n_0.
    \end{equation*}
By \eqref{eq:tight-suff-slow}, we can set $\delta = \min \left\{ \left(\frac{\eta \eps^{c_1}}{2K}\right)^{1/(1+\gamma_1)},   \left(\frac{\eta \eps^{c_2}}{2K}\right)^{1/(1+\gamma_2)}, 0.5\right\}$ to ensure $K \left( \frac{\delta^{1 {+} \gamma_1} }{\eps^{c_1} } + \frac{ \delta^{1 {+} \gamma_2}}{\eps^{c_2}} \right) \le \eta $. 
\end{proof}

\begin{proof}[Proof of Lemma~\ref{lem:tight-twotime}]
	We begin by proving the tightness of $\{ \xbar_{n}(\cdot) \}_{n=1}^\infty$. The tightness of $\{ \zbar_{n}(\cdot) \}_{n=1}^\infty$ can be established using a similar approach. Finally, we address the tightness of $\{ \ybar_{n}(\cdot) \}_{n=1}^\infty$, which is the primary challenge in this proof.
	\vspace{0.2cm}
	
	\textbf{Part~1: Tightness of $\{ \xbar_{n}(\cdot) \}_{n=1}^\infty$}.
	Note that Assumption~\ref{assump:mart-decouple-rate} implies $\EB \| \xbar_{n}(0) \|^2 = \EB \| \xcheck_{n} \|^2 = \OM(1)$. 
	Then by Markov's inequality, the first condition in Proposition~\ref{prop:tight-suff-slow} holds.
	Now we verify the second condition.
	Without loss of generality,  assume $0 \le t < s \le T$.
	From the definition of $\xbar_{n}(t)$ in \eqref{eq:def:xbar}, we have
	\begin{align*}
		\xbar_{n}(s) - \xbar_{n}(t)
		& \overset{\eqref{eq:def:xbar}}{=} \xcheck_{N^\alpha(n,s)} + (s - \underline{s}^\alpha_{n} ) \left( - B_1 \xcheck_{N^\alpha(n,s)} + \frac{ R_{N^\alpha(n,s)}^x }{ \alpha_{N^\alpha(n,s)} } \right) - \frac{s - \underline{s}^\alpha_{n}}{ \sqrt{ \alpha_{N^\alpha(n,s)} } } 
		\tilde{\xi}_{N^\alpha(n,s)} \\
		& \qquad - \xcheck_{N^\alpha(n,t)} - (t - \underline{t}^\alpha_{n} ) \left( - B_1 \xcheck_{N^\alpha(n,t)} + \frac{ R_{N^\alpha(n,t)}^x }{ \alpha_{N^\alpha(n,t)} } \right) + \frac{t - \underline{t}^\alpha_{n}}{ \sqrt{\alpha_{N^\alpha(n,t)} } } 
		\tilde{\xi}_{N^\alpha(n,t)} \\
		& \overset{\eqref{eq:xcheck-twotime}}{=} \Bigg\{ \ssum{k}{N^\alpha(n,t)}{N^\alpha(n,s)-1} \alpha_{k} \left( - B_1 \xcheck_{k} + \frac{R_{k}^x}{\alpha_{k}} \right) - (t - \underline{t}^\alpha_{n} ) \left( - B_1 \xcheck_{N^\alpha(n,t)} + \frac{ R_{N^\alpha(n,t)}^x }{ \alpha_{N^\alpha(n,t)} } \right) \\
		& \qquad + (s - \underline{s}^\alpha_{n} ) \left( - B_1 \xcheck_{N^\alpha(n,s)} + \frac{ R_{N^\alpha(n,s)}^x }{ \alpha_{N^\alpha(n,s)} } \right) \Bigg\} - \Bigg\{ \ssum{k}{N^\alpha(n,t)}{N^\alpha(n,s)-1} \sqrt{\alpha_{k}} \tilde{\xi}_{k} \\ 
		& \qquad -  \frac{t - \underline{t}^\alpha_{n}}{ \sqrt{\alpha_{N^\alpha(n,t)} } } 
		\tilde{\xi}_{N^\alpha(n,t)} 
		+ \frac{s - \underline{s}^\alpha_{n}}{ \sqrt{ \alpha_{N^\alpha(n,s)} } } 
		\tilde{\xi}_{N^\alpha(n,s)} \Bigg\} \\
		& =: \bfB^{(1x,n)}_{t,s} - \Xii^{(n)}_{t,s}.
	\end{align*}
    
    The first term can also be expressed as
    \begin{align*}
        \bfB^{(1x,n)}_{t,s} & = \Bigg\{ \ssum{k}{N^\alpha(n,t)+1}{N^\alpha(n,s)-1} \alpha_{k} \left( - B_1 \xcheck_{k} + \frac{R_{k}^x}{\alpha_{k}} \right) + \left[\alpha_{N^\alpha(n,t)} -  (t - \underline{t}^\alpha_{n} )\right] \left( - B_1 \xcheck_{N^\alpha(n,t)} + \frac{ R_{N^\alpha(n,t)}^x }{ \alpha_{N^\alpha(n,t)} } \right) \\
		& \qquad + (s - \underline{s}^\alpha_{n} ) \left( - B_1 \xcheck_{N^\alpha(n,s)} + \frac{ R_{N^\alpha(n,s)}^x }{ \alpha_{N^\alpha(n,s)} } \right) \Bigg\}.
    \end{align*}
    Note that $t' - (\underline{t'})^\alpha_{n} \le \alpha_{N^\alpha(n,t')}$ for $t'=t,s$.
    Then we have $0 \le \alpha_{N^\alpha(n,t)} -  (t - \underline{t}^\alpha_{n} ) \le \alpha_{N^\alpha(n,t)}$ and $0 \le s - \underline{s}^\alpha_{n} \le \alpha_{N^\alpha(n,s)} $.
	It follows that
	\begin{align}
		\| \bfB^{(1x,n)}_{t,s} \| 
        & \le  \ssum{k}{N^\alpha(n,t)+1}{N^\alpha(n,s) - 1} \alpha_{k} \left\| B_1 \xcheck_{k} - \frac{R_{k}^x}{\alpha_{k}} \right\| + \left[ \alpha_{N^\alpha(n,t)} - (t - \underline{t}^\alpha_{n} ) \right] \left( - B_1 \xcheck_{N^\alpha(n,t)}  + \frac{ R_{N^\alpha(n,t)}^x }{ \alpha_{N^\alpha(n,t)} } \right) \nonumber \\
        & \qquad  + (s - \underline{s}^\alpha_{n}) \left\| B_1 \xcheck_{N^\alpha(n,s)} - \frac{ R_{N^\alpha(n,s)}^x }{ \alpha_{N^\alpha(n,s)} } \right\| \nonumber \\
		& \le \ssum{k}{N^\alpha(n,t)}{N^\alpha(n,s)} \alpha_{k} \left\| B_1 \xcheck_{k} - \frac{R_{k}^x}{\alpha_{k}} \right\|. \label{eq:lem:tight-twotime:B1x-1}
	\end{align}
	By Assumption~\ref{assump:mart-decouple-rate} and Lemma~\ref{lem:check-twotime-xy}, we have that 
	\begin{equation}
		\label{eq:lem:tight-twotime:l2bound-x}
		\sup_{n} \EB \left\| B_1 \xcheck_{n} - \frac{R_{n}^x }{ \alpha_{n}} \right\|^2 \le 2 \| B_1 \|^2 \sup_{n} \EB \| \xcheck_{n} \|^2 + 2 \sup_{n} \frac{ \EB \| R_{n}^x \|^2}{\alpha_{n}^2} \le \cti_x
	\end{equation}
	for some $\cti_x > 0$. 
	Applying the Cauchy-Schwarz inequality, we obtain
	\begin{align}
		\EB \| \bfB^{(1x,n)}_{t,s} \|^2
		& \overset{\eqref{eq:lem:tight-twotime:B1x-1}}{\le}
		\EB \left( \ssum{k}{N^\alpha(n,t)}{N^\alpha(n,s)} \alpha_{k} \left\| B_1 \xcheck_{k} - \frac{R_{k}^x}{\alpha_{k}} \right\|\right)^2 \nonumber \\
		& \le \left( \ssum{k}{N^\alpha(n,t)}{N^\alpha(n,s)} \alpha_{k} \right) \left( \ssum{k}{N^\alpha(n,t)}{N^\alpha(n,s)} \alpha_{k} \EB  \left\| B_1 \xcheck_{k} - \frac{R_{k}^x}{\alpha_{k}} \right\|^2 \right) \nonumber \\
		& \overset{ \eqref{eq:lem:tight-twotime:l2bound-x} }{\le} \cti_x \left( \ssum{k}{N^\alpha(n,t)}{N^\alpha(n,s)} \alpha_{k} \right)^2 
		\le  \cti_x ( s-t + 2 \alpha_{n} )^2
		\overset{(a)}{\le} K_1(s-t)^2,
		\label{eq:lem:tight-twotime:B1x}
	\end{align}
	where there exists $K_1$ such that (a) holds for any $n$ because $\alpha_{n} \to 0$.

    To deal with the second term, we need the following Burkholder's inequality. 
    \begin{prop}[Burkholder's inequality, {\cite{burkholder1988sharp}}]\label{prop:burk}
    Fix any $p \ge 2$.
    Let $\{ X_k \}_{1 \le k \le n}$ be a finite $\RB^d$-valued martingale difference sequence with $\sup_{1 \le k \le n} \EB \| X_k \|^p < \infty$. Then we have
    \begin{equation*}
        \EB \left\| \ssum{k}{1}{n} X_k \right\|^p \le \left[ (p-1) \vee \frac{1}{p-1} \right]^p \EB \left( \ssum{k}{1}{n} \| X_k \|^2 \right)^\frac{p}{2},
    \end{equation*}
    which together with Jensen's inequality imply 
    \begin{equation}\label{eq:burk-jensen}
        \EB \left\| \ssum{k}{1}{n} X_k \right\|^p \le \left[ (p-1) \vee \frac{1}{p-1} \right]^p n^{\frac{p}{2} - 1 } \EB \ssum{k}{1}{n} \| X_k \|^p.
    \end{equation}
	\end{prop}
	With $p \in \left(2, \frac{4}{1 + ( \Hsmooth \vee \Fsmooth \vee \Gsmooth )/2} \right]$,
	we have
	\begin{equation*}
		\EB \| \Xii^{(n)}_{t,s} \|^p
		\overset{\eqref{eq:burk-jensen}}{\le}  \cti_p \left[ N^\alpha(n, s) - N^\alpha(n,t) + 1 \right]^{\frac{p}{2}-1} \ssum{k}{ N^\alpha(n,t) }{N^\alpha(n,s) } \alpha_{k}^\frac{p}{2} \EB \| \tilde{\xi}_{k} \|^p,
	\end{equation*}
	where we have used $t' - \underline{t'}^\alpha_{n} \le \alpha_{N^\alpha(n,t')}$ for $t'=t,s$ 
    and the constant number is 
    $\cti_p = \left[ (p-1) \vee \frac{1}{p-1} \right]^p$.

    To further derive the upper bound, we control different terms separately as follows.
    \begin{itemize}
        \item For $N^\alpha(n, s) - N^\alpha(n,t) + 1$, From the definition of $N^\alpha(n,t')$ in Definition~\ref{defn:time-inter}, we have $\ssum{k}{n}{N^\alpha(n,t')-1} \alpha_k \le t' < \ssum{k}{n}{N^\alpha(n,t')} \alpha_k $ for $t'=s,t$.
        Note that $\alpha_k$ is decreasing in $k$.
        It follows that
        \[s - t \ge \ssum{k}{n}{N^\alpha(n,s)-1} \alpha_k - \ssum{k}{n}{N^\alpha(n,t)} \alpha_k = \ssum{k}{N^\alpha(n,t)+1}{N^\alpha(n,s)-1} \alpha_k \ge (N^\alpha(n,s)-N^\alpha(n,t)-2) \alpha_{N^\alpha(n,s)}.\]
        Thus, $$N^\alpha(n, s) - N^\alpha(n,t) + 1 \le \frac{s-t}{\alpha_{N^\alpha(n,s)}}+3 \le \frac{s-t + 3 \alpha_n}{\alpha_{N^\alpha(n,s)}}.$$
        \item For $\ssum{k}{ N^\alpha(n,t) }{N^\alpha(n,s) } \alpha_{k}^\frac{p}{2}$, since $N^\alpha(n,t) \ge n$ and $\alpha_k$ is decreasing in $k$, we have $\ssum{k}{ N^\alpha(n,t) }{N^\alpha(n,s) } \alpha_{k}^\frac{p}{2} \le \alpha_n^{\frac{p}{2}-1} \ssum{k}{ N^\alpha(n,t) }{N^\alpha(n,s) } \alpha_{k}$.
        From the analysis in the first bullet, we have \[\ssum{k}{ N^\alpha(n,t) }{N^\alpha(n,s) } \alpha_{k} = \ssum{k}{ n }{N^\alpha(n,s)-1 } \alpha_{k} + \alpha_{N^\alpha(n,s)} - \ssum{k}{ n }{N^\alpha(n,t) } \alpha_{k} + \alpha_{N^\alpha(n,t)} \le s - t +2\alpha_n.\]
        Thus,
        $$\ssum{k}{ N^\alpha(n,t) }{N^\alpha(n,s) } \alpha_{k}^\frac{p}{2} \le \alpha_n^{\frac{p}{2}-1} ( s - t +2\alpha_n).$$
        \item For $\EB \| \tilde{\xi}_{k} \|^p$, from the definition of $\tilde{\xi}_{k} = \xi_{k} - \kappa_{k} H^\star \psi_{k}$ and Assumption~\ref{assump:noise-fclt}, there exists $\cti_\xi > 0$ such that $$\sup_{k} \EB \| \tilde{\xi}_{k} \|^p \le \cti_\xi. $$
    \end{itemize}
    Combining the above bounds yields
    \begin{align}
        \EB \| \Xii^{(n)}_{t,s} \|^p 
        & \le  \cti_p \cti_\xi \left[ N^\alpha(n, s) - N^\alpha(n,t) + 1 \right]^{\frac{p}{2}-1} \ssum{k}{ N^\alpha(n,t) }{N^\alpha(n,s) } \alpha_{k}^\frac{p}{2} \nonumber \\
        & \le \cti_p \cti_\xi \left( \frac{s-t + 3 \alpha_n}{\alpha_{N^\alpha(n,s)}} \right)^{\frac{p}{2}-1}  \alpha_n^{\frac{p}{2}-1} ( s - t +2\alpha_n) \nonumber\\
        & \le  \cti_p \cti_\xi ( s-t + 3 \alpha_{n} )^\frac{p}{2} \left( \frac{ \alpha_{n} }{ \alpha_{ N^\alpha(n, s) }  } \right)^{\frac{p}{2}-1 } \nonumber \\
		& \overset{ (a) }{\le} \cti_p \cti_\xi \cti_\alpha ( s-t + 3 \alpha_{n} )^\frac{p}{2}
		\overset{ (b) }{\le} K_2 ( s-t )^\frac{p}{2},
		\label{eq:lem:tight-twotime:xii}
    \end{align}
	where (a) uses Lemma~\ref{lem:aux:bounded} with $\cti_\alpha = 2^{ ( \lceil 2T/ c_\alpha \rceil + 1 )(p/2-1) }$ with $c_\alpha$ specified therein, 
	and there exists $K_2$ such that (b) holds for any $n$.

	Combining \eqref{eq:lem:tight-twotime:B1x} and \eqref{eq:lem:tight-twotime:xii} with Markov's inequality, 
	we have
	\begin{align*}
		\PB \left( 
		\| \xbar_{n}(s) - \xbar_{n}(t) \| \ge \eps \right)
		& \le \PB \left( 
		\| \bfB^{(1x,n)}_{t,s} \| \ge \frac{\eps}{2} \right) + \PB \left( 
		\| \Xii^{(n)}_{t,s} \| \ge \frac{\eps}{2} \right) 
		\\
		& \le \frac{4}{\eps^2} \EB \| \bfB^{(1x,n)}_{t,s} \| ^2 + \frac{2^p}{\eps^p} \EB \| \Xii^{(n)}_{t,s} \|^p \\
		& 
		\le 8 (K_1 \vee K_2) \left[ \frac{ (s-t)^2 }{\eps^2} + \frac{ (s-t)^\frac{p}{2} }{\eps^p} \right].
	\end{align*}
	This upper bound verifies the second condition in Proposition~\ref{prop:tight-suff-slow}. Consequently, $\{ \xbar_{n}(\cdot) \}_{n=1}^\infty$ is tight.
	\vspace{0.2cm}
	
	\textbf{Part~2: Tightness of $\{ \zbar_{n}(\cdot) \}_{n=1}^\infty$}.
	Note that in the proof of Part~1, the only place we employed Lemma~\ref{lem:check-twotime-xy} is \eqref{eq:lem:tight-twotime:l2bound-x} and in fact, we only require $\EB \| R_{n}^x \|^2 = \OM(\alpha_{n}^2)$. Moreover, the dimension $d_x$ would not affect the proof.
	The tightness of $\{ \zbar_{n}(\cdot) \}_{n=1}^\infty$ follows from a similar procedure by applying the residual bounds in Lemma~\ref{lem:check-twotime-z}.
	\vspace{0.2cm}
	
	\textbf{Part~3: Tightness of $\{ \ybar_{n}(\cdot) \}_{n=1}^\infty$}.
	Note that Assumption~\ref{assump:mart-decouple-rate} implies $\EB \| \ybar_{n}(0) \|^2 = \EB \| \ycheck_{n} \|^2 = \OM(1)$. 
	Then by Markov's inequality, the first condition in Proposition~\ref{prop:tight-suff-slow} holds.
	Now we verify the second condition.
	Without loss of generality, we assume $0 \le t < s \le T$.
	For ease of notation, we define $\widetilde{B}_3 := B_3 - \frac{\invdiffslow \mI}{2}$.
	From the definition of $\ybar_{n}(t)$ in \eqref{eq:def:ybar}, we have
	
	\begin{align*}
		\ybar_{n}(s) - \ybar_{n}(t)
		& \overset{\eqref{eq:def:ybar}}{=} \ycheck_{N^\beta(n,s)} + (s - \underline{s}^\beta_{n} ) \left( - \widetilde{B}_3  \ycheck_{N^\beta(n,s)} + \frac{ R_{N^\beta(n,s)}^y }{ \beta_{N^\beta(n,s)} } \right) -  \frac{s - \underline{s}^\beta_{n}}{ \sqrt{ \beta_{N^\beta(n,s) } } }  
		\psi_{N^\beta(n,s)} \\
		& \qquad - \ycheck_{N^\beta(n,t)} - (t - \underline{t}^\beta_{n} ) \left( - \widetilde{B}_3 \ycheck_{N^\beta(n,t)} + \frac{ R_{N^\beta(n,t)}^y }{ \beta_{N^\beta(n,t)} } \right) +  \frac{t - \underline{t}^\beta_{n}}{ \sqrt{ \beta_{N^\beta(n,t)} } }
		\psi_{N^\beta(n,t)} \\
		& \qquad - (s - \underline{s}^\beta_{n}) \sqrt{ \frac{ \alpha_{N^\beta(n,s)} }{ \beta_{N^\beta(n,s)} } } B_2 \xcheck_{N^\beta(n,s)} 
		+ (t - \underline{t}^\beta_{n}) \sqrt{ \frac{ \alpha_{N^\beta(n,t)} }{ \beta_{N^\beta(n,t)} } } B_2 \xcheck_{N^\beta(n,t)} \\
		& \overset{\eqref{eq:ycheck-twotime}}{=} \Bigg\{ \ssum{k}{N^\beta(n,t)}{N^\beta(n,s)-1} \beta_{k} \left( - \widetilde{B}_3 \ycheck_{k} + \frac{R_{k}^y}{\beta_{k}} \right) - (t - \underline{t}^\beta_{n} ) \left( - \widetilde{B}_3 \ycheck_{N^\beta(n,t)} + \frac{ R_{N^\beta(n,t)}^y }{ \beta_{N^\beta(n,t)} } \right) \\  
		& \qquad+ (s - \underline{s}^\beta_{n} ) \left( - \widetilde{B}_3 \ycheck_{N^\beta(n,s)} + \frac{ R_{N^\beta(n,s)}^y }{ \beta_{N^\beta(n,s)} } \right) \Bigg\}  - \Bigg\{ \ssum{k}{N^\beta(n,t)}{N^\beta(n,s)-1} \sqrt{\beta_{k}} \psi_{k} \\
		& \qquad -  \frac{t - \underline{t}^\beta_{n}}{ \sqrt{\beta_{N^\beta(n,t)} } } \psi_{N^\beta(n,t)} 
		+ \frac{s - \underline{s}^\beta_{n}}{ \sqrt{\beta_{N^\beta(n,s)} } } \psi_{N^\beta(n,s)} \Bigg\}
		- \Bigg\{ \ssum{k}{N^\beta(n,t)}{N^\beta(n,s)-1} \beta_{k} \sqrt{ \frac{\alpha_{k}}{\beta_{k}} } B_2 \xcheck_{k} \\
		& \qquad  - (t - \underline{t}^\beta_{n}) \sqrt{ \frac{ \alpha_{N^\beta(n,t)} }{ \beta_{N^\beta(n,t)} } } B_2 \xcheck_{N^\beta(n,t)} + (s - \underline{s}^\beta_{n}) \sqrt{ \frac{ \alpha_{N^\beta(n,s)} }{ \beta_{N^\beta(n,s)} } } B_2 \xcheck_{N^\beta(n,s)} 
		\Bigg\} \\
		& =: \bfB^{(3y,n)}_{t,s} - \Psii^{(n)}_{t,s} - \bfB^{(2x,n)}_{t,s}.
	\end{align*}
	
	For the first and second terms,
	similar to the derivation of \eqref{eq:lem:tight-twotime:B1x} and \eqref{eq:lem:tight-twotime:xii}, with $p \in \left(2, \frac{4}{1 + ( \Hsmooth \vee \Fsmooth \vee \Gsmooth )/2} \right]$, there exist $K_3, K_4$ such that for any $n$,
	\begin{equation}\label{eq:lem:tight-slow:BandPsi}
		\EB \| \bfB^{(3y,n)}_{t,s} \|^2 \le K_3 (s-t)^2 \quad \text{ and } \quad
		\EB \| \Psii^{(n)}_{t,s} \|^p \le K_4 (s-t)^\frac{p}{2}.
	\end{equation}
	
	\textbf{Main difficulty}.
	The main difficulty lies to how to deal with the last term $ \bfB^{(2x,n)}_{t,s}$.
	If we adopt a similar procedure of deriving \eqref{eq:lem:tight-twotime:B1x} by the Cauchy-Schwarz inequality,
	we could only obtain 
	\[\EB \| \bfB^{(2x,n)}_{t,s} \|^2 \lesssim (s-t)  \left( \ssum{k}{N^\beta(n,t)}{N^\beta(n,s)} \beta_{k} \frac{\alpha_{k} }{\beta_{k}} \right)
	\lesssim (s-t)^2 \max_{k \ge N^\beta(n,s)} \frac{\alpha_{k}}{\beta_{k}}.
	\] 
	Since Assumption~\ref{assump:stepsize-twotime}~\ref{assump:stepsize-twotime-tozero} implies $\max_{k \ge N^\beta(n,s)} \alpha_{k} / \beta_{k} \to \infty$ as $n \to \infty$, this upper bound becomes ineffective.
	\vspace{0.2cm}
	
	\textbf{Solution approach}.
	To address this issue, we need to provide a more precise characterization of $ \bfB^{(2x,n)}_{t,s}$.
	A natural approach is to further expand each $\xcheck_{k}$. While this idea is feasible, it leads to more complex calculations, as each $\xcheck_{k}$ depends on both $\xcheck_{k-1}$ and $\ycheck_{k-1}$. Below, we outline the detailed process.
	
	To begin, we define another term 
		$\bfB^{(2x)}_{n_0,n_1}
		:= \ssum{k}{n_0}{n_1} \sqrt{\beta_{k}} B_2 \xhat_{k}$.
		Here, we use the unrescaled iterate to simplify the subsequent calculations. Given that $\xcheck_{k} = \xhat_{k} / \sqrt{\alpha_{k-1} }$, it follows that
		$\bfB^{(2x)}_{n_0,n_1} = \ssum{k}{n_0}{n_1} \sqrt{ \beta_{k} \alpha_{k-1} } B_2 \xcheck_{k}
		$, which demonstrates that this term is closely related to $\bfB^{(2x,n)}_{t,s}$, as shown in the following equation
		\begin{align*}
			\bfB^{(2x,n)}_{t,s} 
			& = \bfB^{(2x)}_{N^\beta(n,t), N^\beta(n,s)-1}
			- (t - \underline{t}^\beta_{n}) \sqrt{ \frac{ \alpha_{N^\beta(n,t)} }{ \beta_{N^\beta(n,t)} } } B_2 \xcheck_{N^\beta(n,t)} \\
			& \qquad  + (s - \underline{s}^\beta_{n})
			\sqrt{ \frac{ \alpha_{N^\beta(n,s)} }{ \beta_{N^\beta(n,s)} } } B_2 \xcheck_{N^\beta(n,s)} 
			+ \ssum{k}{N^\beta(n,t)}{N^\beta(n,s)-1} \sqrt{\beta_{k}} \left( \sqrt{ \frac{\alpha_{k}}{\alpha_{k-1}} } - 1 \right) B_2 \xhat_{k}.
		\end{align*}
		By the Cauchy-Schwarz inequality, we have
		\begin{align*}
			& \EB \left\| \ssum{k}{N^\beta(n,t)}{N^\beta(n,s)-1} \sqrt{\beta_{k}} \left( \sqrt{ \frac{\alpha_{k}}{\alpha_{k-1}} } - 1 \right) B_2 \xhat_{k} \right\|^2 \\
			& \qquad \qquad \le \left( \ssum{k}{N^\beta(n,t)}{N^\beta(n,s)-1} \beta_{k} \right) \left[ \ssum{k}  {N^\beta(n,t)}{N^\beta(n,s)-1}  \Big( \sqrt{ \frac{\alpha_{k}}{\alpha_{k-1}} } - 1 \Big)^2  \EB \| B_2 \xhat_{k} \|^2 \right].
		\end{align*}
		In \eqref{eq:lem:check-twotime:sqrtalpha}, we have shown that $\sqrt{ \frac{\alpha_{k}}{\alpha_{k-1}} } - 1 = o(\beta_k)$.
		Assumption~\ref{assump:mart-decouple-rate} implies $\EB \| B_2 \xhat_{k} \|^2 = o(1)$.
		Moreover,
		$t' - (\underline{t'})^\beta_{n} \le \beta_{N^\beta(n,t')}$ for $t'=t,s$.
		Summarizing the above results and following a similar procedure as \eqref{eq:lem:tight-twotime:B1x}, we can obtain that there exists $K_5$ such that
		\begin{equation}\label{eq:lem:tight-slow:Bdiff}
			\EB \| \bfB^{(2x,n)}_{t,s} - \bfB^{(2x)}_{N^\beta(n,t), N^\beta(n,s)-1} \|^2
			\le K_5 (s-t)^2
		\end{equation}
		for any $n$.
		This inequality controls the difference between $\bfB^{(2x,n)}_{t,s}$ and 
		$\bfB^{(2x)}_{N^\beta(n,t), N^\beta(n,s)-1}$.
		\vspace{0.2cm}
		
		From now on, we always assume $n_0 = N^\beta(n,t)$ and $n_1 = N^\beta(n,s)-1$ and
		It remains to deal with the term $\bfB^{(2x)}_{n_0, n_1}$.
		Recall that the update rule of $\xhat_{n}$ is shown in \eqref{eq:xhat-decom}.
		For ease of notation, we define $A_{n}:= \mI - \alpha_{n} B_1 + \beta_{n} H^\star B_2$, $\tilde{\xi}_{n} := \xi_{n} - \beta_{n} H^\star \psi_{n} / \alpha_{n}$ and 
		$\tilde{R}^x_{n} := R^{\nablaH}_{n} + R^H_{n} + \alpha_{n} R^F_{n} - \beta_{n} H^\star R^G_{n}$.
		Then the update rule of $\xhat_{n}$ can be rewritten as
		\begin{equation*}
			\xhat_{n+1} = A_{n} \xhat_{n} + \beta_{n} H^\star B_3 \yhat_{n} - \alpha_{n} \tilde{\xi}_{n} - \tilde{R}^x_{n}.
		\end{equation*}
		Applying this equation multiple times yields
		\begin{equation*}
			\xhat_{n+m}
			= \bigg( \prod_{i=1}^{m} A_{n+m-i} \bigg) \xhat_{n}
			+ \sum_{l=n}^{n+m-1} \bigg( \prod_{i=1}^{n+m-l-1} A_{n+m-i} \bigg) \big( \beta_{l} H^\star B_3 \yhat_{l} - \alpha_{l} \tilde{\xi}_{l} - \tilde{R}^x_{l} \big).
		\end{equation*}
		It follows that
		\begin{align*}
			\bfB^{(2x)}_{n_0, n_1}
			& = \ssum{k}{n_0}{n_1} \sqrt{\beta_{k}} B_2 \left[ \bigg( \prod_{i=1}^{k-n_{0}} A_{k-i} \bigg) \xhat_{n_0} + \ssum{l}{n_0}{k-1} \bigg( \prod_{i=1}^{k-l-1} A_{k-i} \bigg) \big( \beta_{l} H^\star B_3 \yhat_{l} - \alpha_{l} \tilde{\xi}_{l} - \tilde{R}^x_{l} \big) \right] \\
			& = \ssum{k}{n_0}{n_1} \sqrt{\beta_{k}} B_2 \bigg( \prod_{i=1}^{k-n_{0}} A_{k-i} \bigg) \xhat_{n_0} 
			+ \ssum{l}{n_0}{n_1-1} \ssum{k}{l+1}{n_1} \sqrt{\beta_{k}} B_2 \bigg( \prod_{i=1}^{k-l-1} A_{k-i} \bigg) \big( \beta_{l} H^\star B_3 \yhat_{l} {-} \alpha_{l} \tilde{\xi}_{l} {-} \tilde{R}^x_{l} \big) \\
			& = \underbrace{ B_2 \widetilde{\Gamma}_{n_0, n_1} \xhat_{n_0} 
				+ B_2 \ssum{l}{n_0}{n_1-1} \widetilde{\Gamma}_{l+1, n_1} \big( \beta_{l} H^\star B_3 \yhat_{l}  - \tilde{R}^x_{l} \big) }_{:= \bfB^{(2x,1)}_{n_0,n_1} }
			- \underbrace{ B_2 \ssum{l}{n_0}{n_1-1} \widetilde{\Gamma}_{l+1, n_1} \alpha_{l} \tilde{\xi}_{l} }_{:= \bfB^{(2x,2)}_{n_0,n_1} },
		\end{align*}
		where we define $\widetilde{\Gamma}_{m, j} := \ssum{k}{m}{j} \sqrt{\beta_{k}} \bigg( \prod_{i=1}^{k-m} A_{k-i} \bigg)$ in the last step.
		The magnitude of $\widetilde{\Gamma}_{m, j}$ can be controlled by the following auxiliary lemma, whose proof is deferred to the end of this subsection.
		
		\begin{lem}\label{lem:slow:stepsize-aux}
			Suppose that Assumption~\ref{assump:stepsize-twotime} holds and
			the real parts of the eigenvalues of $B_1$ are no less than $\mu_1$.
			Define $A_{k}:= \mI - \alpha_{k} B_1 + \beta_{k} H^\star B_2$.
			If $t$ and $s$ are fixed positive numbers satisfying $s>t$ and $n$ is sufficiently large,
			then for any $N^\beta(n,t) \le l < N^\beta(n, s)$, 
			we have
			\begin{equation*}
				\left\| \sum_{k=l+1}^{N^\beta(n,s)-1} \sqrt{\beta_{k}} \prod_{i=1}^{k-l-1} A_{k-i}
				\right\|
				\le \frac{ 2^{ \left\lceil 2s/c_\alpha \right\rceil + 3} }{\mu_1} \cdot
				\frac{ \sqrt{\beta_{l}} }{ \alpha_{l} }.
			\end{equation*}
		\end{lem}
		
		For the first term $\bfB^{(2x,1)}_{n_0,n_1}$, the Cauchy-Schwarz inequality implies
		\begin{align*}
			& \quad \EB \| \bfB^{(2x,1)}_{n_0,n_1} \|^2 \\
			& \le 2  \EB \| B_2 \widetilde{\Gamma}_{n_0, n_1} \xhat_{n_0} \|^2 + 2 \EB \left\| B_2 \ssum{l}{n_0}{n_1-1} \widetilde{\Gamma}_{l+1, n_1} (\beta_{l} H^\star B_3 \yhat_{l} - \tilde{R}^x_{l} ) \right\|^2 \\
			& \le 2 (\| B_2 \| \|\widetilde{\Gamma}_{n_0, n_1} \|)^2 \EB \| \xhat_{n_0} \|^2 + 2 \| B_2 \|^2 \left( \ssum{l}{n_0}{n_1-1} \beta_{l} \right) \!
			\left( \ssum{l}{n_0}{n_1-1} \beta_{l} \| \widetilde{\Gamma}_{l+1, n_1} \|^2  \EB \left\| H^\star B_3 \yhat_{l} {-} \frac{\tilde{R}^x_{l}}{\beta_{l}}  \right\|^2
			\right).
		\end{align*}
		By Lemma~\ref{lem:slow:stepsize-aux}, 
		when $n$ is sufficiently large, we $\| \widetilde{\Gamma}_{l+1,n_1} \| \le \frac{ 2^{ \left\lceil 2T/c_\alpha \right\rceil + 3} }{\mu_1} \frac{ \sqrt{\beta_{l}} }{ \alpha_{l} } $ for any $n_0 \le l < n_1$ and $\| \widetilde{\Gamma}_{n_0, n_1} \| \le \frac{ 2^{ \left\lceil 2T/c_\alpha \right\rceil + 3} }{\mu_1} \frac{ \sqrt{\beta}_{n_0} }{ \alpha_{n_0} } + \sqrt{\beta_{n_0}} $, with $\mu_1$ and $c_\alpha$ specified in Lemma~\ref{lem:slow:stepsize-aux};
		Assumption~\ref{assump:mart-decouple-rate} implies $\EB \| \xhat_{n_0} \|^2 = \OM(\alpha_{n_0})$ and $\EB \| \yhat_{l} \|^2 = \OM (\beta_{l})$;
		the upper bounds in \eqref{eq:lem:check-twotime:Rx-upper} and the subsequent analysis therein imply that under Assumption~\ref{assump:stepsize-twotime},
		$\EB \| \tilde{R}^x_{l} \|^2 = \OM( \beta_{l}^2 \alpha_{l}^{1+\Hsmooth} + \beta_{l}^{2+\Hsmooth} + \alpha_{l}^{3+\Fsmooth} + \beta_{l}^2 \alpha_{l}^{1+\Gsmooth} ) = \OM(\alpha_{l}^2 \beta_{l})$.
		Therefore, we have $\|\widetilde{\Gamma}_{n_0, n_1} \|^2 \EB \| \xhat_{n_0} \|^2 = \OM \left( \frac{\beta_{n_0} }{\alpha_{n_0}} \right) = o(1)$ and
		\begin{equation*}
			\|\widetilde{\Gamma}_{l+1, n_1} \|^2 \EB \left\| H^\star B_3 \yhat_{l} - \frac{ \tilde{R}^x_{l} }{\beta_{l}} \right\|^2
			= \OM \left( \frac{\beta_{l}}{\alpha_{l}^2} \right) \left( \EB \| \yhat_{l} \|^2 + \frac{ \EB \| \tilde{R}^x_{l} \|^2 }{ \beta_{l}^2 } \right) = \OM(1).
		\end{equation*}
		Then similar to the derive of \eqref{eq:lem:tight-twotime:B1x}, there exist  $K_6$ and $m_1$ such that
		\begin{equation}\label{eq:lem:tight-slow:B2x-1}
			\EB \| \bfB^{(2x,1)}_{n_0,n_1} \|^2 \le K_6 (s-t)^2
		\end{equation}
		holds for any $n \ge m_1$.
		Modifying $K_6$, we could make \eqref{eq:lem:tight-slow:B2x-1} hold for any $n$.
		
		For the second term $\bfB^{(2x,2)}_{n_0,n_1}$,
		applying $\| \widetilde{\Gamma}_{l+1,n_1} \| \le C_T \frac{ \sqrt{\beta_{l}} }{ \alpha_{l} } $ with $C_T := \frac{ 2^{ \left\lceil 2T/c_\alpha \right\rceil + 3} }{\mu_1}$, we obtain that with $p \in \left(2, \frac{4}{1 + ( \Hsmooth \vee \Fsmooth \vee \Gsmooth )/2} \right]$ and sufficiently large $n$, 
		\begin{align*}
			\EB \| \bfB^{(2x,2)}_{n_0,n_1} \|^p
			& \overset{\eqref{eq:burk-jensen}}{\le} \cti_p \left( n_1 - n_0 \right)^{\frac{p}{2}-1} \| B_2 \|^p \ssum{l}{ n_0 }{n_1-1 } \| \widetilde{\Gamma}_{l+1, n_1} \|^p \alpha_{l}^p \EB \| \tilde{\xi}_{l} \|^p
			\\
			& \le \cti_p C_{T}^p \left( n_1 - n_0 \right)^{\frac{p}{2}-1} \| B_2 \|^p \ssum{l}{ n_0 }{n_1-1 } \beta_{l}^\frac{p}{2} \EB \| \tilde{\xi}_{l} \|^p,
		\end{align*}
		where we have used Burkholder's inequality in  Proposition~\ref{prop:burk} with $\cti_p = \left[ (p-1) \vee \frac{1}{p-1} \right]^p$.
		Then following a similar procedure to
		\eqref{eq:lem:tight-twotime:xii},
		we have that there exist $K_7$ and $m_2$ such that
		\begin{equation}\label{eq:lem:tight-slow:B2x-2}
			\EB \| \bfB^{(2x,2)}_{n_0,n_1} \|^p \le K_7 (s-t)^\frac{p}{2}
		\end{equation}
		for any $n \ge m_2$. Modifying $K_7$, we could make \eqref{eq:lem:tight-slow:B2x-2} hold for any $n$.
		
		The upper bounds in \eqref{eq:lem:tight-slow:Bdiff}, \eqref{eq:lem:tight-slow:B2x-1} and \eqref{eq:lem:tight-slow:B2x-2} collectively control the term $\bfB^{(2x,n)}_{t,s}$.
		Combining them with  \eqref{eq:lem:tight-slow:BandPsi} and Markov's inequality yields the following upper bound
		\begin{align*}
			\PB \left( 
			\| \ybar_{n}(s) - \ybar_{n}(t) \| \ge \eps \right)
			& \le \PB \left( 
			\| \bfB^{(3y,n)}_{t,s} \| \ge \frac{\eps}{5} \right) + \PB \left( 
			\| \Psii^{(n)}_{t,s} \| \ge \frac{\eps}{5} \right) 
			+ \PB \left( \| \bfB^{(2x,1)}_{n_0,n_1} \| \ge \frac{\eps}{5} \right) \\
			& \qquad + \PB \left( \| \bfB^{(2x,2)}_{n_0,n_1} \| \ge \frac{\eps}{5} \right)
			+ \PB \left( \| \bfB^{(2x,n)}_{t,s} - \bfB^{(2x)}_{n_0, n_1} \| \ge \frac{\eps}{5} \right)
			\\
			& \le \frac{25}{\eps^2} \left( \EB  \| \bfB^{(3y,n)}_{t,s} \|^2 + \EB \| \bfB^{(2x,1)}_{n_0,n_1} \|^2 + \EB \| \bfB^{(2x,n)}_{t,s} - \bfB^{(2x)}_{n_0, n_1} \|^2 \right) \\
			& \qquad + \frac{5^p}{\eps^p} \left( \EB \| \Psii^{(n)}_{t,s} \|^p + \EB \| \bfB^{(2x,2)}_{n_0,n_1} \|^p \right) \\
			& 
			\le 5^4 \left( \max_{3 \le i \le 7} K_i \right) \left[ \frac{ (s-t)^2 }{\eps^2} + \frac{ (s-t)^\frac{p}{2} }{\eps^p} \right].
		\end{align*}
		This upper bound verifies the second condition in Proposition~\ref{prop:tight-suff-slow}, and hence $\{ \ybar_{n}(\cdot) \}_{n=1}^\infty$ is tight.
	\end{proof}
	
	\begin{proof}[Proof of Lemma~\ref{lem:slow:stepsize-aux}]
		
		For ease of notation, we define $n_0 := N^\beta(n, t)$ and $n_1 := N^\beta(n, s)-1$.
		Since the real parts of the eigenvalues of $B_1$ are no less than $\mu_1$, $\beta_{j} = o(\alpha_{j})$ and $n$ is sufficiently large, we have $\| A_j \| \le \| \mI - \alpha_{j} B_1 \| + \beta_{j} \| H^\star B_2 \| \le  1 - \mu_1 \alpha_{j} / 2$ for any $j > l \ge n$.
		It follows that
		\begin{align*}
			\left\| \sum_{k=l+1}^{n_1} \sqrt{\beta_{k}} \prod_{i=1}^{k-l-1} A_{k-i}  \right\|
			\le \sum_{k=l+1}^{n_1} \sqrt{\beta_{k}} \prod_{i=l+1}^{k-1} \left( 1 {-} \frac{ \mu_1 \alpha_{i} }{2} \right)
			\le \sum_{k=l+1}^{n_1} \sqrt{\beta_{k}} \exp \left( - \frac{\mu_1}{2} \sum_{i=l+1}^{k-1} \alpha_{i} \right).
		\end{align*}
		Moreover, by Lemma~\ref{lem:aux:bounded} and the monotonicity of the step sizes, 
		we have
		$ \alpha_{k} \ge c_1 \alpha_{l}   $
		and $
		\beta_{k} \le \beta_{l}$ for any $l \le k \le n_1$.
		Here $c_1 := 2^{ -\left\lceil 2s/c_\alpha \right\rceil - 1}$ with $c_\alpha$ specified in Lemma~\ref{lem:aux:bounded}.
		This implies 
		\begin{align*}
			\sum_{k=l+1}^{n_1} \sqrt{\beta_{k}} \exp \left( - \frac{\mu_1}{2} \sum_{i=l+1}^{k-1} \alpha_{i} \right)
			& \le \sqrt{\beta_{l}} \sum_{k=l+1}^{n_1}  \exp \left[ - (k-1-l) \frac{c_1 \mu_1 \alpha_{l} }{2} \right] \\
			& \le \frac{ \sqrt{\beta}_{l} }{1 - \exp(- c_1 \mu_1 \alpha_{l} / 2) }.
		\end{align*}
		For $x \in (0,1)$, we have $\exp(-x) \le 1 - x/2$.
		Therefore, for sufficiently large $n$, we have
		\begin{align*}
			\left\| \sum_{k=l+1}^{n_1} \sqrt{\beta_{k}} \prod_{i=1}^{k-l-1} A_{k-i}  \right\|
			\le \frac{4  }{ c_1 \mu_1  } \cdot
			\frac{ \sqrt{\beta_{l} } }{ \alpha_{l} }
		\end{align*}
		for any $n_0 \le l < n_1$.
		With $c_1 = 2^{ -\left\lceil 2s/c_\alpha \right\rceil - 1}$, the coefficient is $\frac{4}{c_1 \mu_1} = \frac{ 2^{ \left\lceil 2s/c_\alpha \right\rceil + 3} }{\mu_1}$.
	\end{proof}

	\subsection{Proofs in Section~\ref{sec:fclt:mart-prob}}
	\label{sec:proof:fclt:mart-prob}
	In this subsection, we present the proof of Lemma~\ref{lem:mart-prob-twotime} in Section~\ref{sec:fclt:mart-prob}.
	We begin by presenting a discrete version of Lemma~\ref{lem:mart-prob-twotime}, which serves as an intermediate result leading to the full proof.
	
	\begin{lem}\label{lem:generator-twotime}
		Suppose that Assumptions~\ref{assump:smooth} -- \ref{assump:stepsize-twotime}, \ref{assump:noise-fclt} and \ref{assump:mart-decouple-rate} hold. Then
		\begin{enumerate}[(i)]
			\item For any $f \in C^\infty_c (\RB^{d_x})$,  with $\AM^x$ defined in \eqref{eq:generator-twotime:Lx}, we have
			\begin{equation}\label{eq:generator-twotime:x}
				\EB [ f (\xcheck_{n+1}) - f(\xcheck_{n}) | \FM_{n} ]
				= \alpha_{n} \AM^x f(\xcheck_{n}) + R^f_{n},
			\end{equation}
			where $\frac{1}{\alpha_{n}} \EB | R^f_{n} | \to 0$ as $n \to \infty$.
			\item For any $g \in C^\infty_c (\RB^{d_y})$,  with $\AM^y$ defined in \eqref{eq:generator-twotime:Ly},
			we have
			\begin{equation}
				\label{eq:generator-twotime:y}
				\EB [ g (\zcheck_{n+1}) - g(\zcheck_{n}) | \FM_{n} ]
				= \beta_{n} \AM^y g(\zcheck_{n}) + R^g_{n},
			\end{equation}
			where $\frac{1}{\beta_{n}} \EB | R^g_{n} | \to 0$ as $n \to \infty$.
		\end{enumerate}
	\end{lem}
	
	\begin{proof}[Proof of Lemma~\ref{lem:generator-twotime}]
		We first prove the part~(i). Throughout the proof, $C$ will represent a universal constant whose value may change from line to line, for the sake of convenience.
		Because $f \in C^\infty_c (\RB^{d_x})$, $\| \nabla f(\cdot) \|$ and $\| \nabla^2 f(\cdot) \|$ are both bounded functions.
		
		By  Taylor's theorem, we have
		\begin{align}\label{eq:lem:generator-twotime:f}
			\begin{split}
				f(\xcheck_{n+1}) - f(\xcheck_{n})
				& = \inner{ \nabla f (\xcheck_{n}) }{ \xcheck_{n+1} - \xcheck_{n} } + \frac{1}{2} ( \xcheck_{n+1} - \xcheck_{n} )^\top \nabla^2 f(\xcheck_{n}) ( \xcheck_{n+1} - \xcheck_{n} ) \\
				& \quad + \underbrace{ \frac{1}{2} ( \xcheck_{n+1} - \xcheck_{n} )^\top \left( \nabla^2 f(\ucheck_{n}) - \nabla^2 f(\xcheck_{n}) \right) ( \xcheck_{n+1} - \xcheck_{n} ) }_{R^f_{n,0}},
			\end{split}
		\end{align}
		where $\ucheck_{n}$ is a point between $\xcheck_{n}$ and $\xcheck_{n+1}$.
		For the first term, we have
		\begin{equation}\label{eq:lem:generator-twotime:f1}
			\EB [ \inner{ \nabla f (\xcheck_{n}) }{ \xcheck_{n+1} - \xcheck_{n} } | \FM_{n} ]
			\overset{\eqref{eq:xcheck-twotime}}{=}
			- \alpha_{n} \inner{\nabla f(\xcheck_{n})}{B_1 \xcheck_{n}} +
			\underbrace{ \inner{ \nabla f(\xcheck_{n}) }{ \EB[  R^x_{n}  | \FM_{n} ] } }_{R^f_{n,1}}.
		\end{equation}
		By Lemma~\ref{lem:check-twotime-xy}, we have
		\begin{equation}\label{eq:lem:generator-twotime:f2}
			\EB | R^f_{n,1} |
			\lesssim \EB \| \EB[  R^x_{n}  | \FM_{n} ] \| = o(\alpha_{n}).
		\end{equation}
		For the second term, we have
		\begin{align}
			& \qquad \EB \left[ ( \xcheck_{n+1} - \xcheck_{n} )^\top \nabla^2 f(\xcheck_{n}) ( \xcheck_{n+1} - \xcheck_{n} ) | \FM_{n} \right] \nonumber \\
			& \overset{\eqref{eq:xcheck-twotime}}{=}
			\alpha_{n} \EB \left[ \tilde{\xi}_{n}^\top \nabla^2 f(\xcheck_{n}) \tilde{\xi}_{n} | \FM_{n} \right] 
			+ 2 \sqrt{\alpha_{n}} \EB \left[ (-\alpha_{n} B_1 \xcheck_{n} + R^x_{n})^\top \nabla^2 f(\xcheck_{n}) \tilde{\xi}_{n} | \FM_{n} \right] \nonumber \\
			& \qquad + \EB[ (-\alpha_{n} B_1 \xcheck_{n} + R^x_{n})^\top \nabla^2 f(\xcheck_{n}) (-\alpha_{n} B_1 \xcheck_{n} + R^x_{n}) | \FM_{n} ] \nonumber \\
			\begin{split}
				& = \alpha_{n} \tr ( \nabla^2 f(\xcheck_{n})     \Sigma_{\xi} ) + \underbrace{ \alpha_{n} \tr \left( \nabla^2 f(\xcheck_{n}) 
					(\EB [ \tilde{\xi}_{n} \tilde{\xi}_{n}^\top | \FM_{n} ]
					- \Sigma_{\xi} ) 
					\right) }_{2 R^f_{n,2}} 
				+ \underbrace{ 2 \alpha_{n}^{3/2} \EB \left[ \left(\frac{R_{n}^x}{\alpha_{n} } \right)^\top \nabla^2 f(\xcheck_{n}) \tilde{\xi}_{n}  \Big| \FM_{n} \right] }_{2 R^f_{n,3}} \\
				& \qquad + \underbrace{ \alpha_{n}^2 \EB \left[ \left( -B_1 \xcheck_{n} + \frac{R_{n}^x }{ \alpha_{n} } \right)^\top \nabla^2 f(\xcheck_{n}) \left( -B_1 \xcheck_{n} + \frac{R_{n}^x }{ \alpha_{n} } \right) \Big| \FM_{n} \right] }_{2 R^f_{n,4}}.
			\end{split} \label{eq:lem:generator-twotime:f3}
		\end{align}
		Recall that $\tilde{\xi}_{n} = \xi_{n} - \kappa_{n} H^\star \psi_{n}$ with $\kappa_{n} = \beta_{n} / \alpha_{n} = o(1)$. Then we have $\EB [ \tilde{\xi}_{n} \tilde{\xi}_{n}^\top | \FM_{n} ] \overset{p}{\to} \Sigma_\xi $.
		Moreover, Assumption~\ref{assump:noise-fclt} ensures each entry of $\EB [ \tilde{\xi}_{n} \tilde{\xi}_{n}^\top | \FM_{n} ]$ is uniformly integrable, which together with convergence in probability imply $\EB \| \EB [ \tilde{\xi}_{n} \tilde{\xi}_{n}^\top | \FM_{n} ] - \Sigma_{\xi}  \| \to 0 $.
		From our assumptions and the results in Assumption~\ref{assump:mart-decouple-rate} and Lemma~\ref{lem:check-twotime-xy}, we have 
		\begin{align}\label{eq:lem:generator-twotime:f4}
			\begin{split}
				\EB | R^f_{n,2} | & \lesssim \alpha_{n} \EB \| \EB [ \tilde{\xi}_{n} \tilde{\xi}_{n}^\top | \FM_{n} ] - \Sigma_\xi \| = o(\alpha_{n}), \\
				\EB | R^f_{n,3} | & \lesssim \alpha_{n}^{3/2} \left( \frac{ \EB\| R^x_{n} \|^2}{\alpha_{n}^2} + \EB \| \tilde{\xi}_{n} \|^2 \right) = \OM(\alpha_{n}^{3/2}) = o(\alpha_{n}), \\
				\EB | R^f_{n,4} | & \lesssim \alpha_{n}^2 \left( \EB \| \xcheck_{n} \|^2 + \frac{ \EB\| R^x_{n} \|^2}{\alpha_{n}^2} \right) = \OM(\alpha_{n}^2) = o(\alpha_{n}).
			\end{split}
		\end{align}
		Since $\nabla^2 f$ is Lipschitz continuous and compactly supported, $\nabla^2 f$ is also $\eps$-H\"older
		continuous for any $\eps \in (0,1]$.
		It follows that
		\begin{align}
			\EB | R^f_{n,0} |
			& \lesssim \EB \| \xcheck_{n+1} - \xcheck_{n} \|^{2+\eps}
			\overset{\eqref{eq:xcheck-twotime}}{\le} \EB \left\| \alpha_{n} B_1 \xcheck_{n} + \sqrt{\alpha_{n}} \tilde{\xi}_{n} - R^x_{n} \right\|^{2+\eps} \nonumber \\
			& \lesssim \alpha_{n}^{2+\eps} \EB \| \xcheck_{n} \|^{2+\eps} + \alpha_{n}^{1+\frac{\eps}{2}} \EB \| \tilde{\xi}_{n} \|^{2+\eps} + \EB \| R^x_{n} \|^{2+\eps} = o(\alpha_{n}), \label{eq:lem:generator-twotime:f5}
		\end{align}
		where we employ Assumption~\ref{assump:mart-decouple-rate} and Lemma~\ref{lem:check-twotime-xy} with $\eps \le \frac{4}{1 + (\Hsmooth \vee \Fsmooth \vee \Gsmooth) / 2} - 2$ for the last inequality.
		Plugging \eqref{eq:lem:generator-twotime:f1} to \eqref{eq:lem:generator-twotime:f5} into \eqref{eq:lem:generator-twotime:f} yields
		\begin{equation*}
			\EB [ f (\xcheck_{n+1}) - f(\xcheck_{n}) | \FM_{n} ]
			= \alpha_{n} \AM^x f(\xcheck_{n})
			+ R^f_{n},
		\end{equation*}
		where
		$R^f_{n} := \ssum{i}{0}{4} R^f_{n,i}$ satisfies $\frac{1}{\alpha_{n}} \EB | R^f_{n} | \to 0$ as $n \to \infty$. Thus we obtain Part~(i).
		\vspace{0.2cm}
		
		Note that aside from the assumptions, the proof of Part (i) primarily relies on the fast-time-scale part of Lemma~\ref{lem:check-twotime-xy}. Similarly, the proof of part (ii) follows an analogous procedure, leveraging the results in Lemma~\ref{lem:check-twotime-z}. A key aspect worth emphasizing is about the asymptotic covariance. In the proof of Part (i), it is established that
		$\EB [ \tilde{\xi}_{n} \tilde{\xi}_{n}^\top | \FM_{n} ] \overset{p}{\to} \Sigma_\xi $,
		indicating that the asymptotic covariance of $\tilde{\xi}_{n}$ matches that of $\xi_{n}$.
		In contrast, for Part (ii), based on the definition in Lemma~\ref{lem:check-twotime-z}, $\tilde{\psi}_{n}$ can be expressed as
		\begin{equation*}
			\tilde{\psi}_{n} = \psi_{n} - B_2 B_1^{-1} \tilde{\xi}_{n}
			= \psi_{n} - B_2 B_1^{-1} \xi_{n} + \kappa_{n} B_2 B_1^{-1} H^\star \psi_{n}.
		\end{equation*}
		It follows that $\EB [ \tilde{\psi}_{n} \tilde{\psi}_{n}^\top | \FM_{n} ] \overset{p}{\to} \widetilde{\Sigma}_\psi$, where $\widetilde{\Sigma}_{\psi}$ is defined in \eqref{eq:cov-z}.
		Using a similar procedure as in the proof of Part (i), we can show that for any $g \in C^\infty_c (\RB^{d_y})$,
		\begin{equation*}
			\EB [ g (\zcheck_{n+1}) - g(\zcheck_{n}) | \FM_{n} ]
			= \beta_{n} \AM^y g(\zcheck_{n})
			+ R^g_{n},
		\end{equation*}
		where $R^g_{n}$ satisfies  $\frac{1}{\beta_{n}} \EB | R^g_{n} | \to 0$ as $n \to \infty$.
	\end{proof}
	
	With Lemma~\ref{lem:generator-twotime}, we are prepared to prove Lemma~\ref{lem:mart-prob-twotime}.
	
	\begin{proof}[Proof of Lemma~\ref{lem:mart-prob-twotime}]
		We first prove Part~(i).
		For convenience, with $\Gamma^\alpha_{n,m}$ introduced in Definition~\ref{defn:time-inter}, we define the following notation
		\begin{equation*}
			\bar{N}^\alpha(n, t) := \min \left\{ m \ge n \colon \Gamma^\alpha_{n, m} \ge t \right\}.
		\end{equation*}
		Then the filtration defined in \eqref{eq:filtration-main} can be written succinctly as $\FM^\xbar_{n,t} = \FM_{ \bar{N}^\alpha(n, t) }$.
		We construct $\MM_{n}^f(t)$ as
		\begin{equation}\label{eq:Mnf-def}
			\MM_{n}^f(t) := \ssum{k}{n}{\bar{N}^\alpha(n,t)-1} \big\{ f(\xcheck_{k+1}) - f(\xcheck_{k}) - \EB[ f(\xcheck_{k+1}) - f(\xcheck_{k}) | \FM_{k} ] \big\}.
		\end{equation}
		Clearly, $\MM_{n}^f(t)$ is adapted to the filtration $(\FM^\xbar_{n,t})_{t \ge 0}$.
		Next we verify that $\MM_{n}^f(t)$ is a martingale w.r.t. $(\FM^\xbar_{n,t})_{t \ge 0}$.
		
		For any $t > s \ge 0$,
		if $k \le \bar{N}^\alpha(n,s)-1$, then $\xcheck_{k}, \xcheck_{k+1} \in \FM^\xbar_{n,s}$ and $\FM_k \subset \FM^\xbar_{n,s}$, implying
		\[ \EB \big\{ f(\xcheck_{k+1}) {-} f(\xcheck_{k}) {-} \EB[ f(\xcheck_{k+1}) {-} f(\xcheck_{k}) | \FM_{k} ] \big| \FM^\xbar_{n,s} \big\} {=} f(\xcheck_{k+1}) {-} f(\xcheck_{k}) {-} \EB[ f(\xcheck_{k+1}) {-} f(\xcheck_{k})  | \FM_{k} ].  \]
		Otherwise, $k \ge \bar{N}^\alpha(n,s) $ and  $\FM^\xbar_{n,s} \subseteq \FM_{k}$, implying
		\begin{equation*}
		\EB \big\{ f(\xcheck_{k+1}) - f(\xcheck_{k}) - \EB[ f(\xcheck_{k+1}) - f(\xcheck_{k}) | \FM_{k} ]\, \big|\, \FM^\xbar_{n,s} \big\} = 0.
		\end{equation*}
		Consequently,
		\begin{align*}
			\EB [ \MM_{n}^f(t) | \FM^\xbar_{n,s} ]
			& = \ssum{k}{n}{\bar{N}^\alpha(n,t)-1} \EB \big\{ f(\xcheck_{k+1}) - f(\xcheck_{k}) - \EB[ f(\xcheck_{k+1}) - f(\xcheck_{k}) | \FM_{k} ]\, \big|\, \FM^\xbar_{n,s} \big\} \\
			& = \ssum{k}{n}{\bar{N}^\alpha(n,s)-1} \big\{ f(\xcheck_{k+1}) - f(\xcheck_{k}) - \EB[ f(\xcheck_{k+1}) - f(\xcheck_{k}) | \FM_{k} ] \big\} 
			= \MM_{n}^f(s),
		\end{align*}
		which implies that $\MM_{n}^f (t)$ is a martingale w.r.t.\ $\FM^\xbar_{n,t}$.
		
		Now we focus on the residual term $\RM_{n}^f(t)$. For notational simplicity, we define $\overline{t}^\alpha_{n} := \Gamma^\alpha_{n, \bar{N}^\alpha(n,t)} 
		$.
		Then we have $\overline{t}^\alpha_{n} - \alpha_{\bar{N}^\alpha(n,t) - 1} < t \le \overline{t}^\alpha_{n}$.
		With this notation and $\underline{t}^\alpha_{n}$ introduced in Definition~\ref{defn:time-inter},\footnote{In the notation $\underline{t}^\alpha_{n}$, we view $t$ as a variable.}
		$\RM_{n}^f(t)$ can be expressed as
		\begin{align}
			\RM_{n}^f(t) 
			& = f( \xbar_{n}(t) ) - f(\xbar_{n}(0) ) - \int_0^t \AM^x f( \xbar_{n}(s) )\, \d s - \MM_{n} ^f(t)  \nonumber \\
			& \overset{\eqref{eq:Mnf-def}}{=} f(\xbar_{n}(t)) - f(\xbar_{n}(\overline{t}^\alpha_{n}) ) - \int_0^t \AM^x f(\xbar_{n}(s)) \,\d s  + \ssum{k}{n}{\bar{N}^\alpha(n,t)-1} \EB [ f(\xcheck_{k+1}) - f(\xcheck_{k}) | \FM_{k} ] \nonumber \\
			& \overset{\eqref{eq:generator-twotime:x}}{=} f(\xbar_{n}(t)) - f(\xbar_{n}(\overline{t}^\alpha_{n}) )
			- \int_0^t \AM^x f(\xbar_{n}(s)) \,\d s
			+ \ssum{k}{n}{\bar{N}^\alpha(n,t)-1} \left( \alpha_{k} \AM^x f(\xcheck_{k}) + R^f_k \right) \nonumber \\
			& \overset{(a)}{=}
			f(\xbar_{n}(t)) - f(\xbar_{n}(\overline{t}^\alpha_{n}) )
			+ \int_0^{\overline{t}^\alpha_{n}} \left( \AM^x f(\xbar_{n}(\underline{s}^\alpha_{n}) ) - \AM^x f(\xbar_{n}(s)) \right) \,\d s  \nonumber \\
			& \qquad   + \int_t^{\overline{t}^\alpha_{n}} \AM^x f(\xbar_{n}(s)) \,\d s
			+ \ssum{k}{n}{\bar{N}^\alpha(n,t) - 1} R^f_k, \label{eq:them:fclt-twotime:res}
		\end{align}
		where for the second equality, the telescoping component $f(\xcheck_{k+1}) - f(\xcheck_{k})$ in \eqref{eq:Mnf-def} helps us simplify the expression of $\MM_{n}^f(t) $ as 
        \begin{align*}
            \MM_{n}^f(t)
            & = f(\xcheck_{\bar{N}^\alpha(n,t)} ) - f(\xcheck_{n}) - \ssum{k}{n}{\bar{N}^\alpha(n,t)-1} \big\{ \EB[ f(\xcheck_{k+1}) - f(\xcheck_{k}) | \FM_{k} ] \big\} \\
            & = f(\xbar_{n}(\overline{t}^\alpha_{n}) ) - f(\xbar_{n}(0)) - \ssum{k}{n}{\bar{N}^\alpha(n,t)-1} \big\{ \EB[ f(\xcheck_{k+1}) - f(\xcheck_{k}) | \FM_{k} ] \big\};
        \end{align*} 
        for step (a), we used Fact~\ref{fact2}~\ref{fact2-3} such that  $\alpha_{k} \AM^x f(\xcheck_{k}) = \int_{\Gamma^\alpha_{n,k}}^{\Gamma^\alpha_{n,k+1}}  \AM^x f( \xbar_{n}( \underline{s}^\alpha_n ) ) \,\d s  $.
		
		Notice that  $f \in C^\infty_c (\RB^{d_x})$.
		For a fixed $t$,
		to control $\RM_{n}^f(t)$ in the $L_1$ sense, we need to control the following $L_1$ distances $\EB \| \xbar_{n}(s) - \xbar_{n}( \underline{s}^\alpha_{n} ) \| $ and  $\EB \| \xbar_{n}(s) - \xbar_{n}( \overline{s}^\alpha_{n} ) \|$ for any $s \in [0,\overline{t}_{n}^\alpha]$. 
		From the definition of $\xbar_{n}(\cdot)$ in \eqref{eq:def:xbar}, we have $\xbar_{n}( \underline{s}^\alpha_{n} ) = \xcheck_{N^\alpha(n,s)}$. Consequently, 
		\begin{equation*}
			\xbar_{n}(s) - \xbar_{n}( \underline{s}^\alpha_{n}  )
			= (s - \underline{s}^\alpha_{n} ) \left(- B_1 \xcheck_{N^\alpha(n,s) } + \frac{R^x_{N^\alpha (n,s)} }{ \alpha_{N^\alpha (n,s)} } \right) -  \frac{s - \underline{s}^\alpha_{n}}{ \sqrt{ \alpha_{N^\alpha(n,s)} } }  
			\tilde{\xi}_{N^\alpha(n,s)}. 
		\end{equation*}
		It follows that for each $s \in [0, \overline{t}_{n}^\alpha] $, the following results hold
		\begin{align*}
			\EB \| \xbar_{n}(s) - \xbar_{n}( \underline{s}^\alpha_{n} ) \| 
			& \le (s - \underline{t}^\alpha_{n} ) \EB \left\| - B_1 \xcheck_{N^\alpha(n,s) } + \frac{R^x_{N^\alpha (n,s)} }{ \alpha_{N^\alpha (n,s)} } \right\| +  \frac{s - \underline{s}^\alpha_{n}}{ \sqrt{ \alpha_{N^\alpha(n,s)} } }  
			\EB \| \tilde{\xi}_{N^\alpha(n,s)} \| \\
			& \overset{\eqref{eq:lem:tight-twotime:l2bound-x}}{\lesssim} (s - \underline{s}^\alpha_{n} ) + \sqrt{s - \underline{s}^\alpha_{n} } 
			\lesssim \sqrt{ \alpha_{N^\alpha(n,s)} }, \\
			\EB \| \xbar_{n}(s) - \xbar_{n}( \overline{s}^\alpha_{n} ) \| 
			& \le \EB \| \xbar_{n}(s) - \xbar_{n}( \underline{s}^\alpha_{n} ) \| + \EB \| \xbar_{n}( \overline{s}^\alpha_{n} ) - \xbar_{n}( \underline{s}^\alpha_{n} ) \| \\
			& \le \EB \| \xbar_{n}(s) - \xbar_{n}( \underline{s}^\alpha_{n} ) \| + ( \overline{s}^\alpha_{n} - \underline{s}^\alpha_{n} ) \EB \left\| - B_1 \xcheck_{N^\alpha(n,s) } + \frac{R^x_{N^\alpha (n,s)} }{ \alpha_{N^\alpha (n,s)} } \right\| \\
			& \qquad +  \frac{ \overline{s}^\alpha_{n} - \underline{s}^\alpha_{n}}{ \sqrt{ \alpha_{N^\alpha(n,s)} } }  
			\EB \| \tilde{\xi}_{N^\alpha(n,s)} \| \\
			& \lesssim \sqrt{ \alpha_{N^\alpha(n,s)} }.
		\end{align*}
        Since $f\in C_c^\infty(\RB^{d_x})$,  let $K:=\mathrm{supp}(f)$ be its compact support.
Then there exist finite constants
\[
R:=\sup_{x\in K}\|x\|,\ 
L_1:=\sup_{x}\|\nabla f(x)\|,\ 
L_2:=\sup_{x}\|\nabla^2 f(x)\|,\ 
L_3:=\sup_{x\neq y}\frac{\|\nabla^2 f(x)-\nabla^2 f(y)\|}{\|x-y\|}.
\]
In particular, $f$ is $L_1$-Lipschitz, $\nabla f$ is $L_2$-Lipschitz and $\nabla^2 f$ is $L_3$-Lipschitz.
Recall $
\AM^x f(x)=\langle -B_1 x,\nabla f(x)\rangle+\frac{1}{2}\tr\!\big(\nabla^2 f(x)\,\Sigma_\xi\big)$.
It is bounded in the sense that
\[
 |\AM^x f(x)| \le \|B_1\| R L_1 + \frac{1}{2} \| \nabla^2 f(x) \|_F \| \Sigma_\xi \|_F \le \|B_1\| R L_1 + \frac{1}{2} d_x L_2 \| \Sigma_\xi \|.
\]
For the trace term, Lipschitz continuity follows from that of $\nabla^2 f$:
\[
\left|\tr\!\big((\nabla^2 f(x)-\nabla^2 f(y))\,\Sigma_\xi\big)\right|
\le \|\Sigma_\xi \|_F\,\|\nabla^2 f(x)-\nabla^2 f(y)\|_F
\le d L_3 \| \Sigma_\xi\| \|x-y\|.
\]
For the drift term, we have
\begin{align*}
    \big|\langle B_1 x,\nabla f(x)\rangle-\langle B_1 y,\nabla f(y)\rangle\big|
    & \le \|B_1\|\Big(\|x\|\,\|\nabla f(x)-\nabla f(y)\|+\|x-y\|\,\|\nabla f(y)\|\Big) \\
    & \le \|B_1 \| R L_2  \| x-y \| + \|B_1 \| L_1 \|x-y \|.
\end{align*}
Combining the two parts shows that $\AM^x f$ is Lipschitz.
The above results can be summarized as 
\begin{equation*}
    |f(x) - f(y)| \lesssim\| x - y \|,\quad 
    |\AM^x f(x)| \lesssim 1 ,\quad
    |\AM^x f(x) - \AM^x f(y)| \lesssim  \| x - y \|.
\end{equation*}       
        Then we have
		\begin{align}\label{eq:thm:fclt-twotime:res1}
			\begin{split}
				\EB | f(\xbar_{n}(t)) - f(\xbar_{n}(\overline{t}^\alpha_{n}) ) |
				& \lesssim \EB \| \xbar_{n}(t) - \xbar_{n}(\overline{t}^\alpha_{n}) \| \lesssim \sqrt{\alpha_{N^\alpha(n,t)}} = o(1), \\
				\EB \left| \int_t^{\overline{t}^\alpha_{n}} \AM^x f(\xbar_{n}(s)) \,\d s \right|
				& \lesssim \int_t^{\overline{t}^\alpha_{n}} 1 \,\d s 
				= o(1), \\
				\EB \left| \int_0^{\overline{t}^\alpha_{n}} \left( \AM^x f(\xbar_{n} (\underline{s}^\alpha_{n}) ) - \AM^x f(\xbar_{n}(s)) \right) \,\d s \right|
				& \lesssim \int_0^{\overline{t}^\alpha_{n}} \EB \| \xbar_{n}(s) - \xbar_{n} (\underline{s}^\alpha_{n} ) \| \,\d s \\
				&   \lesssim \int_0^{\overline{t}^\alpha_{n}} \sqrt{\alpha_{N^\alpha(n,s)} } ds \lesssim \sqrt{\alpha_{n}} = o(1).
			\end{split}
		\end{align} 
		Moreover, for a fixed $t$, Lemma~\ref{lem:generator-twotime} implies
		\begin{equation}\label{eq:thm:fclt-twotime:res2}
			\EB \left| \ssum{k}{n}{\bar{N}^\alpha(n,t) - 1} R^f_k \right|
			\le \ssum{k}{n}{\bar{N}^\alpha(n,t) - 1} \alpha_{k} \EB \left| \frac{R^f_k}{\alpha_{k}} \right|
			\le \sup_{k' \ge n} \EB \left| \frac{R^f_{k'}}{\alpha_{k'}} \right| \cdot \ssum{k}{n}{\bar{N}^\alpha(n,t) - 1} \alpha_{k} \lesssim o(1).
		\end{equation}
		Plugging \eqref{eq:thm:fclt-twotime:res1} and \eqref{eq:thm:fclt-twotime:res2} into \eqref{eq:them:fclt-twotime:res} yields that for a fixe $t$, $\EB | \RM_{n}^f(t) | \to 0$ as $n \to \infty$.
		\vspace{0.2cm}
		
		For Part~(ii), we similarly define
		\begin{equation*}
			\bar{N}^\beta(n, t) := \min \left\{ m \ge n \colon \Gamma^\beta_{n, m} \ge t \right\} 
		\end{equation*}
		and
		\begin{equation*}
			\MM_{n}^g(t) = \ssum{k}{n}{\bar{N}^\beta(n,t)-1} \big\{ f(\zcheck_{k+1}) - f(\zcheck_{k}) - \EB[ f(\zcheck_{k+1}) - f(\zcheck_{k}) | \FM_{k} ] \big\}.
		\end{equation*}
		The remaining proof is similar to that of Part~(i).
	\end{proof}
	
	\subsection{Proofs in Section~\ref{sec:fclt:intergrate}}
	\label{sec:fclt:proof:intergrate}
	In this subsection, we present the detailed proof for the procedure in Section~\ref{sec:fclt:intergrate}. 
	We first give the formal proof of Theorem~\ref{thm:fclt-twotime}~\ref{thm:fclt-x} based on Proposition~\ref{prop:mart-prob-approx} and the following proposition, which guarantees the geometric ergodicity of the solutions of the SDEs in \eqref{eq:fclt-twotime:x-sde} and \eqref{eq:fclt-twotime:y-sde}.
	
	\begin{prop}[Geoometric ergodicity for SDEs 
    ]\label{prop:ergo}
		Under Assumption~\ref{assump:hurwitz}, the SDE in \eqref{eq:fclt-twotime:x-sde} has a unique invariant distribution $\pi=\NM(0, \Sigma_x)$, where $\Sigma_x$ is the unique solution to \eqref{eq:fclt-twotime:x-var}.
        Moreover, suppose $\ermX(t)$ is a solution, for any compact set $K \subset \RB^{d_x}$ and function $f \in C^\infty_c(\RB^{d_x})$. Then there exist positive numbers $A$ and $b$ such that
		\begin{equation*}
			\sup_{x \in K}  | \EB[ f(\ermX(t)) | X(0) = x ] - \EB_{\ermX \sim \pi^\star} f(\ermX)  | \le A e^{-bt}, \ \forall t \ge 0.
		\end{equation*}
		Similar results also hold for the SDE in \eqref{eq:fclt-twotime:y-sde}.
	\end{prop}

\begin{proof}[ Proof of Proposition~\ref{prop:ergo} ]
Without loss of generality, we focus on the following SDE
\begin{equation}\label{eq:ergo-sde}
    d \ermX(t) = -A \ermX(t) dt + B d \ermW(t),
\end{equation}
where $A, B \in \RB^{d \times d}$, $-A$ is a Hurwitz matrix, 
and $\ermW(t)$ is an $d$-dimensional standard Brownian motion.
For the SDE in \eqref{eq:fclt-twotime:x-sde}, $A=B_1$ and $B = \Sigma_\xi^{1/2}$, while for the SDE in \eqref{eq:fclt-twotime:y-sde}, $A = B_3 - \frac{\invdiffslow I}{2}$ and $B = \widetilde{\Sigma}_\psi^{1/2}$.
By Proposition~\ref{prop:hurwitz-expbound}, we have
\begin{equation}\label{eq:M_b}
\|e^{-At}\| \le M e^{-b t},\forall t\ge0,\ \text{with} \ M := \sqrt{\frac{\lambda_{\max}(P)}{\lambda_{\min}(P)} }  b := \frac{1}{2\lambda_{\max}(P)},
\end{equation}
where $P := \int_0^\infty e^{-A^\top t} e^{-At} dt$.

\textbf{ Step 1: Uniqueness of invariant distribution. }

Let
\[\Sigma_t:=\int_0^t e^{-As}BB^{\top}e^{-A^{\top}s}\,ds,\qquad \Sigma:=\int_0^{\infty} e^{-As}BB^{\top}e^{-A^{\top}s}\,ds.\]
By \eqref{eq:M_b}, $\Sigma$ is well-defined.
The solution $X(t)$ of \eqref{eq:ergo-sde} has the explicit representation
\[X(t) = e^{-At} X(0) + \int_0^t e^{-A(t-s)}B\,dW(s),\]
so conditionally on $X(0)=x$, $X(t)\sim\mathcal N(e^{-At}x,\,\Sigma_t)$.
\\
We first prove that $\pi:=\mathcal N(0,\Sigma)$ is an invariant distribution.
Let $\mu_t^x = \mathcal N(e^{-At}x,\,\Sigma_t)$ denote the distribution of $X(t)$ when $X_0=x$. To check the invariance of $\pi$, it suffices to check for its characteristic function
\(\widehat{\pi}(\theta)=\exp\big(-\tfrac12\theta^{\top}\Sigma\theta\big)\).  
From the definition of $\Sigma_t$ and $\Sigma$, we have
$$\Sigma-\Sigma_t=\int_t^{\infty} e^{-As}BB^{\top}e^{-A^{\top}s}\,ds = e^{-At}\Big(\int_0^{\infty} e^{-As}BB^{\top}e^{-A^{\top}s}\,ds\Big)e^{-A^{\top}t}=e^{-At}\Sigma e^{-A^{\top}t}.$$
Then the characteristic function of $X(t)$ under initial distribution $\pi$ equals
\begin{align*}
\mathbb E_{X(0)\sim\pi}[e^{i\theta^{\top}X(t)}] 
&= \mathbb E_{X(0)\sim\pi}\big[ e^{i\theta^{\top}e^{-At}X(0)} \big]\cdot e^{-\tfrac12\theta^{\top}\Sigma_t\theta} 
= \widehat{\pi}(e^{-A^{\top}t}\theta)\, e^{-\tfrac12\theta^{\top}\Sigma_t\theta} \\
& = \exp\Big(-\tfrac12\theta^{\top}(e^{-At}\Sigma e^{-A^{\top}t}+\Sigma_t)\theta\Big) =  \exp\Big(-\tfrac12\theta^{\top}\Sigma\theta\Big) = \widehat{\pi}(\theta)
\end{align*}
Thus the characteristic function of the distribution of $X_t$ under initial law $\pi$ equals $\widehat{\pi}(\theta)$ for every $t$, showing that $\pi$ is the invariant distribution. 

Next, we verify the uniqueness.
Let $\mu$ be an arbitrary invariant distribution of \eqref{eq:ergo-sde}. Then for every bounded continuous test function $f$ and $t>0$,
\[\int \mathbb E[f(X_t)\mid X_0=x]\,\mu(dx) = \int f\,d\mu.\]
Then the characteristic function $\widehat\mu$ of $\mu$ satisfies that
for all $\theta\in\mathbb R^d$ and all $t>0$,
\begin{equation*}
\widehat{\mu}(\theta) = 
\mathbb E_{X(0)\sim\mu}[e^{i\theta^{\top}X(t)}] 
= \mathbb E_{X(0)\sim\mu}\big[ e^{i\theta^{\top}e^{-At}X(0)} \big]\cdot e^{-\tfrac12\theta^{\top}\Sigma_t\theta} 
= \widehat{\mu}(e^{-A^{\top}t}\theta)\, e^{-\tfrac12\theta^{\top}\Sigma_t\theta}
\end{equation*}
Fix $\theta\in\mathbb R^d$. Let $t\to\infty$. The exponential bound in \eqref{eq:M_b} implies $e^{-A^{\top}t}\theta\to0$ and $\Sigma_t\to\Sigma$. Moreover, the characteristic function $\widehat{\mu}$ is continuous at $0$ with $\widehat{\mu}(0)=1$, leading to 
\(\widehat{\mu}(\theta) = 1\cdot e^{-\tfrac12\theta^{\top}\Sigma\theta}.\)
Because this holds for every $\theta$, $\widehat{\mu}$ coincides with the characteristic function of $\mathcal N(0,\Sigma)$, hence $\mu=\mathcal N(0,\Sigma)=\pi$. This proves uniqueness.
Moreover, $\Sigma$ is the unique solution to the Lyapunov equation $A X + X A^\top = BB^\top$ \cite[Theorem~7.5]{antsaklis2005linear}.

\textbf{ Step 2: Exponential ergodicity. } 

For convenience, we denote
\[R:=\sup_{x\in K}\|x\|,\quad L_1:=\sup_{y}|\nabla f(y)|,\quad L_2:=\sup_{y}\|\nabla^2 f(y)\|,\]
Then we have
\begin{align}
\|\Sigma-\Sigma_t\| 
&= \Big\|\int_t^{\infty} e^{-As}BB^{\top}e^{-A^{\top}s}\,ds\Big\|_2 
\le \|BB^{\top}\| \int_t^{\infty}\|e^{-As}\|^2\,ds \nonumber \\
& \le \|BB^{\top}\| M^2 \int_t^{\infty} e^{-2b s}\,ds 
= C_{\Sigma} e^{-2bt},\label{eq:sigma_est}
\end{align}
where $C_{\Sigma}:=\frac{\|BB^{\top}\|_2 M^2}{2b}$ and $b$ is defined in \eqref{eq:M_b}.
Then we have the following decomposition 
\begin{align}
& \quad \mathbb E[f(X(t))\mid X(0)=x] - \mathbb E_{X\sim\pi}f(X) = \mathbb E\big[f(e^{-At}x + Y_t)\big] - \mathbb E\big[f(Y_{\infty})\big] \nonumber \\
&= \big(\mathbb E[f(e^{-At}x+Y_t)] - \mathbb E[f(Y_t)]\big) + \big(\mathbb E[f(Y_t)] - \mathbb E[f(Y_{\infty})]\big), \label{eq:ergo-decom}
\end{align}
where $Y_t\sim\mathcal N(0,\Sigma_t)$ and $Y_{\infty}\sim\mathcal N(0,\Sigma)$.

To control the first term, we apply a second-order Taylor expansion around $y$.
For any fixed $y$ and vector $m:=e^{-At}x$, there exists a point $\tilde y$ on the line segment between $y$ and $y+m$ such that
\[f(y+m)-f(y) = \nabla f(y)\cdot m + \tfrac12 m^{\top}\nabla^2 f(\tilde y)m.
\]
Take expectation w.r.t. $Y_t$ and use uniform bounds $L_1,L_2$. For $x\in K$ with $\|x\|\le R$,
\begin{align}
\big|\mathbb E[f(e^{-At}x+Y_t)] - \mathbb E[f(Y_t)]\big| &\le L_1 \|e^{-At}\| \|x\| + \tfrac12 L_2 \|e^{-At}\|^2 \|x\|^2 \nonumber \\
&\le L_1 R M e^{-b t} + \tfrac12 L_2 R^2 M^2 e^{-2b t}. \label{eq:ergo-decom-1}
\end{align}
\indent
For the second term, we use Gaussian interpolation: for $s\in[0,1]$ set $\Sigma'_s:=\Sigma + s(\Sigma_t-\Sigma)$ and let $Y'_s\sim\mathcal N(0,\Sigma'_s)$. The map $s\mapsto \mathbb E[f(Y'_s)]$ is differentiable and \citep[Lemma~7.2.7]{vershynin2018high}
\[\frac{d}{ds}\mathbb E[f(Y'_s)] = \tfrac12\mathbb E\big[\mathrm{tr}((\Sigma_t-\Sigma)\nabla^2 f(Y'_s))\big].\]
Integrating $s$ from $0$ to $1$ gives
\begin{equation}
\big|\mathbb E[f(Y_t)] - \mathbb E[f(Y_{\infty})]\big| 
\le \tfrac12 \| \Sigma_t - \Sigma \|_F \EB \| \nabla^2 f(Y'_s) \|_F
\le \tfrac12 dL_2 \|\Sigma_t-\Sigma\|_2 
\overset{\eqref{eq:sigma_est}}{\le} \tfrac12 dL_2C_{\Sigma}  e^{-2bt}.\label{eq:ergo-decom-2}
\end{equation}

Substituting \eqref{eq:ergo-decom-1} and \eqref{eq:ergo-decom-2} into \eqref{eq:ergo-decom} yields, for all $x\in K$,
\begin{equation*}
\sup_{x\in K}\big|\mathbb E[f(X_t)\mid X_0=x] - \mathbb E_{\pi}f\big| 
\le L_1 R M e^{-b t} + \Big(\tfrac12 L_2 R^2 M^2 + \tfrac12 dL_2 C_{\Sigma}\Big) e^{-2b t}
\le Ae^{-b t}
\end{equation*}
where
\begin{equation}\label{eq:A_constant}
A := L_1 R M + \tfrac12 L_2 R^2 M^2 + \tfrac12 d L_2 C_{\Sigma},
\end{equation}
This proves the theorem with explicit constants $b,A$ given in \eqref{eq:M_b} and \eqref{eq:A_constant}.
\end{proof}

	\begin{proof}[Proof of Theorem~\ref{thm:fclt-twotime}~\ref{thm:fclt-x}]
		The proof is divided into three steps.
		In the first step, we apply Proposition~\ref{prop:mart-prob-approx} to establish the convergence of convergent subsequences.
		In the second step, we verify Condition~\ref{prop:final:condition:initial} in Proposition~\ref{prop:mart-prob-approx}.
		Finally, we apply Proposition~\ref{prop:mart-prob-approx} to establish the convergence of the whole sequence.
		\vspace{0.2cm}
		
		\textbf{Step~1: Establishing subsequence convergence}.
		For the conditions in Proposition~\ref{prop:mart-prob-approx},
		Condition~\ref{prop:final:condition:unique} has been discussed in Section~2.2.
		Lemma~\ref{lem:tight-twotime} ensures that Condition~\ref{prop:final:condition:tight} holds.
		To verify Condition~\ref{prop:final:condition:mart-prob}, we apply Lemma~\ref{lem:mart-prob-twotime}, which shows that
		for any $f \in C^\infty_c (\RB^{d_x})$ and $t \ge 0$, one has that
		\begin{equation*}
			f(\xbar_{n}(t)) - f(\xbar_{n}(0)) - \int_0^{t} \AM^x f(\xbar_{n}(s)) ds
			= \MM_{n}^f(t) + \RM_{n}^f(t), 
		\end{equation*}
		where for each $n$, $\MM_{n}^f(t)$ is a martingale w.r.t.\ $(\FM^\xbar_{n,t})_{t \ge 0}$ defined in \eqref{eq:filtration-main} and for each $t$, $\EB | \RM_{n}^f(t) | \to 0$ as $n \to \infty$. 
		Then for any $k \ge 0$, $0 \le t_1 < t_2 < \dots < t_k \le t \le t+s $, $f \in C^\infty_c (\RB^{d_x})$ and $h_1, h_2, \dots, h_k \in C_c (\RB^{d_x})$
		\begin{align*}
			& \quad \ \Bigg|\, \EB  \bigg[ \Big( f(\xbar_{n}(t+s)) - f(\xbar_{n}(t)) - \int_t^{t+s} \AM^x f(\xbar_{n}(r)) dr \Big) \prod_{i=1}^k h_i(\xbar_{n}(t_i)) \bigg]\, \Bigg| \\
			& = \Bigg|\, \EB \bigg[ \big(\MM_{n}^f(t+s) - \MM_{n}^f(t) + \RM_{n}^f(t+s) - \RM_{n}^f(t) \big) \prod_{i=1}^k h_i(\xbar_{n}(t_i)) \bigg] \, \Bigg| \\
			& = \Bigg|\, \EB \bigg\{ \prod_{i=1}^k h_i(\xbar_{n}(t_i))\, \EB \big[ \MM_{n}^f(t+s) - \MM_{n}^f(t) + \RM_{n}^f(t+s) - \RM_{n}^f(t) \,\big|\, \FM_{n,t}^\xbar \big]  \bigg\} \, \Bigg| \\
			& \le  \EB \bigg\{ \prod_{i=1}^k | h_i(\xbar_{n}(t_i)) | \, \EB \big[ | \RM_{n}^f(t+s) | + | \RM_{n}^f(t) | \,\big|\, \FM_{n,t}^\xbar \big]  \bigg\} \to 0
		\end{align*}
		as $n \to \infty$.
		Now we assume $ \{ \xbar_{n_k} (\cdot) \}_{k=1}^\infty $ is a weakly convergent subsequence of $ \{ \xbar_{n} (\cdot) \}_{n=1}^\infty$.
		Clearly, this subsequence satisfies Condition~\ref{prop:final:condition:initial} in Proposition~\ref{prop:mart-prob-approx}.
		Suppose $\xbar_{n_k} (0) \weakconverge \tilde{\pi}$. Then Proposition~\ref{prop:mart-prob-approx} ensures that $\xbar_{n_k}(\cdot) \weakconverge \ermX^{\tilde{\pi}}(\cdot)$ as $k \to \infty$, where $\ermX^{\tilde{\pi}}(\cdot)$ is the solution of the SDE in \eqref{eq:fclt-twotime:y-sde} with the initial distribution $\tilde\pi$.
		\vspace{0.2cm}
		
		\textbf{Step~2: Verifying initial conditions}. 
		To establish the convergence of $\xbar_{n}(0)$, it suffices to prove that all the convergent subsequences have the same limit.
		In fact, this limit is $\pi^\star$, the unique invariant distribution of \eqref{eq:fclt-twotime:x-sde} . The uniqueness is guaranteed by Proposition~\ref{prop:ergo}.
		In the next, we prove that if the subsequence $\xbar_{n_k}(0) \weakconverge \tilde{\pi}$, one must have $\tilde{\pi} = \pi^\star$. 
		The proof is divided to the following two parts: (i) for any fixed $T > 0$, we show that there exists a solution $\ermX(t)$ to \eqref{eq:fclt-twotime:x-sde} such that $\tilde{\pi}$ has the same distribution as $\ermX(T)$;
		(ii) when $T$ is sufficiently large, the geometric ergodicity implies that the distribution of $\ermX(T)$ can be arbitrarily close to $\pi^\star$.
		
		\textbf{Part (i)}.
		We first introduce some additional notation.
		Recall that $\Gamma^\alpha_{n,n'} = \ssum{k}{n}{n'-1} \alpha_k$.
		For $n \in \NB$ and $t \ge 0$, we define 
		\begin{equation*}
			M^\alpha (n,t) := \min\left\{ m \ge 0 \colon \Gamma^\alpha_{m,n} \le t \right\}\ \text{ and } \
			\tilde{t}^\alpha_{n} := \Gamma^\alpha_{M^\alpha(n,t), n}.
		\end{equation*}
		Since $\alpha_{n} \to 0$, we have $\tilde{t}^\alpha_{n} \uparrow t$.
		
		For a fixed $T>0$, we consider the subsequence $\{ \xbar_{M^\alpha(n_k, T)}(\cdot) \}_{k=0}^\infty$, which satisfies 
		\begin{equation*}
			\xbar_{M^\alpha(n_k, T)}(0) = \xcheck_{M^\alpha(n_k, T)} \ \text{ and }\  \xbar_{M^\alpha(n_k, T)}( \widetilde{T}^\alpha_{n_k})= \xcheck_{n_k},
		\end{equation*}
		where $\widetilde{T}^\alpha_{n} := \Gamma^\alpha_{M^\alpha(n,T), n}$
		By Lemma~\ref{lem:tight-twotime}, $\{ \xbar_{n}(t) \}_{n=1}^\infty$ is tight. Then Prokhorov's theorem (i.e., Proposition~\ref{prop:tight}~\ref{prop:tight:prok}) implies that the subsequence $\{ \xbar_{M^\alpha(n_k, T)}(\cdot) \}_{k=0}^\infty$ further has a weakly convergent subsequence.
		Without loss of generality, we assume the subsequence $\{ \xbar_{M^\alpha(n_k, T)}(\cdot) \}_{k=0}^\infty$ itself weakly converges.
		Then its initial distribution sequence also converges weakly, i.e., $\xcheck_{M^\alpha(n_k, T)} \weakconverge \mu_T$ for some distribution $\mu_T$.
		Then by Proposition~\ref{prop:mart-prob-approx}, $\xbar_{M^\alpha(n_k, T)}(\cdot) \weakconverge \ermX^{\mu_T}(\cdot) $, where $\ermX^{\mu_T}(\cdot)$ is a solution to \eqref{eq:fclt-twotime:x-sde} with initial distribution $\mu_T$.
		Consequently, we also have $\xbar_{M^\alpha(n_k, T)}(T) \weakconverge \ermX^{\mu_T}(T)$.
		Moreover, recall that we have verified the second condition of Proposition~\ref{prop:tight-suff-slow} in the proof of Lemma~\ref{lem:tight-twotime}. Then $|T - \widetilde{T}^\alpha_{n_k}| \to 0$ implies that $\left\| \xbar_{M^\alpha(n_k, T)}(\widetilde{T}^\alpha_{n_k}) - \xbar_{M^\alpha(n_k, T)}(T) \right\| \overset{p}{\rightarrow} 0$ as $k \to \infty$. 
		By Slutsky's theorem,
		$\xbar_{M^\alpha(n_k, T)}(\widetilde{T}^\alpha_{n_k}) \weakconverge \ermX^{\mu_T}(T)$, i.e., $\xcheck_{n_k} \weakconverge \ermX^{\mu_T}(T)$. 
		Since $\xcheck_{n_k} \weakconverge \tilde\pi$, this implies that for any $T > 0$, $\ermX^{\mu_T}(T) \sim  \tilde{\pi}$.
		\vspace{0.1cm}
		
		\textbf{Part (ii)}.
		Next, we apply $\xcheck_{M^\alpha(n_k, T)} \weakconverge \mu_T$ and $\ermX^{\mu_T}(T) \sim  \tilde{\pi}$ to prove $\tilde\pi = \pi^\star $.
		Since $\{ \xcheck_{n} \}_{n=1}^\infty$ is tight,  for any $\eps > 0$, there exists a compact set $K_\eps \subset \RB^{d_x}$ only depending on $\eps$ such that $\sup_{n \ge 0} \PB(\xcheck_{n} \in K_\eps) \ge 1 - \eps$.
		It follows that $\mu_T(K_\eps) \ge 1-\eps$ for any $T > 0$.
		Moreover,
		Proposition~\ref{prop:ergo} implies that for a fixed $g \in C^\infty_c  (\RB^{d_x})$, we can find $T_\eps$ such that 
		\begin{equation}\label{eq:erg_diff}
			\sup_{\xx \in K_\eps} \left| \PM_{T_\eps}\, g(\xx)- \EB_{\ermX \sim \pi^\star}\, g(\ermX) \right| \le \eps,
		\end{equation}
		where $\PM_t$ represents a semigroup induced by the SDE in \eqref{eq:fclt-twotime:x-sde}, that is, 
		\begin{equation}\label{eq:semigroup}
			\PM_t f(x) = \EB [ f(\ermX(t)) \,|\, \ermX(0) = x ],
		\end{equation}
		where $\ermX(\cdot)$ is the solution to the SDE in \eqref{eq:fclt-twotime:x-sde}. 
		Setting $T = T_\eps$, we have $\ermX^{\mu_{T_\eps}}(T_\eps) \sim  \tilde{\pi}$.
		Since $\ermX^{\mu_{T_\eps}}(\cdot)$ is a solution to \eqref{eq:fclt-twotime:x-sde} with initial distribution $\mu_{T_\eps}$,  we have
		\begin{equation*}
			\EB_{\ermX \sim \tilde\pi}\, g(\ermX)
			= \EB\, [ g (\ermX^{\mu_{T_\eps}}(T_\eps) ) \,|\, \ermX(0) \sim \mu_{T_\eps} ]
			\overset{\eqref{eq:semigroup}}{=} \int_{\RB^{d_x}} \PM_{T_\eps}\, g(\xx) d \mu_{T_\eps}(\xx)
		\end{equation*}
		It follows that
		\begin{equation*}
			\begin{aligned}
				\left| \EB_{\ermX \sim \tilde\pi}\, g(\ermX) - \EB_{\ermX \sim \pi^\star}\, g(\ermX) \right|
				& \le \int_{\RB^{d_x}} \left| \PM_{T_\eps} g(\xx) -\EB_{\ermX \sim \pi^\star} g(\ermX) \right| d \mu_{T_\eps}(\xx)\\ 
				& = \int_{K_\eps} \left| \PM_{T_\eps} g(\xx) - \EB_{\ermX \sim \pi^\star} g(\ermX) \right| d \mu_{T_\eps}(\xx) \\
				& \qquad + \int_{K_\eps^c} \left| \PM_{T_\eps} g(\xx) - \EB_{\ermX \sim \pi^\star} g(\ermX) \right| d \mu_{T_\eps}(\xx) \\
				&\le \int_{K_\eps} \left| \PM_{T_\eps} g(\xx) - \EB_{\ermX \sim \pi^\star} g(\ermX) \right| d \mu_{T_\eps}(\xx) + 2 \|g\|_\infty \mu_{T_\eps} (K_\eps^c) \\ 
				& \overset{(a)}{\le} \eps + 2\|g\|_\infty \eps,
			\end{aligned}
		\end{equation*}
		where (a) is due to $\mu_{T_\eps}(K_\eps) \ge 1-\eps$ and \eqref{eq:erg_diff}.
		Letting $\eps \to \infty$, we obtain $\EB_{\ermX \sim \tilde\pi}\, g(\ermX) = \EB_{\ermX \sim \pi^\star}\, g(\ermX) $.
		For any $g \in C^\infty_c (\RB^{d_x})$, we can repeat the above procedure, implying $\tilde\pi = \pi^\star$.
		\vspace{0.2cm}
		
		\textbf{Step~3: Concluding whole sequence convergence}. This step is straightforward.
		In Steps~1 and 2, we have verified the all the conditions in Proposition~\ref{prop:mart-prob-approx}.
		Moreover, $\xbar_{n}(0)  \weakconverge \pi^\star$, the invariant distribution of \eqref{eq:fclt-twotime:x-sde}.
		Consequently, $\xbar_{n}(\cdot)$ converges weakly to the solution of \eqref{eq:fclt-twotime:x-sde} with the initial distribution $\pi^\star$, implying that this solution is a stationary process.
		Since $B_1$ is a Hurwitz matrix, the invariant distribution $\pi^\star=\NM(0, \Sigma_{x})$ with $\Sigma_x$ satisfying the equation in \eqref{eq:fclt-twotime:x-var} \cite[Section~6.5]{risken1996fokker}.
	\end{proof}
	
	The proof of Corollary~\ref{cor:fclt-z} is similar to that of Theorem~\ref{thm:fclt-twotime}~\ref{thm:fclt-x}.
	Next, we give the proof of Theorem~\ref{thm:fclt-twotime}~\ref{thm:fclt-y} based on Corollary~\ref{cor:fclt-z}.
	
	\begin{proof}[Proof of Theorem~\ref{thm:fclt-twotime}~\ref{thm:fclt-y}]
		Under the tightness in Lemma~\ref{lem:tight-twotime}, 
		Proposition~\ref{prop:tight} shows that it suffices to examine the finite-dimensional distributions of $\ybar_{n}(\cdot)$.
		
		Recall that $\zcheck_{n}$ is defined as $\zcheck_{n} = \ycheck_{n} - \sqrt{ \kappa_{n-1} } B_2 B_1^{-1} \xcheck_{n}$ with $\kappa_{n-1} = \beta_{n-1} / \alpha_{n-1} = o(1)$; Assumption~\ref{assump:mart-decouple-rate} implies $\EB \| \xcheck_{n} \|^2 = \OM(1)$.
		It follows that $\ycheck_{n} - \zcheck_{n} \overset{p}{\to} 0$.
		Meanwhile, from the definition in \eqref{eq:def:ybar} and \eqref{eq:def:zbar}, for each fixed $t$, $\ybar_{n}(t) - \zbar_{n}(t)$ is the linear combination of $\ycheck_{N^\beta(n,t)} -  \zcheck_{N^\beta(n,t)}$ and $\ycheck_{N^\beta(n,t)+1} -  \zcheck_{N^\beta(n,t)+1}$, both converging to the zero vector in probability.
		Consequently, $\ybar_{n}(t) - \zbar_{n}(t) \overset{p}{\to} 0$. 
		Similarly, we could prove that for any positive integer $k$ and $t_1, \ldots, t_k$,
		$(\ybar_{n}(t_1) - \zbar_{n}(t_1), \dots, \ybar_{n}(t_k) - \zbar_{n} (t_k) ) \overset{p}{\to} 0$.
		Applying Slutsky's theorem and Theorem~\ref{thm:fclt-twotime}, we obtain that the finite-dimensional distributions of $\ybar_{n}(\cdot)$ converge weakly to the finite-dimensional distributions of the stationary distribution of \eqref{eq:fclt-twotime:y-sde}.
		Consequently, $\ybar_{n}(\cdot)$ converges weakly to the stationary distribution of \eqref{eq:fclt-twotime:y-sde}.
		
		Note that Assumption~\ref{assump:hurwitz} implies that $-\left(B_3 - \frac{\invdiffslow}{2} \mI \right)$ is a Hurwitz matrix.
		The property of the SDE in \eqref{eq:fclt-twotime:y-sde} implies that the invariant distribution is $\NM(0, \Sigma_{y})$ with $\Sigma_y$ satisfying the equation in \eqref{eq:fclt-twotime:y-var} in \cite[Section~6.5]{risken1996fokker}.
	\end{proof}

	\begin{proof}[Proof of Corollary~\ref{cor:clt}]
		From the definition of $\xcheck_{n}$, we have $\alpha_{n-1}^{-1/2} \xhat_{n} \weakconverge \NM(0, \Sigma_x)$.
		Assumption~\ref{assump:stepsize-twotime} ensures that $\alpha_{n-1} / \alpha_{n} \to 1$.
		Then Slutsky's theorem implies $\alpha_{n}^{-1/2} \xhat_{n} \weakconverge \NM(0, \Sigma_x)$.
		Similarly, we could prove $\beta_{n}^{-1/2} \yhat_{n} \weakconverge \NM(0, \Sigma_y)$.

		Moreover, since $\xx^\star = H(\yy^\star)$,  we have $\alpha_{n}^{-1/2} (\xx_{n} - \xx^\star) = \alpha_{n}^{-1/2} \xhat_{n} + \alpha_{n}^{-1/2} (H(\yy_n) - H(\yy^\star))$.
		Using the delta method, we have $\beta_{n}^{-1/2} (H(\yy_{n}) - H(\yy^\star)) \weakconverge \NM(0, H^\star \Sigma_y (H^\star)^\top ) $ with $H^\star = \nabla H(\yy^\star) \in \RB^{d_x \times d_y}$.
		Since $\beta_{n} / \alpha_{n} \to 0$,  applying Slutsky's theorem yields $\alpha_{n}^{-1/2} (H(\yy_n) - H(\yy^\star)) \weakconverge 0 $. Consequently, $\alpha_{n}^{-1/2} (\xx_{n} - \xx^\star) \weakconverge \NM(0, \Sigma_{x})$.
	\end{proof}